\documentclass[twoside]{article}
\usepackage{}
\usepackage[utf8]{inputenc}
\usepackage{amsfonts}
\usepackage{graphicx}
\usepackage[left=2.5cm,right=2.5cm,top=2cm,bottom=3cm]{geometry}
\usepackage{subfigure}
\usepackage{epsfig,graphics}
\usepackage{amssymb,amsmath,amsthm,bm}
\usepackage{url}
\usepackage{float}
\usepackage{multirow}
\usepackage{algorithm}
\usepackage{color}   
\newtheorem{theorem}{Theorem}[section]
\newtheorem{lemma}{Lemma}[section]

\newtheorem{definition}{Definition}[section]

\def\donchitre#1#2{\vskip 6.5cm\noindent
\parbox[t]{1in}{\special{eps:#1.eps x=6.5cm y=5.5cm}}
\hbox to 7cm{}\parbox[t]{0.0cm}{\special{eps:#2.eps x=6.5cm y=5.5cm}}}

\title{Stochastic Variance Reduced Gradient for affine rank minimization problem}
\author{
Ningning Han\thanks{School of Mathematical Sciences, Tiangong University, Tianjin, 300387, China
\textsf{email:} ningninghan@tiangong.edu.cn.},
Juan Nie\thanks{ Shenzhen Key Laboratory of Advanced Machine Learning and Applications, College of Mathematics and Statistics, Shenzhen University, Shenzhen, 518060, China. \textsf{email:} niejuan0522@163.com.},
Jian Lu\thanks{Corresponding author. Shenzhen Key Laboratory of Advanced Machine Learning and Applications, College of Mathematics and Statistics, Shenzhen University, Shenzhen, 518060, China.
\textsf{email:} jianlu@szu.edu.cn.},
Michael K. Ng\thanks{Department of Mathematics, the University of Hong Kong, Pokfulam, Hong Kong SAR.
\textsf{email:} mng@maths.hku.hk.}
}
\date{}
\begin{document}
\maketitle
\maketitle
 \begin{abstract}
We develop an efficient stochastic variance reduced gradient descent algorithm to solve the affine rank minimization problem consists of finding a matrix of minimum rank from linear measurements. The proposed algorithm as a stochastic gradient descent strategy enjoys a more favorable complexity than full gradients. It also reduces the variance of the stochastic gradient at each iteration and accelerate the rate of convergence. We prove that the proposed algorithm converges linearly in expectation to the solution under a restricted isometry condition. The numerical experiments show that the proposed algorithm has a clearly advantageous balance of efficiency, adaptivity, and accuracy compared with other state-of-the-art greedy algorithms. 

 \end{abstract}


\section{Introduction}
Affine rank minimization problem is a fundamental problem that arises in many practical applications of computer vision, machine learning and signal processing, such as collaborative filtering \cite{Rennie}-\cite{Rao}, image and video processing \cite{Ji}-\cite{Shin}, phaseless signal recovery \cite{Candes}-\cite{Eldar}, communication system \cite{Petropulu}-\cite{Qin}, multi-task learning \cite{Pong}-\cite{RZhang}, etc. Let $X^{\ast}=\{X^{\ast}_{i,j}\}\in\mathbb{R}^{n_{1}\times n_{2}}$ be the ground truth low-rank matrix, and we acquire information about $X^{\ast}$ through a linear mapping $\mathcal {A}:\mathbb{R}^{n_{1}\times n_{2}}\rightarrow\mathbb{R}^{m}$, i.e., $y=\mathcal {A}(X^{\ast})$ or $y_{\ell}=\mathcal {A}_{\ell}(X^{\ast})=\langle A_{\ell},X^{\ast}\rangle$, $\ell=1,\ldots,m$, where $A_{\ell}\in\mathbb{R}^{n_{1}\times n_{2}}$ denote the sensing matrices making up the linear mapping $\mathcal {A}(\cdot)$, then the low-rank matrix minimization problem can be formulated as follows:
\begin{equation}\label{rankformulation0}
\min \limits_{X} \text{rank}(X),\ \  \text{subj. to}\ \  y=\mathcal {A}(X).
\end{equation}

(\ref{rankformulation0}) are clearly combinatorial and computationally intractable. Various computationally efficient algorithms for solving (\ref{rankformulation0}) have been extensively studied. A large majority of algorithms are based on two strategies: convex or non-convex relaxations and greedy iterative algorithms. The renowned advance of relaxations is to replace the optimization problem with the rank function by nuclear norm, namely,
\begin{equation}\label{rankformulation}
\min \limits_{X} \|X\|_{\ast},\ \  \text{subj. to}\ \  y=\mathcal {A}(X).
\end{equation}
For a given $n\times n$ square matrix, Cand$\acute{e}$s et al prove that if the number $m$ of sampled entries satisfies $m\geq Cn^{1.2}r\log n$ for some positive numerical constant $C$, then with very high probability, most $n\times n$ matrices of rank $r$ can be exactly recovered by solving the convex optimization (\ref{rankformulation}) \cite{Recht}. Readers are referred to a series of articles focused on the theoretical analysis \cite{Tao}-\cite{Benjamin} and numerical algorithms \cite{Cai}-\cite{cLin} of the nuclear norm approach.

The singular values indicate clear geometric interpretations and should be regularize differently. As the nuclear norm penalizes each singular value equally, the nuclear norm may not be a good surrogate to the rank function. To get a more accurate and robust approximation to the rank function, a novel method called truncated nuclear norm regularization \cite{YHu,bZhang} is proposed, which only minimized the smallest $p$ singular values to recover the low-rank component. Note that all the existing nonconvex penalty functions are concave and their gradients are decreasing functions, iterative reweighted nuclear norms are proposed to solve low-rank matrix completion \cite{bZhang,Mohan,Fornasier}. Inspired by the paradigm of $\ell_{p}$ quasi-norm ($0<p<1$) in compressive sensing, some try to expand this concept to the traditional nuclear norm \cite{FPNie,YXie}, which can approximate the rank function better.

Alternating minimization \cite{Escalante} is also widely used for affine rank minimization problem. Among these algorithms, a symbolic work, known as the factorization $Z = UV'$, where $U\in\mathbb{R}^{n_{1}\times r}$ and $V\in\mathbb{R}^{n_{2}\times r}$, explicitly optimize on the manifold of rank $r$ matrices. The renowned advance of relaxations is to replace the optimization problem (\ref{rankformulation}) with the following non-convex problem
\begin{equation}\label{fenjie}
\begin{aligned}
\min\limits_{U,V}\frac{1}{2}\|y-\mathcal {A}(UV')\|_{F}^{2}.
\end{aligned}
\end{equation}
Two representatives alternating minimization schemes for solving model (\ref{fenjie}) are the power factorization algorithm \cite{Haldar} and the low-rank matrix fitting algorithm \cite{yinn}. In \cite{Tanner}, the authors propose an alternating steepest descent and a scaled variant scaled alternating steepest descent, where an exact line-search is incorporated to update the solutions of the model (\ref{fenjie}). Yao et al. \cite{Yao} propose a general nonconvex loss instead of $\ell_{1}$ loss to improve robustness of matrix factorization. For a nonconvex function $f(UV')$ w.r.t. $U$ and $V$, the bi-factored gradient descent (BFGD) algorithm, as an efficient first-order method is proposed to operate directly on the $U$, $V$ factors \cite{DYPark}. Li et al. \cite{XiaoLi} study the problem of recovering a low-rank matrix from a number of random linear measurements that are corrupted by outliers, where authors propose a nonsmooth nonconvex formulation of the problem and enforce the low-rank property of the solution by using a factored representation of the matrix variable. An simple iterative algorithm based on a Gauss-Newton is proposed to solve low rank matrix recovery, where a key property of Gauss-Newton Matrix Recovery is that it implicitly keeps the factor matrices approximately balanced throughout its iterations \cite{PZilber}. The authors in \cite{XJiang} formulate matrix completion as a feasibility problem and an alternating projection algorithm is devised to find a feasible point in the intersection of the low-rank constraint set and fidelity constraint set. Scaled gradient descent (ScaledGD) viewed as preconditioned or diagonally-scaled gradient descent has also been developed, where the preconditioners are adaptive and iteration-varying with a minimal computational overhead \cite{TTong}. In addition, Riemannian conjugated gradient method  minimizes the least-square distance on the sampling set over the Riemannian manifold of fixed-rank matrices and the algorithm is an adaptation of classical non-linear conjugate gradients \cite{Vandereycken}.

Greedy algorithms such as iterative hard thresholding (IHT) are another class of popular approaches, one advantage of greedy approaches is that they are considerably low computational complexity. A representative strategy called singular value thresholding, which produces a sequence of matrices, and at each step mainly performs a soft-thresholding operation on the singular values of matrix \cite{SVT}. Similarly, a fast singular value projection algorithm  \cite{PJain} has been proposed where the hard thresholding operator is employed to penalize singular values. It has been shown in \cite{PJain} that if the sensing operator $\mathcal {A}(\cdot)$ satisfies constrained restricted isometry property, then iterative hard thresholding with appropriate constant stepsize is guaranteed to recover any low rank matrix. Tanner et al. \cite{Tannern} introduce an efficient alternating projection algorithm, where the proposed algorithm uses an adaptive stepsize calculated to be exact for a restricted subspace. Furthermore, the authors develop a conjugate gradient iterative hard thresholding family of algorithms, which can balance the low per iteration complexity of simple hard thresholding algorithms with the fast asymptotic convergence rate of employing the conjugate gradient method \cite{Blanchard}. A family of Riemannian optimization algorithms for low rank matrix has also been introduced for low rank matrix recovery, which are first interpreted as iterative hard thresholding algorithms with subspace projections \cite{Kwei}. 

Recent technological advances in data collection and storage raise new challenges in large-scale signal processing problems that essentially involve optimization over particularly large-scale data. Stochastic gradient descent as effective and efficient optimization methods has been widely used for training machine learning models on massive datasets. Stochastic gradient descent (SGD) algorithms have been applied to solve low-rank matrix recovery \cite{Nguyen}, where these algorithms avoid computing the full gradient and possess favorable properties in solving large-scale problems especially when computing the full gradient is expensive or prohibitive. Note that stochastic gradient descent iterates with the inherent variance, which preserves slow convergence asymptotically.  To remedy this problem, stochastic variance reduced gradient (SVRG) \cite{SVRG} has been introduced as an explicit variance reduction strategy for stochastic gradient descent. 
The main aim of this paper is to exploit IHT and SVRG, and 
propose a new algorithm to solve affine rank minimization problem. The advantage of this algorithm is to reduce the variance of the stochastic gradient at each iteration and accelerate the rate of convergence. We prove the proposed algorithm converges linearly for affine rank minimization problem and conduct a series of numerical experiments to illustrate that the proposed algorithms have a clearly advantageous balance of efficiency, adaptivity and accuracy compared with other state-of-the-art algorithms.

%
 
 \begin{center}
     \begin{tabular}{lp{120mm}}
      \hline
      &{Algorithm~$1$~SVRG for affine rank minimization problem}\\ 
      \hline
      &{\bf Input}: $K$, $n$, $r$, $y$, $\mathcal {A}$, $\epsilon$,$\eta$\\
      &{\bf Output}: $\widehat{X}=\widetilde{X}_{k}$\\
      &{\bf Initialize}: $\widetilde{X}_{0}$\\
      &{\bf for}~~~~$k=0,1,\ldots,K-1$ {\bf do}\\
      &~~~~~~~~~$g_{k}=\frac{1}{m}\sum\limits_{\ell=1}^{m}\nabla f_{\ell}(\widetilde{X}_{k})$\\
      &~~~~~~~~~$X_{0}=\widetilde{X}_{k}$\\
      &~~~~~~~~~{\bf for}~~~~$t=0,\ldots,n-1$~~ {\bf do}\\
      &~~~~~~~~~~~~~~~~~~Randomly pick $i_{t}\in\{1,\ldots,m\}$\\
      &~~~~~~~~~~~~~~~~~~$W_{t}=X_{t}-\eta\left(\nabla f_{i_{t}}\left(X_{t}\right)-\nabla f_{i_{t}}\left(\widetilde{X}_{k}\right)+g_{k}\right)$\\
      &~~~~~~~~~~~~~~~~~~$X_{t+1}=\mathcal {H}_{r}(W_{t})$\\
      &~~~~~~~~~{\bf end for}\\
      &~~~~~~~~~~$\widetilde{X}_{k+1}=X_{n}$ \\
      &~~~~~~~~~~If  $\|y-\mathcal {A}(X_{k+1})\|_{2}^{2}\leq\epsilon$ or $\|\widetilde{X}_{k+1}-\widetilde{X}_{k}\|_{F}^{2}\leq\epsilon$, exit\\
      &{\bf end for}\\
     \hline
     \end{tabular}
\end{center}
\section{SVRG algorithm for affine rank minimization problem}

The cost function $F(x)$ can be defined by
\begin{equation*}\label{costFunction}
\begin{aligned}
F(X)=\frac{1}{m}\|y-\mathcal {A}(X)\|_{2}^{2}=\frac{1}{m}\sum\limits_{\ell=1}^{m}\left(y_{\ell}-\langle A_{\ell},X\rangle\right)^{2}=\frac{1}{m}\sum\limits_{\ell=1}^{m}f_{\ell}(X),
\end{aligned}
\end{equation*}
we perform the following minimization to recover $X^{\ast}$
\begin{equation}\label{l0formulation}
\min \limits_{X} F(X),\ \  \text{subj. to}\ \  \text{rank}(X)\leq r.
\end{equation}
A standard method for solving (\ref{l0formulation}) is gradient descent, which updates the iterations by
\begin{equation*}
\begin{aligned}
X_{t}=X_{t-1}-\eta_{t}\nabla F(X_{t-1})=X_{t-1}-\frac{\eta_{t}}{m}\sum\limits_{\ell=1}^{m}\nabla f_{\ell}(X_{t-1}).
\end{aligned}
\end{equation*}
Note that gradient descent strategy requires evaluation of $m$ derivatives, which is computationally expensive. A
popular modification is stochastic gradient descent, where we can choose a random training sample set $i_{t}$ of size $|i_{t}|$ from $\{1,2,\ldots,m\}$ and the variable is updated by
 \begin{equation*}
\begin{aligned}
X_{t}=X_{t-1}-\eta_{t}\nabla f_{i_{t}}(X_{t-1}),
\end{aligned}
\end{equation*}
where $f_{i_{t}}(X)=\frac{1}{|i_{t}|}\sum\limits_{\ell\in i_{t}}\left(y_{\ell}-\langle A_{\ell},X\rangle\right)^{2}$. Although the computational cost of stochastic gradient descent is smaller than full gradient descent strategy, it introduces variance due to random selection. In this paper, we employ  stochastic variance reduced gradient (SVRG) \cite{SVRG} to reduce the variance and accelerate convergence rate. 
\begin{figure}[H]
\centering
\begin{tabular}{c}
\includegraphics[width=7.5cm]{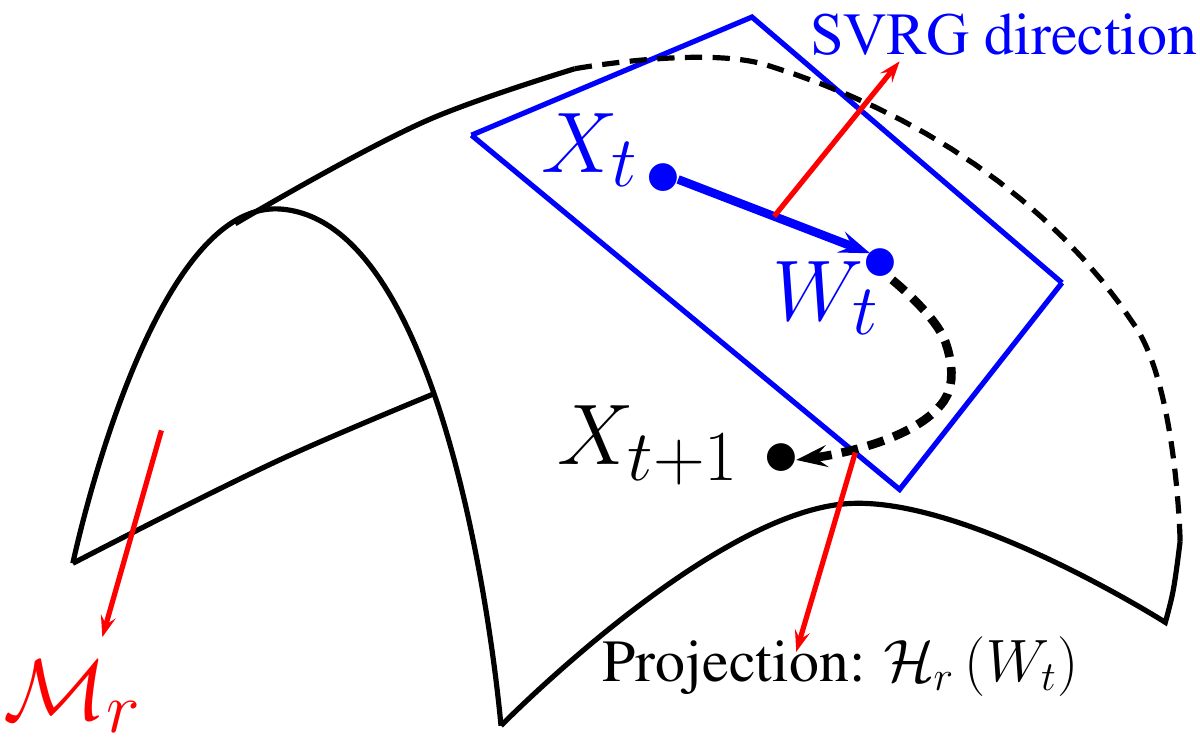}
\end{tabular}
\caption{A geometric description of SVRG algorithm for affine rank minimization problem.}
\label{figure1}
\end{figure}
The proposed stochastic variance reduced gradient for affine rank minimization problem (SVRG-ARM) is provided in Algorithm $1$. The outer loop computes a full gradient $g_{k}$, which is designed to reduce the variance caused by stochastic gradient descent. The inner loop first selects randomly an index set $i_{t}$ from the set $\{1,\ldots,m\}$ and then compute the stochastic variance reduced gradient associated with the selected index set. Note that $\mathbb{E}\left(\nabla f_{i_{t}}\left(X_{t}\right)\right)=g_{k}$ in the inner loop, we can force the gradient to be unbiased by letting gradient as $\nabla f_{i_{t}}\left(X_{t}\right)-\left(\nabla f_{i_{t}}\left(\widetilde{X}_{k}\right)-g_{k}\right)$ and then move the solution along the gradient direction to obtain solution $W_{t}$. The current solution $W_{t}$ needs to be projected onto the constraint space $\mathcal {M}_{r}$ via hard thresholding operator $\mathcal {H}_{r}(W_{t})$, where $\mathcal {M}_{r}=\{X\in\mathbb{R}^{n_{1} \times n_{2}}:\text{rank}(X)=r\}$. Figure \ref{figure1} shows a geometric description of SVRG-ARM.


\section{Linear Convergence analysis of SVRG-ARM}
In this section, we provide linear convergence analysis of the proposed SVRG-ARM algorithm. 
It should be pointed out that the linear convergence condition is not necessarily optimal at present times, which can be relaxed with perhaps plenty of rooms to improve. We first present the key preliminary results needed in the subsequent analyses.

\begin{definition}\label{erf}
(Restricted isometry property (RIP) \cite{BRecht}). Let $\mathcal {A}(\cdot)$: $\mathbb{R}^{n_{1}\times n_{2}}\rightarrow \mathbb{R}^{m}$ be a linear map of $n_{1}\times n_{2}$ matrices to vectors of length $m$. For every integer $1\leq r\leq$\text{min}$(n_{1},n_{2})$, the restricted isometry constant $\delta_{r}$ of $\mathcal {A}(\cdot)$ is defined as the smallest number such that
\begin{equation}\label{costFunction1}
\begin{aligned}
(1-\delta_{r})\|X\|_{F}^{2}\leq\frac{1}{m}\|\mathcal {A}(X)\|_{2}^{2},
\end{aligned}
\end{equation}
\begin{equation}\label{RIP2}
\begin{aligned}
\frac{1}{|i_{t}|}\|\mathcal {A}_{i_{t}}(X)\|_{2}^{2}\leq(1+\delta_{r})\|X\|_{F}^{2},~~i_{t}\in\{1,\ldots,m\},
\end{aligned}
\end{equation}
holds for all matrices $X$ of rank at most $r$.
\end{definition}

\begin{lemma}\label{lla}
For any two low-rank matrices $X$ and $Y$, let $\Gamma$ be a space spanned by $X$ and $Y$ and the rank of any matrix in $\Gamma$ is at most $s$, then
\begin{equation}\label{wocao}
\begin{aligned}
\langle X-Y,\nabla F(X)-\nabla F(Y)\rangle\geq2\left(1-\delta_{s}\right)\|X-Y\|_{F}^{2},
\end{aligned}
\end{equation}
and
\begin{equation}\label{rtyu}
\begin{aligned}
F(Y)\geq F(X)+\langle\nabla F(X),Y-X \rangle+\left(1-\delta_{s}\right)\|X-Y\|_{F}^{2}.
\end{aligned}
\end{equation}
\end{lemma}
\begin{proof}
It follows from the RIP that
\begin{equation*}
\begin{aligned}
\langle X-Y,\nabla F(X)-\nabla F(Y)\rangle&=\frac{2}{m}\left\langle X-Y,\sum\limits_{\ell=1}^{m}A_{\ell}\left\langle A_{\ell},X-Y\right\rangle\right\rangle\\
&=\frac{2}{m}\sum\limits_{\ell=1}^{m}\left\langle A_{\ell}, X-Y\right\rangle^{2}\\
&=\frac{2}{m}\|\mathcal {A}(X-Y)\|_{2}^{2}\\
&\geq2\left(1-\delta_{s}\right)\|X-Y\|_{F}^{2}\\
\end{aligned}
\end{equation*}
 The equivalent conditions $(iv)$ and $(iii)$ in {\em{Lemma 2}}  (\cite{XYZhou}) implies
 \begin{equation*}
\begin{aligned}
F(Y)\geq F(X)+\langle\nabla F(X),Y-X \rangle+\left(1-\delta_{s}\right)\|X-Y\|_{F}^{2}.
\end{aligned}
\end{equation*}
\end{proof}

\begin{lemma}\label{leee}
For any two low-rank matrices $X$ and $Y$, let $\Gamma$ be a space spanned by $X$ and $Y$ and the rank of any matrix in $\Gamma$ is at most $s$. Then, we have 
\begin{equation*}
\begin{aligned}
\|\mathcal {P}_{\Gamma}\left(\nabla f_{i_{t}}\left(X\right)-\nabla f_{i_{t}}\left(Y\right)\right)\|_{F} ^{2}\leq2\left(1+\delta_{s}\right)\left\langle X-Y,\nabla f_{i_{t}}(X)-\nabla f_{i_{t}}(Y) \right\rangle.
\end{aligned}
\end{equation*}
\end{lemma}
\begin{proof}
In view of $f_{i_{t}}(X)=\frac{1}{|i_{t}|}\sum\limits_{\ell\in i_{t}}\left(y_{\ell}-\langle A_{\ell},X\rangle\right)^{2}$ and $\nabla f_{i_{t}}(X)=\frac{2}{|i_{t}|}\sum\limits_{\ell\in i_{t}} A_{\ell}\left(\left\langle A_{\ell},X\right\rangle-y_{\ell}\right)$, we obtain
\begin{equation*}
\begin{aligned}
\frac{1}{|i_{t}|}\|\mathcal {A}_{i_{t}}\left(X-Y\right)\|_{2}^{2}&=\frac{1}{|i_{t}|}\sum\limits_{\ell\in i_{t}}\langle A_{\ell},X-Y\rangle^{2}\\
&=\frac{1}{|i_{t}|}\left\langle\sum\limits_{\ell\in i_{t}} A_{\ell}\left\langle A_{\ell},X-Y\right\rangle,X-Y\right\rangle\\
&=\frac{1}{2}\left\langle \nabla f_{i_{t}}\left(X\right)-\nabla f_{i_{t}}\left(Y\right),X-Y\right\rangle.\\
\end{aligned}
\end{equation*}
Together with with (\ref{RIP2}), we have
\begin{equation}\label{efer}
\begin{aligned}
\left\langle \nabla f_{i_{t}}\left(X\right)-\nabla f_{i_{t}}\left(Y\right),X-Y\right\rangle\leq2(1+\delta_{s})\|X-Y\|_{F}^{2}.
\end{aligned}
\end{equation}
Since $f_{i_{t}}(\cdot)$ is a convex function, the equivalent conditions $(3)$ and $(0)$ in {\em{Lemma 4}}  (\cite{XYZhou}) implies
\begin{equation}\label{eeee}
\begin{aligned}
\|\nabla f_{i_{t}}\left(Y\right)-\nabla f_{i_{t}}\left(X\right)\|_{F}\leq 2\left(1+\delta_{s}\right)\|X-Y\|_{F}.
\end{aligned}
\end{equation}
Define the function $h_{i_{t}}(Z)=f_{i_{t}}(Z)-\langle\nabla f_{i_{t}}(X),Z\rangle$ and follow from (\ref{eeee}),
\begin{equation*}
\begin{aligned}
\|\nabla h _{i_{t}}(Z_{1})-\nabla h _{i_{t}}(Z_{2})\|_{F}=\|\nabla f _{i_{t}}(Z_{1})-\nabla f _{i_{t}}(Z_{2})\|_{F}\leq2(1+\delta_{s})\|Z_{1}-Z_{2}\|_{F}
\end{aligned}
\end{equation*}
holds for $\forall$ $Z_{1}, Z_{2}\in\Gamma$. Using the equivalent conditions $(0)$ and $(2)$ in {\em{Lemma 4}} (\cite{XYZhou}) about the convex function $h_{i_{t}}(\cdot)$, we have
\begin{equation}\label{equation123}
\begin{aligned}
h_{i_{t}}\left(Z_{1}\right)-h_{i_{t}}\left(Z_{2}\right)-\left\langle \nabla h_{i_{t}}\left(Z_{2}\right),Z_{1}-Z_{2} \right\rangle\leq(1+\delta_{s})\|Z_{1}-Z_{2}\|_{F}^{2}.
\end{aligned}
\end{equation}
For $\forall Z\in\Gamma$, according to the definitions of $f _{i_{t}}(Z)$, $f _{i_{t}}(X)$, and $\nabla f _{i_{t}}(X)$, we obtain
\begin{equation}\label{equation321}
\begin{aligned}
h _{i_{t}}(Z)-h _{i_{t}}(X)&=f _{i_{t}}(Z)-f _{i_{t}}(X)-\langle\nabla f _{i_{t}}(X),Z-X\rangle\\
&=\frac{1}{|i_{t}|}\sum\limits_{\ell\in i_{t}}\langle A_{\ell}, Z-X\rangle^{2}\geq0.\\
\end{aligned}
\end{equation}
Define $Z=Y-\frac{1}{2\left(1+\delta_{s}\right)}\mathcal {P}_{\Gamma}\nabla h_{i_{t}}\left(Y\right)$. Applying (\ref{equation321}) and  (\ref{equation123}) gives
\begin{equation*}
\begin{aligned}
h _{i_{t}}(X)&\leq h _{i_{t}}(Z)=h _{i_{t}}\left(Y-\frac{1}{2\left(1+\delta_{s}\right)}\mathcal {P}_{\Gamma}\nabla h_{i_{t}}\left(Y\right)\right)\\
&\leq h _{i_{t}}(Y)+\left\langle\nabla h _{i_{t}}(Y),-\frac{1}{2\left(1+\delta_{s}\right)}\mathcal {P}_{\Gamma}\nabla h_{i_{t}}\left(Y\right)\right\rangle+\frac{1}{4\left(1+\delta_{s}\right)}\|\mathcal {P}_{\Gamma}\nabla h_{i_{t}}\left(Y\right)\|_{F}^{2}\\
&=h _{i_{t}}(Y)-\frac{1}{4\left(1+\delta_{s}\right)}\|\mathcal {P}_{\Gamma}\nabla h_{i_{t}}\left(Y\right)\|_{F}^{2}\\
\end{aligned}
\end{equation*}
Based on the definition of $h_{i_{t}}(\cdot)$ and the above inequality, we have 
\begin{equation*}
\begin{aligned}
\frac{1}{4\left(1+\delta_{s}\right)}\|\mathcal {P}_{\Gamma}\nabla h_{i_{t}}\left(Y\right)\|_{F}^{2}&=\frac{1}{4\left(1+\delta_{s}\right)}\|\mathcal {P}_{\Gamma}\left(\nabla f_{i_{t}}\left(Y\right)-\nabla f_{i_{t}}\left(X\right)\right)\|_{F}^{2}\\
&\leq h _{i_{t}}(Y)-h _{i_{t}}(X)=f _{i_{t}}(Y)-f _{i_{t}}(X)-\langle\nabla f_{i_{t}}(X),Y-X\rangle\\
\end{aligned}
\end{equation*}
Similarly, interchanging the role of $Y$ and $X$ leads to
\begin{equation*}
\begin{aligned}
\frac{1}{4\left(1+\delta_{s}\right)}\|\mathcal {P}_{\Gamma}\left(\nabla f_{i_{t}}\left(X\right)-\nabla f_{i_{t}}\left(Y\right)\right)\|_{F}^{2}&\leq f _{i_{t}}(X)-f _{i_{t}}(Y)-\langle\nabla f_{i_{t}}(Y),X-Y\rangle\\
\end{aligned}
\end{equation*}
Taking the summation, we derive 
\begin{equation*}
\begin{aligned}
\|\mathcal {P}_{\Gamma}\left(\nabla f_{i_{t}}\left(X\right)-\nabla f_{i_{t}}\left(Y\right)\right)\|_{F}^{2}&\leq 2\left(1+\delta_{s}\right)\left\langle X-Y,\nabla f_{i_{t}}(X)-\nabla f_{i_{t}}(Y) \right\rangle\\
\end{aligned}
\end{equation*}
\end{proof}

\begin{lemma}\label{eeeee}
For any two low-rank matrices $X$ and $Y$, let $\Gamma$ be a space spanned by $X$ and $Y$ and the rank of any matrix in $\Gamma$ is at most $s$ and $\eta\leq\frac{1}{1+\delta_{s}}$. Then, we have
\begin{equation*}
\begin{aligned}
\|X-Y-\eta \mathcal {P}_{\Gamma}\left(\nabla F(X)-\nabla F(Y)\right)\|_{F}\leq\sqrt{1-2\left(1-\delta_{s}\right)\left(2\eta-2\eta^{2}\left(1+\delta_{s}\right)\right)}\|X-Y\|_{F}.
\end{aligned}
\end{equation*}
\end{lemma}
\begin{proof}
\begin{equation*}
\begin{aligned}
&\|X-Y-\eta \mathcal {P}_{\Gamma}\left(\nabla F(X)-\nabla F(Y)\right)\|_{F}^{2}\\
&=\|X-Y\|_{F}^{2}+\eta^{2}\|\mathcal {P}_{\Gamma}\left(\nabla F(X)-\nabla F(Y)\right)\|_{F}^{2}-2\eta\langle X-Y,\mathcal {P}_{\Gamma}\left(\nabla F(X)-\nabla F(Y)\right)\rangle\\
&\leq\|X-Y\|_{F}^{2}+2\eta^{2}\left(1+\delta_{s}\right)\left\langle X-Y,\nabla F(X)-\nabla F(Y) \right\rangle-2\eta\langle X-Y,\mathcal {P}_{\Gamma}\left(\nabla F(X)-\nabla F(Y)\right)\rangle\\
&=\|X-Y\|_{F}^{2}-\left(2\eta-2\eta^{2}\left(1+\delta_{s}\right)\right)\left\langle X-Y,\nabla F(X)-\nabla F(Y) \right\rangle\\
&\leq\|X-Y\|_{F}^{2}-2\left(1-\delta_{s}\right)\left(2\eta-2\eta^{2}\left(1+\delta_{s}\right)\right)\|X-Y\|_{F}^{2}\\
&=\left[1-2\left(1-\delta_{s}\right)\left(2\eta-2\eta^{2}\left(1+\delta_{s}\right)\right)\right]\|X-Y\|_{F}^{2},\\
\end{aligned}
\end{equation*}
where the first inequality follows from {\em Lemma \ref{leee}} with taking $i_{t}=\{1,\ldots,m\}$ and the last inequality follows from {\em Lemma \ref{lla}}.
\end{proof}

\begin{lemma}\label{rtr}
For any two low-rank matrices $X$ and $Y$, let $\Gamma$ be a space spanned by $X$ and $Y$ and the rank of any matrix in $\Gamma$ is at most $s$. Denote $i_{t}$ be the index randomly selected from $\{1,\ldots,m\}$ and $\eta\leq\frac{1}{1+\delta_{s}}$, then we have
\begin{equation*}
\begin{aligned}
\mathbb{E}_{i_{t}}\|X-Y-\eta\mathcal {P}_{\Gamma}\left(\nabla f_{i_{t}}\left(X\right)-\nabla f_{i_{t}}\left(Y\right)\right)\|_{F}\leq\sqrt{1-2\left(1-\delta_{s}\right)\left(2\eta-2\eta^{2}\left(1+\delta_{s}\right)\right)}\|X-Y\|_{F} \\
\end{aligned}
\end{equation*}
\end{lemma}
\begin{proof}
\begin{equation*}
\begin{aligned}
&\mathbb{E}_{i_{t}}\|X-Y-\eta\mathcal {P}_{\Gamma}\left(\nabla f_{i_{t}}\left(X\right)-\nabla f_{i_{t}}\left(Y\right)\right)\|_{F}^{2} \\
&=\|X-Y\|_{F}^{2}+\eta^{2}\mathbb{E}_{i_{t}}\|\mathcal {P}_{\Gamma}\left(\nabla f_{i_{t}}\left(X\right)-\nabla f_{i_{t}}\left(Y\right)\right)\|_{F}^{2}-2\eta\mathbb{E}_{i_{t}}\langle X-Y,\mathcal {P}_{\Gamma}\left(\nabla f_{i_{t}}\left(X\right)-\nabla f_{i_{t}}\left(Y\right)\right)\rangle\\
&\leq\|X-Y\|_{F}^{2}+ 2\eta^{2}\left(1+\delta_{s}\right)\mathbb{E}_{i_{t}}\left\langle X-Y,\nabla f_{i_{t}}(X)-\nabla f_{i_{t}}(Y) \right\rangle-2\eta\mathbb{E}_{i_{t}}\langle X-Y,\nabla f_{i_{t}}\left(X\right)-\nabla f_{i_{t}}\left(Y\right)\rangle\\
&=\|X-Y\|_{F}^{2}-\left(2\eta-2\eta^{2}\left(1+\delta_{s}\right)\right)\mathbb{E}_{i_{t}}\left\langle X-Y,\nabla f_{i_{t}}(X)-\nabla f_{i_{t}}(Y) \right\rangle\\
&=\|X-Y\|_{F}^{2}-\left(2\eta-2\eta^{2}\left(1+\delta_{s}\right)\right)\left\langle X-Y,\nabla F(X)-\nabla F(Y) \right\rangle\\
&\leq\|X-Y\|_{F}^{2}-2\left(1-\delta_{s}\right)\left(2\eta-2\eta^{2}\left(1+\delta_{s}\right)\right)\|X-Y\|_{F}^{2}\\
\end{aligned}
\end{equation*}
where the first inequality is based on {\em Lemma \ref{leee}}, the last equality follows from the fact that $\nabla f_{i_{t}}(X)$ is an unbiased estimation to $\nabla F$, i.e., $\mathbb{E}[\nabla f_{i_{t}}(x_{t})|x_{t}]=\nabla F(x_{t})$ and the last inequality follows from {\em Lemma \ref{lla}}. The desired result follows by applying Jensen inequality $(\mathbb{E}Z)^{2}\leq\mathbb{E}(Z)^{2}$. 
\end{proof}

\begin{theorem}
Assume that $X^{\ast}$ is the optimal solution to (\ref{l0formulation}), the linear mapping $\mathcal {A}$ satisfies RIP defined in {\em Definition \ref{erf}} with $\delta_{3r}\leq\frac{1}{71}$, and the step size satisfies
\begin{equation*}
\begin{aligned}
\frac{6-6\delta_{3r}-\sqrt{71\delta_{3r}^{2}-72\delta_{3r}+1}}{12-12\delta_{3r}^{2}}<\eta<\frac{6-6\delta_{3r}+\sqrt{71\delta_{3r}^{2}-72\delta_{3r}+1}}{12-12\delta_{3r}^{2}}
\end{aligned}
\end{equation*}
then SVRG-ARM converges linearly in expectation:
\begin{equation*}
\begin{aligned}
\mathbb{E}_{i_{t}}\|\widetilde{X}_{k}-X^{\ast}\|_{F}\leq\kappa_{3r}^{k}\|\widetilde{X}_{0}-X^{\ast}\|_{F}\\
\end{aligned}
\end{equation*}
where $\rho_{3r}=2\sqrt{1-2\left(1-\delta_{3r}\right)\left(2\eta-2\eta^{2}\left(1+\delta_{3r}\right)\right)}$ and $\kappa_{3r}=\frac{-3\rho_{3r}^{n+1}+\rho_{3r}^{n}+2\rho_{3r}}{1-\rho_{3r}}<1.$
\end{theorem}

\begin{proof}
Note that Eckart–Young theorem guarantees that $X_{t+1}$ is the rank $r$ matrix nearest to $W_{t}$ in the Frobenius norm, we have $\|X_{t+1}-W_{t}\|_{F}^{2}\leq\|X^{\ast}-W_{t}\|_{F}^{2}$. It follows that
\begin{equation*}
\begin{aligned}
\|X_{t+1}-X^{\ast}\|_{F}^{2}&=\|X_{t+1}-X^{\ast}+X^{\ast}-W_{t}\|_{F}^{2}-\|X^{\ast}-W_{t}\|_{F}^{2}-2\langle X_{t+1}-X^{\ast},X^{\ast}-W_{t}\rangle\\
&=\|X_{t+1}-W_{t}\|_{F}^{2}-\|X^{\ast}-W_{t}\|_{F}^{2}-2\langle X_{t+1}-X^{\ast},X^{\ast}-W_{t}\rangle\\
&\leq2\langle X_{t+1}-X^{\ast},W_{t}-X^{\ast}\rangle\\
&=2\left\langle X_{t+1}-X^{\ast},X_{t}-\eta\left(\nabla f_{i_{t}}\left(X_{t}\right)-\nabla f_{i_{t}}\left(\widetilde{X}_{k}\right)+g_{k}\right)-X^{\ast}\right\rangle\\
&=2\left\langle  X_{t+1}-X^{\ast},X_{t}-X^{\ast}-\eta\left(\nabla f_{i_{t}}\left(X_{t}\right)-\nabla f_{i_{t}}\left(X^{\ast}\right)\right)\right\rangle\\
&-2\left\langle  X_{t+1}-X^{\ast},\widetilde{X}_{k}-X^{\ast}-\eta\left(\nabla f_{i_{t}}\left(\widetilde{X}_{k}\right)-\nabla f_{i_{t}}\left(X^{\ast}\right)\right)\right\rangle\\
&+2\left\langle  X_{t+1}-X^{\ast},\widetilde{X}_{k}-X^{\ast}-\eta\left(\nabla F\left(\widetilde{X}_{k}\right)-\nabla F\left(X^{\ast}\right)\right)\right\rangle\\
\end{aligned}
\end{equation*}
where the fourth equality follows from $g_{k}=\nabla F\left(\widetilde{X}_{k}\right)$ and $\nabla F\left(X^{\ast}\right)=0$.
Denote $\Omega_{t}$ as the subspace spanned by $X_{t+1}$, $X_{t}$ and $X^{\ast}$ and $\Omega_{t}' $ as the subspace spanned by $X_{t+1}$, $\widetilde{X}_{k}$ and $X^{\ast}$. Define $\mathcal {P}_{\Omega_{t}}:\mathbb{R}^{n_{1}\times n_{2}}\rightarrow \Omega_{t}$ as the orthogonal projection onto $\Omega_{t}$. Obviously, $\mathcal {P}_{\Omega_{t}}(X_{t+1})=X_{t+1}$, $\mathcal {P}_{\Omega_{t}}(X_{t})=X_{t}$, $\mathcal {P}_{\Omega_{t}}(X^{\ast})=X^{\ast}$, $\mathcal {P}_{\Omega_{t}'}(X_{t+1})=X_{t+1}$, $\mathcal {P}_{\Omega_{t}' }(\widetilde{X}_{k})=\widetilde{X}_{k}$ and $\mathcal {P}_{\Omega_{t}'}(X^{\ast})=X^{\ast}$. The rank of any matrix in $\Omega_{t}$ and $\Omega_{t}' $ is at most $3r$. Consequently, we obtain
\begin{equation*}
\begin{aligned}
\|X_{t+1}-X^{\ast}\|_{F}^{2}
&\leq2\left\langle  X_{t+1}-X^{\ast},X_{t}-X^{\ast}-\eta\mathcal {P}_{\Omega_{t}}\left(\nabla f_{i_{t}}\left(X_{t}\right)-\nabla f_{i_{t}}\left(X^{\ast}\right)\right)\right\rangle\\
&-2\left\langle  X_{t+1}-X^{\ast},\widetilde{X}_{k}-X^{\ast}-\eta\mathcal {P}_{\Omega_{t}'}\left(\nabla f_{i_{t}}\left(\widetilde{X}_{k}\right)-\nabla f_{i_{t}}\left(X^{\ast}\right)\right)\right\rangle\\
&+2\left\langle  X_{t+1}-X^{\ast},\widetilde{X}_{k}-X^{\ast}-\eta\mathcal {P}_{\Omega_{t}'}\left(\nabla F\left(\widetilde{X}_{k}\right)-\nabla F\left(X^{\ast}\right)\right)\right\rangle\\
&\leq2\|X_{t+1}-X^{\ast}\|_{F}(\|X_{t}-X^{\ast}-\eta{P}_{\Omega_{t}}\left(\nabla f_{i_{t}}\left(X_{t}\right)-\nabla f_{i_{t}}\left(X^{\ast}\right)\right)\|_{F}\\
&+\|\widetilde{X}_{k}-X^{\ast}-\eta{P}_{\Omega_{t}'}\left(\nabla f_{i_{t}}\left(\widetilde{X}_{k}\right)-\nabla f_{i_{t}}\left(X^{\ast}\right)\right)\|_{F}\\
&+\|\widetilde{X}_{k}-X^{\ast}-\eta{P}_{\Omega_{t}'}\left(\nabla F\left(\widetilde{X}_{k}\right)-\nabla F\left(X^{\ast}\right)\right)\|_{F}).\\ 
\end{aligned}
\end{equation*}
Canceling $\|X_{t+1}-X^{\ast}\|_{F}$ in the above inequality gives the inequality
\begin{equation}\label{equation11}
\begin{aligned}
\|X_{t+1}-X^{\ast}\|_{F}
&\leq2(\|X_{t}-X^{\ast}-\eta{P}_{\Omega_{t}}\left(\nabla f_{i_{t}}\left(X_{t}\right)-\nabla f_{i_{t}}\left(X^{\ast}\right)\right)\|_{F}\\
&+\|\widetilde{X}_{k}-X^{\ast}-\eta{P}_{\Omega_{t}'}\left(\nabla f_{i_{t}}\left(\widetilde{X}_{k}\right)-\nabla f_{i_{t}}\left(X^{\ast}\right)\right)\|_{F}\\
&+\|\widetilde{X}_{k}-X^{\ast}-\eta{P}_{\Omega_{t}'}\left(\nabla F\left(\widetilde{X}_{k}\right)-\nabla F\left(X^{\ast}\right)\right)\|_{F}),\\
\end{aligned}
\end{equation}
As defined $\eta<\frac{6-6\delta_{3r}+\sqrt{71\delta_{3r}^{2}-72\delta_{3r}+1}}{12-12\delta_{3r}^{2}}<\frac{1}{1+\delta_{3r}}$, note that $i_{t}$ determines the solution $X_{t+1}$, taking the expectation on both sides of (\ref{equation11}) yields,
\begin{equation*}
\begin{aligned}
\mathbb{E}_{i_{t}}\|X_{t+1}-X^{\ast}\|_{F}
&\leq2(\mathbb{E}_{i_{t}}\|X_{t}-X^{\ast}-\eta{P}_{\Omega_{t}}\left(\nabla f_{i_{t}}\left(X_{t}\right)-\nabla f_{i_{t}}\left(X^{\ast}\right)\right)\|_{F}\\
&+\mathbb{E}_{i_{t}}\|\widetilde{X}_{k}-X^{\ast}-\eta{P}_{\Omega_{t}'}\left(\nabla f_{i_{t}}\left(\widetilde{X}_{k}\right)-\nabla f_{i_{t}}\left(X^{\ast}\right)\right)\|_{F}\\
&+\|\widetilde{X}_{k}-X^{\ast}-\eta{P}_{\Omega_{t}'}\left(\nabla F\left(\widetilde{X}_{k}\right)-\nabla F\left(X^{\ast}\right)\right)\|_{F})\\
&\leq2(\sqrt{1-2\left(1-\delta_{3r}\right)\left(2\eta-2\eta^{2}\left(1+\delta_{3r}\right)\right)}\|X_{t}-X^{\ast}\|_{F} \\
&+\sqrt{1-2\left(1-\delta_{3r}\right)\left(2\eta-2\eta^{2}\left(1+\delta_{3r}\right)\right)}\|\widetilde{X}_{k}-X^{\ast}\|_{F} \\
&+\sqrt{1-2\left(1-\delta_{3r}\right)\left(2\eta-2\eta^{2}\left(1+\delta_{3r}\right)\right)}\|\widetilde{X}_{k}-X^{\ast}\|_{F})\\
&=2\sqrt{1-2\left(1-\delta_{3r}\right)\left(2\eta-2\eta^{2}\left(1+\delta_{3r}\right)\right)}\|X_{t}-X^{\ast}\|_{F}\\
&+4\sqrt{1-2\left(1-\delta_{3r}\right)\left(2\eta-2\eta^{2}\left(1+\delta_{3r}\right)\right)}\|\widetilde{X}_{k}-X^{\ast}\|_{F} \\
\end{aligned}
\end{equation*}
where the second inequality follows from {\em Lemma \ref{rtr}} and {\em Lemma \ref{eeeee}}.
By recursively applying the above inequality over $t$, and noting that $\widetilde{X}_{k}=X_{0}$ and $\widetilde{X}_{k+1}=X_{n}$, we can obtain
\begin{equation*}
\begin{aligned}
&\mathbb{E}_{i_{t}}\|\widetilde{X}_{k+1}-X^{\ast}\|_{F}
=\mathbb{E}_{i_{t}}\|X_{n}-X^{\ast}\|_{F}\leq\left(\rho_{3r}^{n}+2\rho_{3r}^{n}+\cdots+2\rho_{3r}\right)\|\widetilde{X}_{k}-X^{\ast}\|_{F}\\
&\leq \frac{-3\rho_{3r}^{n+1}+\rho_{3r}^{n}+2\rho_{3r}}{1-\rho_{3r}}\|\widetilde{X}_{k}-X^{\ast}\|_{F},\\
\end{aligned}
\end{equation*}
where $\rho_{3r}=2\sqrt{1-2\left(1-\delta_{3r}\right)\left(2\eta-2\eta^{2}\left(1+\delta_{3r}\right)\right)}$.
 Since $\delta_{3r}<\frac{1}{71}$ and $\frac{6-6\delta_{3r}-\sqrt{71\delta_{3r}^{2}-72\delta_{3r}+1}}{12-12\delta_{3r}^{2}}<\eta<\frac{6-6\delta_{3r}+\sqrt{71\delta_{3r}^{2}-72\delta_{3r}+1}}{12-12\delta_{3r}^{2}}$, we have $\frac{-3\rho_{3r}^{n+1}+\rho_{3r}^{n}+2\rho_{3r}}{1-\rho_{3r}}<1$. The linear convergence of SVRG-ARM algorithm for affine rank minimization problem follows immediately.
\end{proof}

\section{Complexity analysis}
This section contains the result about complexity analysis of SVRG-ARM. We first present the key lemmas needed in the subsequent analyses of the number of iterations for obtaining accuracy of $\epsilon$.
\begin{lemma}\label{lemma41}
Let $\Gamma$ be a space spanned by $X$ and $X^{\ast}$, and the rank of any matrix in $\Gamma$ is at most $s$. Then we have 
\begin{equation*}
\begin{aligned}
\mathbb{E}\|\mathcal {P}_{\Gamma}\left(\nabla f_{i_{t}}\left(X\right)-\nabla f_{i_{t}}\left(X^{\ast}\right)\right)\|_{F}^{2}&\leq4\left(1+\delta_{s}\right)\left(F\left(X\right)-F\left(X^{\ast}\right)\right).\\
\end{aligned}
\end{equation*}
\end{lemma}
\begin{proof}
For any $i_{t}\in\{1,2,\ldots,m\}$, and $X\in \mathbb{R}^{n_{1}\times n_{2}}$, we define
\begin{equation}\label{eqerr}
\begin{aligned}
\varphi_{i_{t}}(X)=f_{i_{t}}(X)-f_{i_{t}}(X^{\ast})-\langle\nabla f_{i_{t}}(X^{\ast}),X-X^{\ast}\rangle
\end{aligned}
\end{equation}
Then, we can get a similar inequality as in (\ref{eeee})
\begin{equation}\label{erwq}
\begin{aligned}
\|\nabla\varphi_{i_{t}}(X)-\nabla\varphi_{i_{t}}(Y)\|_{F}^{2}=\|\nabla f_{i_{t}}(X)-\nabla f_{i_{t}}(Y)\|_{F}^{2}\leq2(1+\delta_{s})\|X-Y\|_{F}^{2}.
\end{aligned}
\end{equation}
Since $\nabla\varphi_{i_{t}}(X^{\ast})=0$, we have $\varphi_{i_{t}}(X^{\ast})=\min\limits_{X}\varphi_{i_{t}}(X)$. Together with (\ref{erwq}), and the equivalent conditions $(0)$ and $(2)$ in {\em{Lemma 4}}  (\cite{XYZhou}), it results in 
\begin{equation*}
\begin{aligned}
0=\varphi_{i_{t}}(X^{\ast})&\leq\varphi_{i_{t}}\left(X-\mathcal {P}_{\Gamma}\left(\frac{1}{2\left(1+\delta_{s}\right)}\nabla\varphi_{i_{t}}\left(X\right)\right)\right)\\
&\leq\varphi_{i_{t}}\left(X\right)-\left\langle\nabla\varphi_{i_{t}}\left(X\right),\mathcal {P}_{\Gamma}\left(\frac{1}{2\left(1+\delta_{s}\right)}\nabla\varphi_{i_{t}}\left(X\right)\right)\right\rangle+\left(1+\delta_{s}\right)\|\frac{1}{2\left(1+\delta_{s}\right)}\mathcal {P}_{\Gamma}\left(\nabla\varphi_{i_{t}}\left(X\right)\right)\|_{F}^{2}\\
&=\varphi_{i_{t}}\left(X\right)-\frac{1}{4\left(1+\delta_{s}\right)}\|\mathcal {P}_{\Gamma}\left(\nabla\varphi_{i}\left(X\right)\right)\|_{F}^{2}.\\
\end{aligned}
\end{equation*}
From the definition (\ref{eqerr}), we have 
\begin{equation*}
\begin{aligned}
\|\mathcal {P}_{\Gamma}\left(\nabla f_{i_{t}}\left(X\right)-\nabla f_{i_{t}}\left(X^{\ast}\right)\right)\|_{F}^{2}\leq4\left(1+\delta_{s}\right)\left(f_{i_{t}}\left(X\right)-f_{i_{t}}\left(X^{\ast}\right)-\left\langle\nabla f_{i_{t}}\left(X^{\ast}\right),X-X^{\ast}\right\rangle\right).
\end{aligned}
\end{equation*}
Taking expectation with respect to $i_{t}$, we get 
\begin{equation*}
\begin{aligned}
\mathbb{E}\|\mathcal {P}_{\Gamma}\left(\nabla f_{i_{t}}\left(X\right)-\nabla f_{i_{t}}\left(X^{\ast}\right)\right)\|_{F}^{2}&\leq4\left(1+\delta_{s}\right)\left(F\left(X\right)-F\left(X^{\ast}\right)-\left\langle\nabla F\left(X^{\ast}\right),X-X^{\ast}\right\rangle\right)\\
&=4\left(1+\delta_{s}\right)\left(F\left(X\right)-F\left(X^{\ast}\right)\right),\\
\end{aligned}
\end{equation*}
where the last equality follows from the fact $\nabla F\left(X^{\ast}\right)=0$.
\end{proof}

\begin{lemma}\label{lemma42}
Let $X^{\ast}$ is the optimal solution to (\ref{l0formulation}). Given a low-rank matrix $X$, where rank$(X)=r$, let $\Lambda$ be a space spanned by $X$ and $X^{\ast}$, and the rank of any matrix in $\Lambda$ is at most $\tau$, then we have
\begin{equation*}
\begin{aligned}
\|\mathcal {P}_{\Lambda}\left(\nabla F\left(X\right)\right)\|_{F}^{2}\geq\frac{2\left(1-\delta_{\tau}\right)}{1+\delta_{\tau}}\left(F\left(X\right)-F\left(X^{\ast}\right)\right)\\
\end{aligned}
\end{equation*} 
\end{lemma}
\begin{proof}
We first note that
\begin{equation*}
\begin{aligned}
\left\langle X-X^{\ast},\nabla F\left(X\right)-\nabla F\left(X^{\ast}\right)\right\rangle=\left\langle X-X^{\ast},\mathcal {P}_{\Lambda}\left(\nabla F\left(X\right)-\nabla F\left(X^{\ast}\right)\right)\right\rangle\leq\|X-X^{\ast}\|_{F}\|\mathcal {P}_{\Lambda}\left(\nabla F\left(X\right)\right)\|_{F}\\
\end{aligned}
\end{equation*} 
Together with (\ref{wocao}), we have
\begin{equation}\label{edcv}
\begin{aligned}
\|\mathcal {P}_{\Lambda}\left(\nabla F\left(X\right)\right)\|_{F}^{2}\geq4\left(1-\delta_{\tau}\right)^{2}\|X-X^{\ast}\|_{F}^{2}\\
\end{aligned}
\end{equation} 
Let $i_{t}=\{1,\ldots,m\}$ in (\ref{efer}), we have
\begin{equation*}
\begin{aligned}
\left\langle \nabla F\left(X\right)-\nabla F\left(X^{\ast}\right),X-X^{\ast}\right\rangle\leq2(1+\delta_{\tau})\|X-X^{\ast}\|_{F}^{2}.
\end{aligned}
\end{equation*}
The equivalent conditions $(3)$ and $(2)$ in {\em{Lemma 4}}  (\cite{XYZhou}) implies
\begin{equation}\label{ewqee}
\begin{aligned}
F(X)\leq F(X^{\ast})+\left\langle\nabla F(X^{\ast}),X-X^{\ast}\right\rangle+\left(1+\delta_{\tau}\right)\|X-X^{\ast}\|_{F}^{2}= F(X^{\ast})+\left(1+\delta_{\tau}\right)\|X-X^{\ast}\|_{F}^{2}.
\end{aligned}
\end{equation}
Combining (\ref{edcv}) and (\ref{ewqee}) yields the desired result
\begin{equation*}
\begin{aligned}
\|\mathcal {P}_{\Lambda}\left(\nabla F\left(X\right)\right)\|_{F}^{2}\geq\frac{4\left(1-\delta_{\tau}\right)^{2}}{1+\delta_{\tau}}\left(F\left(X\right)-F\left(X^{\ast}\right)\right)\\
\end{aligned}
\end{equation*} 
\end{proof}

\begin{lemma}\label{reop}
Denote $\Omega_{t}$ as the subspace spanned by $\widetilde{X}_{k}$, $X_{t}$ and $X^{\ast}$, and the rank of any matrix in $\Omega_{t}$ is at most $3r$. Let $V_{t}=\nabla f_{i_{t}}\left(X_{t}\right)-\nabla f_{i_{t}}\left(\widetilde{X}_{k}\right)+\nabla F\left(\widetilde{X}_{k}\right)$ as defined in {\em Algorithm $1$}, then we have 
\begin{equation*}
\begin{aligned}
\mathbb{E}_{i_{t}}\|\mathcal {P}_{\Omega_{t}}\left(V_{t}\right)\|_{F}^{2}\leq8\left(1+\delta_{3r}\right)\left(F\left(X_{t}\right)-F\left(X^{\ast}\right)\right)+\frac{32\delta_{3r}}{1+\delta_{3r}}\left(F\left(\widetilde{X}_{k}\right)-F\left(X^{\ast}\right)\right).
\end{aligned}
\end{equation*}
\end{lemma}

\begin{proof}
By the definition of $\mathcal {P}_{\Omega_{t}}\left(V_{t}\right)$, we have 
\begin{equation*}
\begin{aligned}
\mathbb{E}_{i_{t}}\|\mathcal {P}_{\Omega_{t}}\left(V_{t}\right)\|_{F}^{2}&=\mathbb{E}_{i_{t}}\|\mathcal {P}_{\Omega_{t}}\left(\nabla f_{i_{t}}\left(X_{t}\right)-\nabla f_{i_{t}}\left(\widetilde{X}_{k}\right)+\nabla F\left(\widetilde{X}_{k}\right)\right)\|_{F}^{2}\\
&=\mathbb{E}_{i_{t}}\|\mathcal {P}_{\Omega_{t}}\left(\left(\nabla f_{i_{t}}\left(X_{t}\right)-\nabla f_{i_{t}}\left(X^{\ast}\right)\right)-\left(\nabla f_{i_{t}}\left(\widetilde{X}_{k}\right)-\nabla f_{i_{t}}\left(X^{\ast}\right)\right)+\nabla F\left(\widetilde{X}_{k}\right)\right)\|_{F}^{2}\\
&\leq2\mathbb{E}_{i_{t}}\|\mathcal {P}_{\Omega_{t}}\left(\left(\nabla f_{i_{t}}\left(X_{t}\right)-\nabla f_{i_{t}}\left(X^{\ast}\right)\right)\right)\|_{F}^{2}+2\mathbb{E}_{i_{t}}\|\mathcal {P}_{\Omega_{t}}\left(\left(\nabla f_{i_{t}}\left(\widetilde{X}_{k}\right)-\nabla f_{i_{t}}\left(X^{\ast}\right)\right)-\nabla F\left(\widetilde{X}_{k}\right)\right)\|_{F}^{2}\\
&=2\mathbb{E}_{i_{t}}\|\mathcal {P}_{\Omega_{t}}\left(\left(\nabla f_{i_{t}}\left(X_{t}\right)-\nabla f_{i_{t}}\left(X^{\ast}\right)\right)\right)\|_{F}^{2}+2\mathbb{E}_{i_{t}}\|\mathcal {P}_{\Omega_{t}}\left(\left(\nabla f_{i_{t}}\left(\widetilde{X}_{k}\right)-\nabla f_{i_{t}}\left(X^{\ast}\right)\right)\right)\|_{F}^{2}\\
&+2\mathbb{E}_{i_{t}}\|\mathcal {P}_{\Omega_{t}}\left(\nabla F\left(\widetilde{X}_{k}\right)\right)\|_{F}^{2}-4\mathbb{E}_{i_{t}}\langle\mathcal {P}_{\Omega_{t}}\left(\left(\nabla f_{i_{t}}\left(\widetilde{X}_{k}\right)-\nabla f_{i_{t}}\left(X^{\ast}\right)\right)\right),\mathcal {P}_{\Omega_{t}}\left(\nabla F\left(\widetilde{X}_{k}\right)\right)\rangle\\
&=2\mathbb{E}_{i_{t}}\|\mathcal {P}_{\Omega_{t}}\left(\left(\nabla f_{i_{t}}\left(X_{t}\right)-\nabla f_{i_{t}}\left(X^{\ast}\right)\right)\right)\|_{F}^{2}+2\mathbb{E}_{i_{t}}\|\mathcal {P}_{\Omega_{t}}\left(\left(\nabla f_{i_{t}}\left(\widetilde{X}_{k}\right)-\nabla f_{i_{t}}\left(X^{\ast}\right)\right)\right)\|_{F}^{2}\\
&-2\|\mathcal {P}_{\Omega_{t}}\left(\nabla F\left(\widetilde{X}_{k}\right)\right)\|_{F}^{2}\\
&\leq8\left(1+\delta_{3r}\right)\left(F\left(X_{t}\right)-F\left(X^{\ast}\right)\right)+8\left(1+\delta_{3r}\right)\left(F\left(\widetilde{X}_{k}\right)-F\left(X^{\ast}\right)\right)-\frac{8\left(1-\delta_{3r}\right)^{2}}{1+\delta_{3r}}\left(F\left(\widetilde{X}_{k}\right)-F\left(X^{\ast}\right)\right)\\
&=8\left(1+\delta_{3r}\right)\left(F\left(X_{t}\right)-F\left(X^{\ast}\right)\right)+\frac{32\delta_{3r}}{1+\delta_{3r}}\left(F\left(\widetilde{X}_{k}\right)-F\left(X^{\ast}\right)\right).\\
\end{aligned}
\end{equation*}
where the last inequality is due to {\em Lemma \ref{lemma41}} and {\em Lemma \ref{lemma42}} .
\end{proof}

\begin{theorem}\label{ertyui}
Assume that $X^{\ast}$ is the optimal solution to (\ref{l0formulation}), the linear mapping $\mathcal {A}$ satisfies RIP defined in {\em Definition \ref{erf}} with $\delta_{3r}\leq\frac{1}{20}$, and the step size satisfies
\begin{equation*}
\begin{aligned}
\frac{2\left(1+\delta_{3r}\right)\sqrt{1-\delta_{3r}}-\sqrt{-68\delta_{3r}^{3}-388\delta_{3r}^{2}-60\delta_{3r}+4}}{\left(16\delta_{3r}^{2}+96\delta_{3r}+16\right)\sqrt{1-\delta_{3r}}}\leq\eta\leq\frac{2\left(1+\delta_{3r}\right)\sqrt{1-\delta_{3r}}+\sqrt{-68\delta_{3r}^{3}-388\delta_{3r}^{2}-60\delta_{3r}+4}}{\left(16\delta_{3r}^{2}+96\delta_{3r}+16\right)\sqrt{1-\delta_{3r}}}
\end{aligned}
\end{equation*}
then the sequence produced by SVRG-ARM satisfies
\begin{equation*}
\begin{aligned}
\mathbb{E}\left(F\left(\widetilde{X}_{k+1}\right)-F\left(X^{\ast}\right)\right)&\leq\beta_{3r}\mathbb{E}\left(F\left(\widetilde{X}_{k}\right)-F\left(X^{\ast}\right)\right)\\\\
\end{aligned}
\end{equation*}
where $\beta_{3r}=\left(\mu_{3r}^{n}+\frac{\nu_{3r}\left(1-\mu_{3r}^{n}\right)}{1-\mu_{3r}}\right)<1$, $\mu_{3r}=\frac{1+\delta_{3r}}{1-\delta_{3r}}-2\eta\left(1+\delta_{3r}\right)\left(1-4\eta\left(1+\delta_{3r}\right)\right)$ and $\nu_{3r}=32\delta_{3r}\eta^{2}$.
\end{theorem}
\begin{proof}
Let $\Omega_{t}$ be a space spanned by $\widetilde{X}_{k}$, $X_{t}$ and $X^{\ast}$. Since rank$(\widetilde{X}_{k})=r$, rank$(X_{t})=r$ and rank$(X^{\ast})=r$, the rank of any matrix in $\Omega_{t}$ is at most $3r$. Then

\begin{equation}\label{eoru}
\begin{aligned}
\mathbb{E}_{i_{t}}\|X_{t+1}-X^{\ast}\|_{F}^{2}&=\mathbb{E}_{i_{t}}\|X_{t}-\eta V_{t}-X^{\ast}\|_{F}^{2}\\
&=\mathbb{E}_{i_{t}}\|X_{t}-\eta\mathcal {P}_{\Omega_{t}}\left(V_{t}\right)-X^{\ast}\|_{F}^{2}\\
&=\|X_{t}-X^{\ast}\|_{F}^{2}+\eta^{2}\mathbb{E}_{i_{t}}\|\mathcal {P}_{\Omega_{t}}\left(V_{t}\right)\|_{F}^{2}-2\eta\mathbb{E}_{i_{t}}\langle X_{t}-X^{\ast},\mathcal {P}_{\Omega_{t}}\left(V_{t}\right)\rangle\\
&=\|X_{t}-X^{\ast}\|_{F}^{2}+\eta^{2}\mathbb{E}_{i_{t}}\|\mathcal {P}_{\Omega_{t}}\left(V_{t}\right)\|_{F}^{2}-2\eta\mathbb{E}_{i_{t}}\langle X_{t}-X^{\ast},\mathcal {P}_{\Omega_{t}}\left(V_{t}\right)+\mathcal {P}_{\Omega_{t}^{c}}\left(V_{t}\right)\rangle\\
&=\|X_{t}-X^{\ast}\|_{F}^{2}+\eta^{2}\mathbb{E}_{i_{t}}\|\mathcal {P}_{\Omega_{t}}\left(V_{t}\right)\|_{F}^{2}-2\eta\mathbb{E}_{i_{t}}\langle X_{t}-X^{\ast},V_{t}\rangle\\
&=\|X_{t}-X^{\ast}\|_{F}^{2}+\eta^{2}\mathbb{E}_{i_{t}}\|\mathcal {P}_{\Omega_{t}}\left(V_{t}\right)\|_{F}^{2}-2\eta\langle X_{t}-X^{\ast},\nabla F\left(X_{t}\right)\rangle\\
&\leq\left(F\left(X_{t}\right)-F\left(X^{\ast}\right)\right)/\left(1-\delta_{3r}\right)+8\eta^{2}\left(1+\delta_{3r}\right)\left(F\left(X_{t}\right)-F\left(X^{\ast}\right)\right)+\\
&\eta^{2}\frac{32\delta_{3r}}{1+\delta_{3r}}\left(F\left(\widetilde{X}_{k}\right)-F\left(X^{\ast}\right)\right)-2\eta\left(F\left(X_{t}\right)-F\left(X^{\ast}\right)\right)\\
&=\left(\frac{1}{1-\delta_{3r}}-2\eta\left(1-4\eta\left(1+\delta_{3r}\right)\right)\right)\left(F\left(X_{t}\right)-F\left(X^{\ast}\right)\right)\\
&+\eta^{2}\frac{32\delta_{3r}}{1+\delta_{3r}}\left(F\left(\widetilde{X}_{k}\right)-F\left(X^{\ast}\right)\right)\\
\end{aligned}
\end{equation}
where the inequality follows from {\em Lemma \ref{reop}} and (\ref{rtyu}).
Taking $i_{t}=\{1,\ldots,m\}$ in (\ref{efer}), the equivalent conditions $(3)$ and $(2)$ in {\em{Lemma 4}} (\cite{XYZhou}), we can derive
\begin{equation}\label{bbbvv}
\begin{aligned}
\|X_{t+1}-X^{\ast}\|_{F}^{2}\geq \left(F\left(X_{t+1}\right)-F\left(X^{\ast}\right)\right)/\left(1+\delta_{3r}\right)
\end{aligned}
\end{equation}
Combining (\ref{eoru}) and (\ref{bbbvv}), we obtain
\begin{equation*}
\begin{aligned}
\mathbb{E}_{i_{t}}\left(F\left(X_{t+1}\right)-F\left(X^{\ast}\right)\right)&\leq\left(\frac{1+\delta_{3r}}{1-\delta_{3r}}-2\eta\left(1+\delta_{3r}\right)\left(1-4\eta\left(1+\delta_{3r}\right)\right)\right)\left(F\left(X_{t}\right)-F\left(X^{\ast}\right)\right)\\
&+32\delta_{3r}\eta^{2}\left(F\left(\widetilde{X}_{k}\right)-F\left(X^{\ast}\right)\right)\\
\end{aligned}
\end{equation*}
By recursively applying the above inequality over $t$, and noting that $\widetilde{X}_{k}=X_{0}$ and $\widetilde{X}_{k+1}=X_{n}$, we can obtain
\begin{equation*}
\begin{aligned}
\mathbb{E}\left(F\left(\widetilde{X}_{k+1}\right)-F\left(X^{\ast}\right)\right)&\leq\left(\mu_{3r}^{n}+\frac{\nu_{3r}\left(1-\mu_{3r}^{n}\right)}{1-\mu_{3r}}\right)\mathbb{E}\left(F\left(\widetilde{X}_{k}\right)-F\left(X^{\ast}\right)\right)\\
\end{aligned}
\end{equation*}
where $\mu_{3r}=\frac{1+\delta_{3r}}{1-\delta_{3r}}-2\eta\left(1+\delta_{3r}\right)\left(1-4\eta\left(1+\delta_{3r}\right)\right)$ and $\nu_{3r}=32\delta_{3r}\eta^{2}$.

By choosing $\delta_{3r}\leq\frac{1}{20}$ and $\frac{2\left(1+\delta_{3r}\right)\sqrt{1-\delta_{3r}}-\sqrt{-68\delta_{3r}^{3}-388\delta_{3r}^{2}-60\delta_{3r}+4}}{\left(16\delta_{3r}^{2}+96\delta_{3r}+16\right)\sqrt{1-\delta_{3r}}}\leq\eta\leq\frac{2\left(1+\delta_{3r}\right)\sqrt{1-\delta_{3r}}+\sqrt{-68\delta_{3r}^{3}-388\delta_{3r}^{2}-60\delta_{3r}+4}}{\left(16\delta_{3r}^{2}+96\delta_{3r}+16\right)\sqrt{1-\delta_{3r}}}$, we have $\beta_{3r}=\left(\mu_{3r}^{n}+\frac{\nu_{3r}\left(1-\mu_{3r}^{n}\right)}{1-\mu_{3r}}\right)<1$.
\end{proof}
\begin{figure}[H]
\centering
\subfigure[Frequency of exact recovery as a function of rank]{
\begin{minipage}[b]{0.321\textwidth}
\label{fig1: parameters-a}\includegraphics[width=1.1\textwidth]{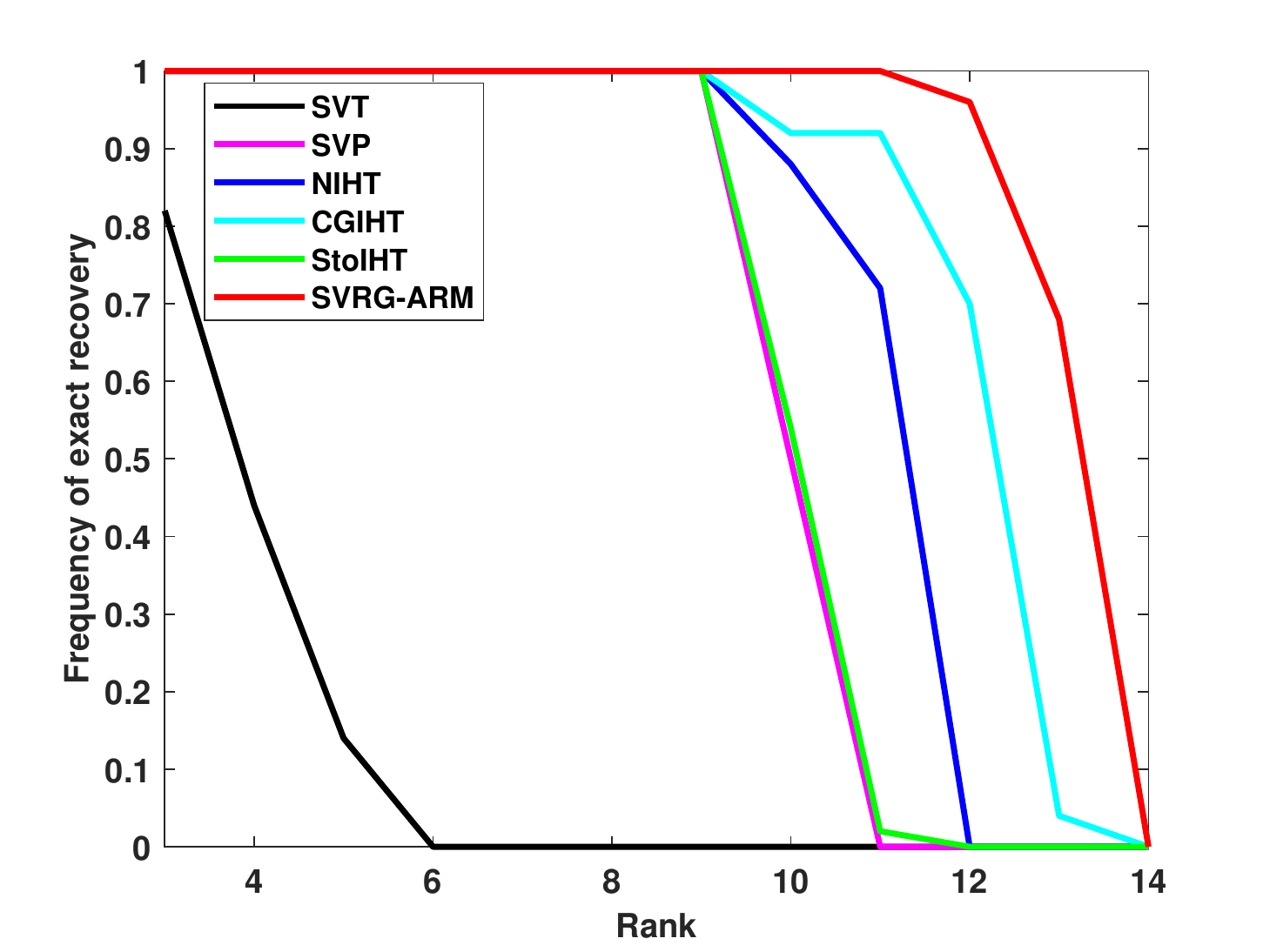}
\end{minipage}}
\subfigure[Convergence speed]{
\begin{minipage}[b]{0.321\textwidth}
\label{fig: parameters-b} \includegraphics[width=1.1\textwidth]{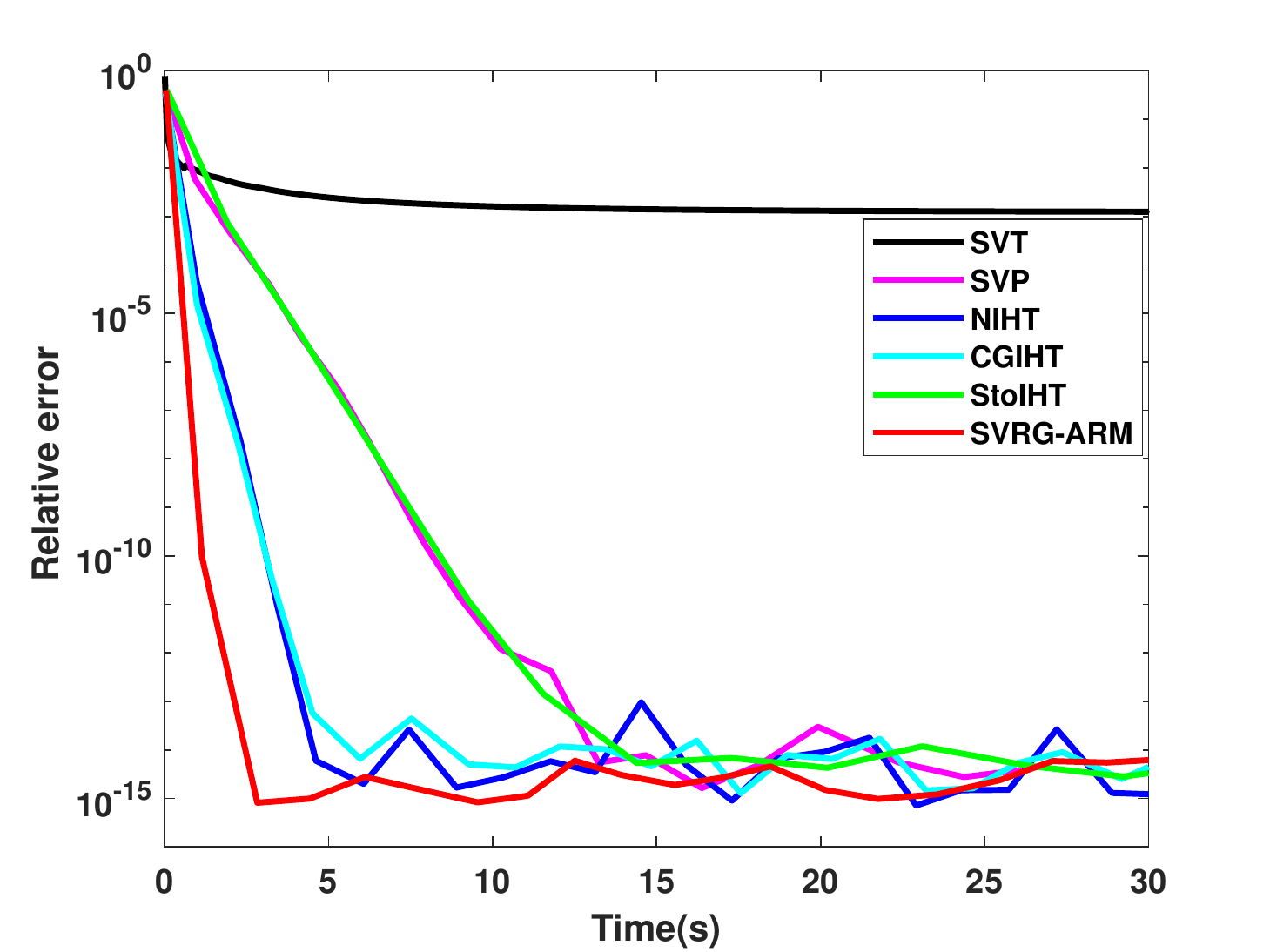}
\end{minipage}}
\subfigure[Normalized mean square
error as a function of noise level]{
\begin{minipage}[b]{0.321\textwidth}
\label{fig: parameters-c} \includegraphics[width=1.1\textwidth]{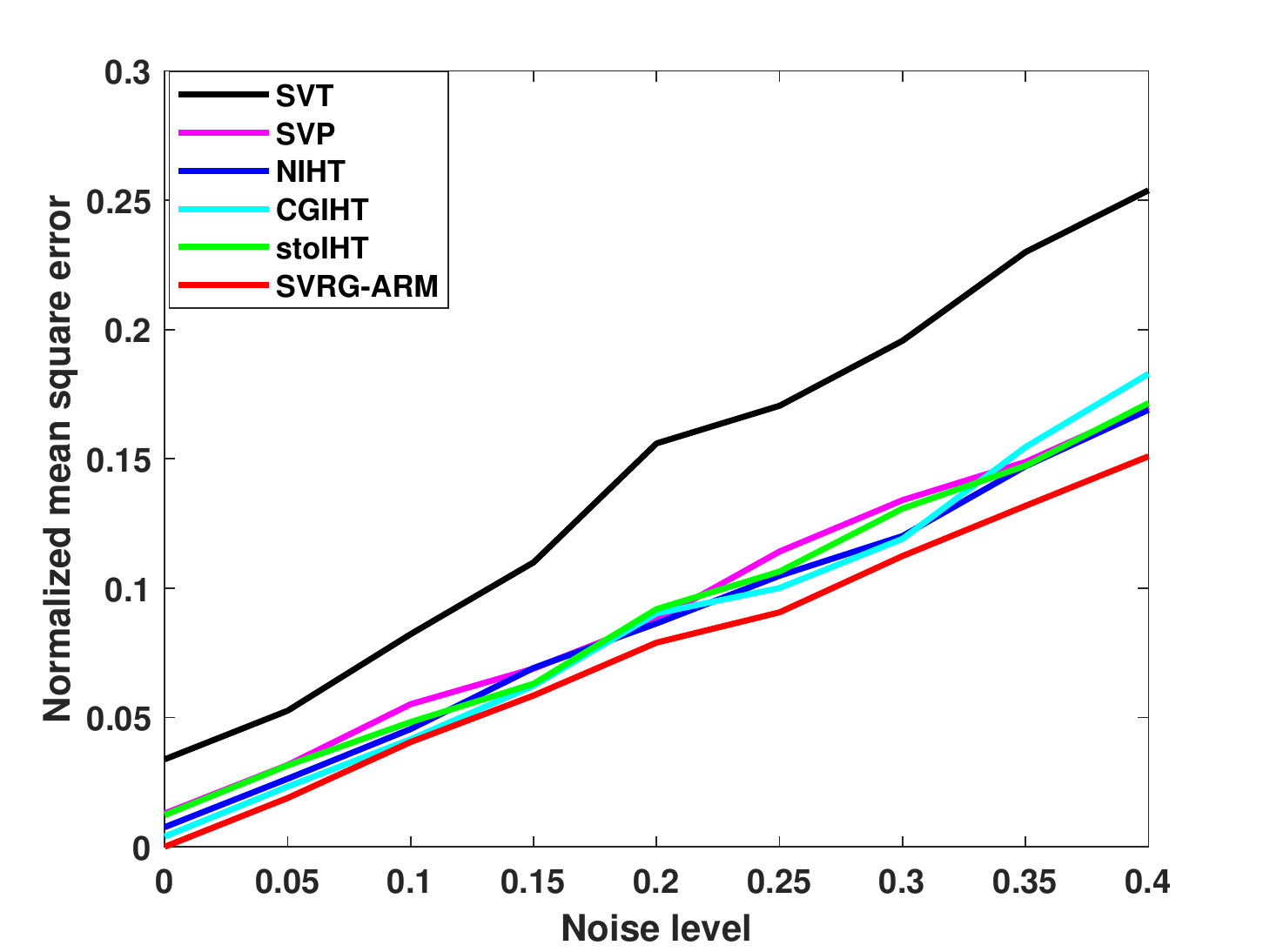}
\end{minipage}}
\quad
\caption{\small (a) Frequency of exact recovery as a function of rank. (b) Convergence speed. (c) Normalized mean square
error as a function of noise level}
\label{fig1}
\end{figure}

\begin{figure}[H]
\centering
\subfigure[SVT]{
\begin{minipage}[b]{0.321\textwidth}
\label{fig2: parameters-a}\includegraphics[width=1.1\textwidth]{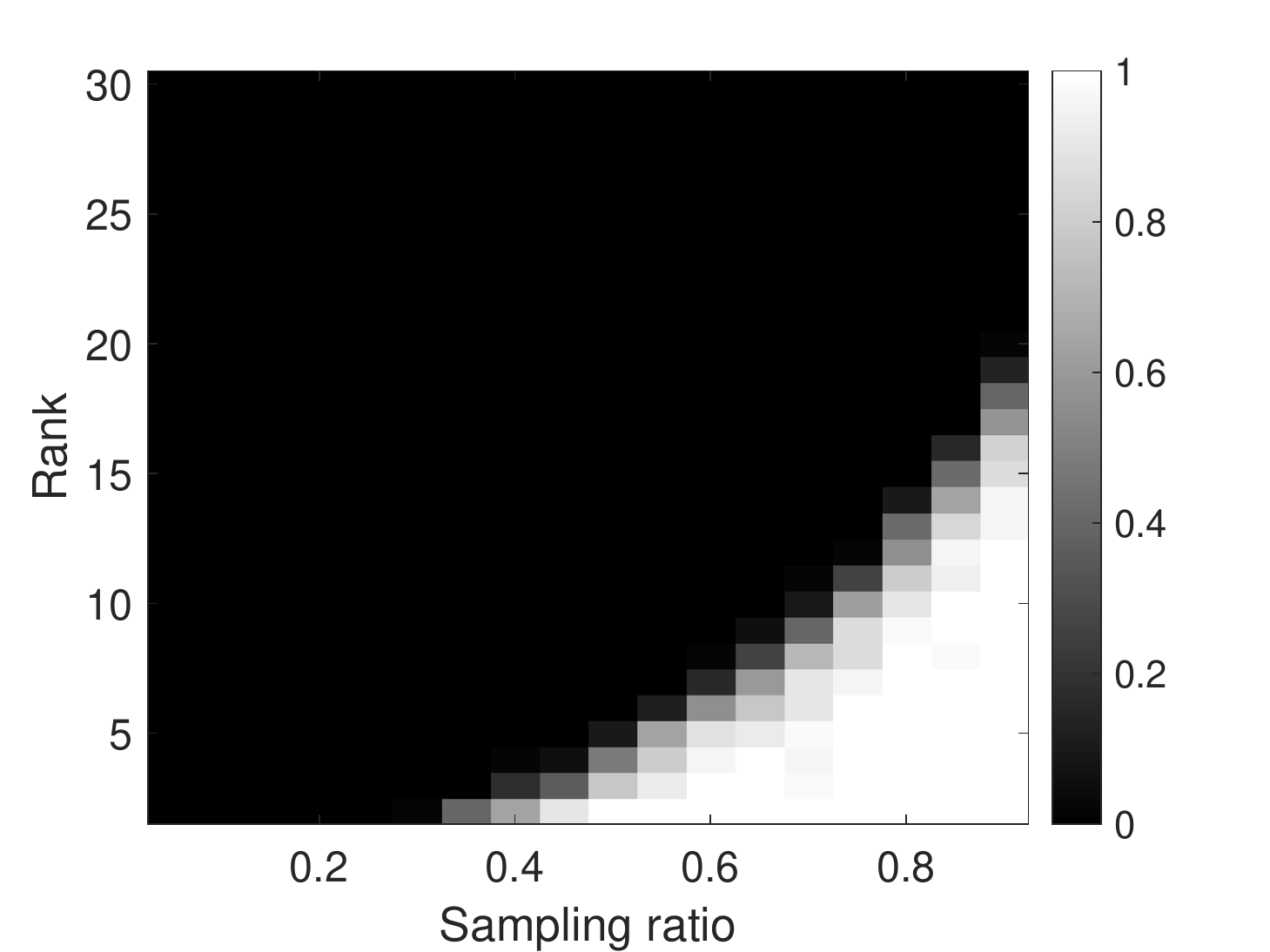}
\end{minipage}}
\subfigure[SVP]{
\begin{minipage}[b]{0.321\textwidth}
\label{fig2: parameters-b} \includegraphics[width=1.1\textwidth]{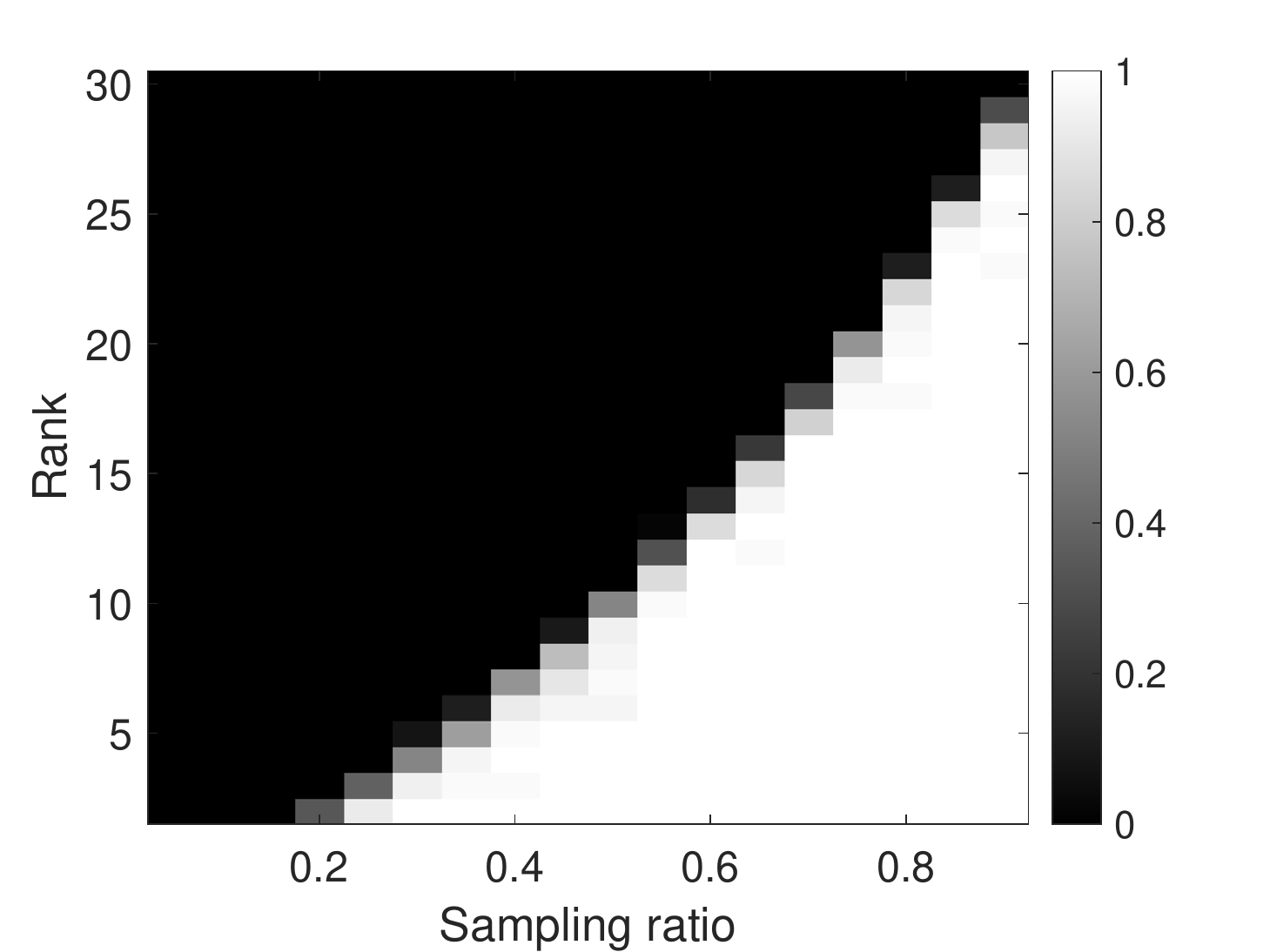}
\end{minipage}}
\subfigure[NIHT]{
\begin{minipage}[b]{0.321\textwidth}
\label{fig2: parameters-c} \includegraphics[width=1.1\textwidth]{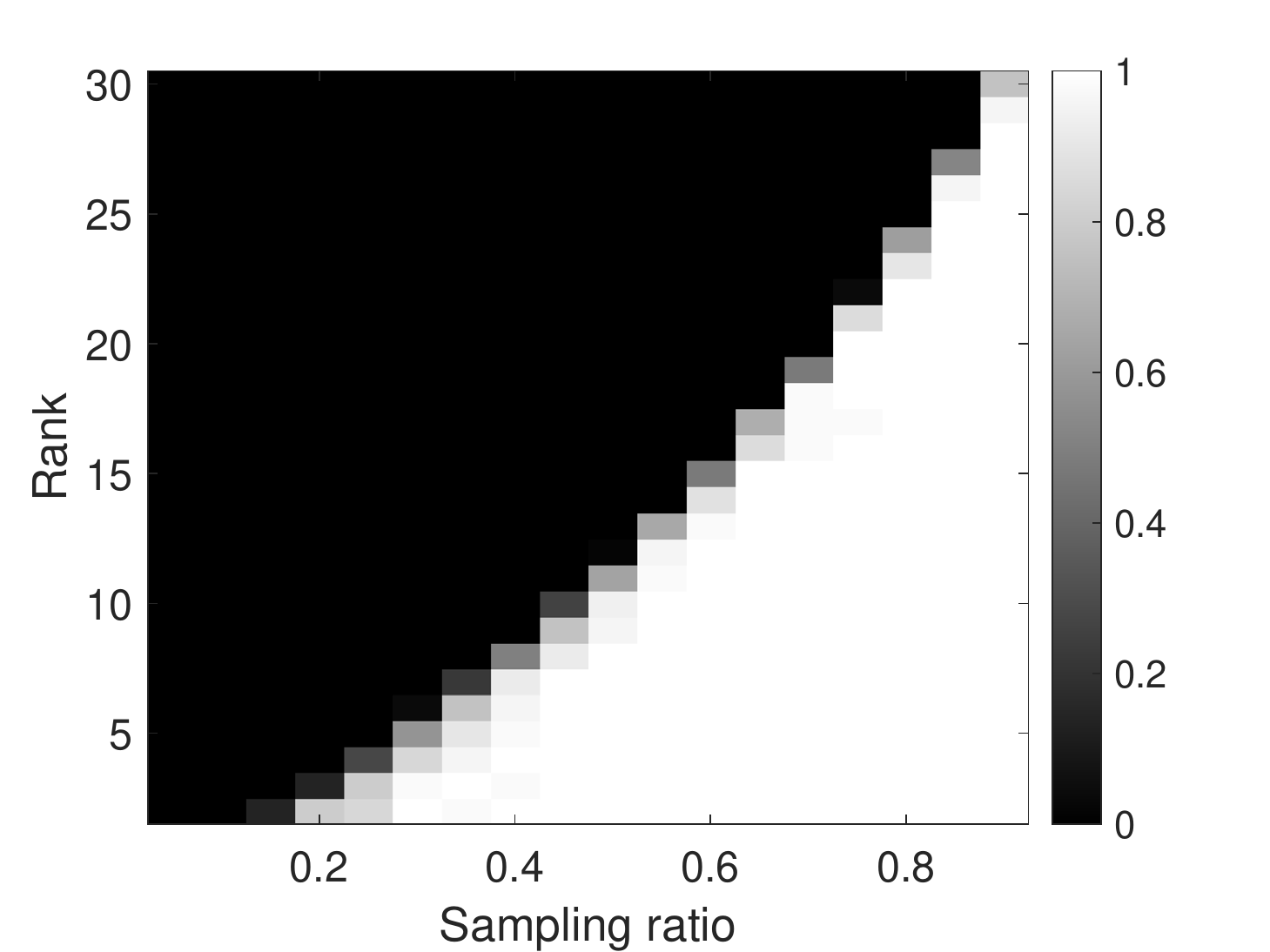}
\end{minipage}}
\subfigure[CGIHT]{
\begin{minipage}[b]{0.321\textwidth}
\label{fig2: parameters-d}\includegraphics[width=1.1\textwidth]{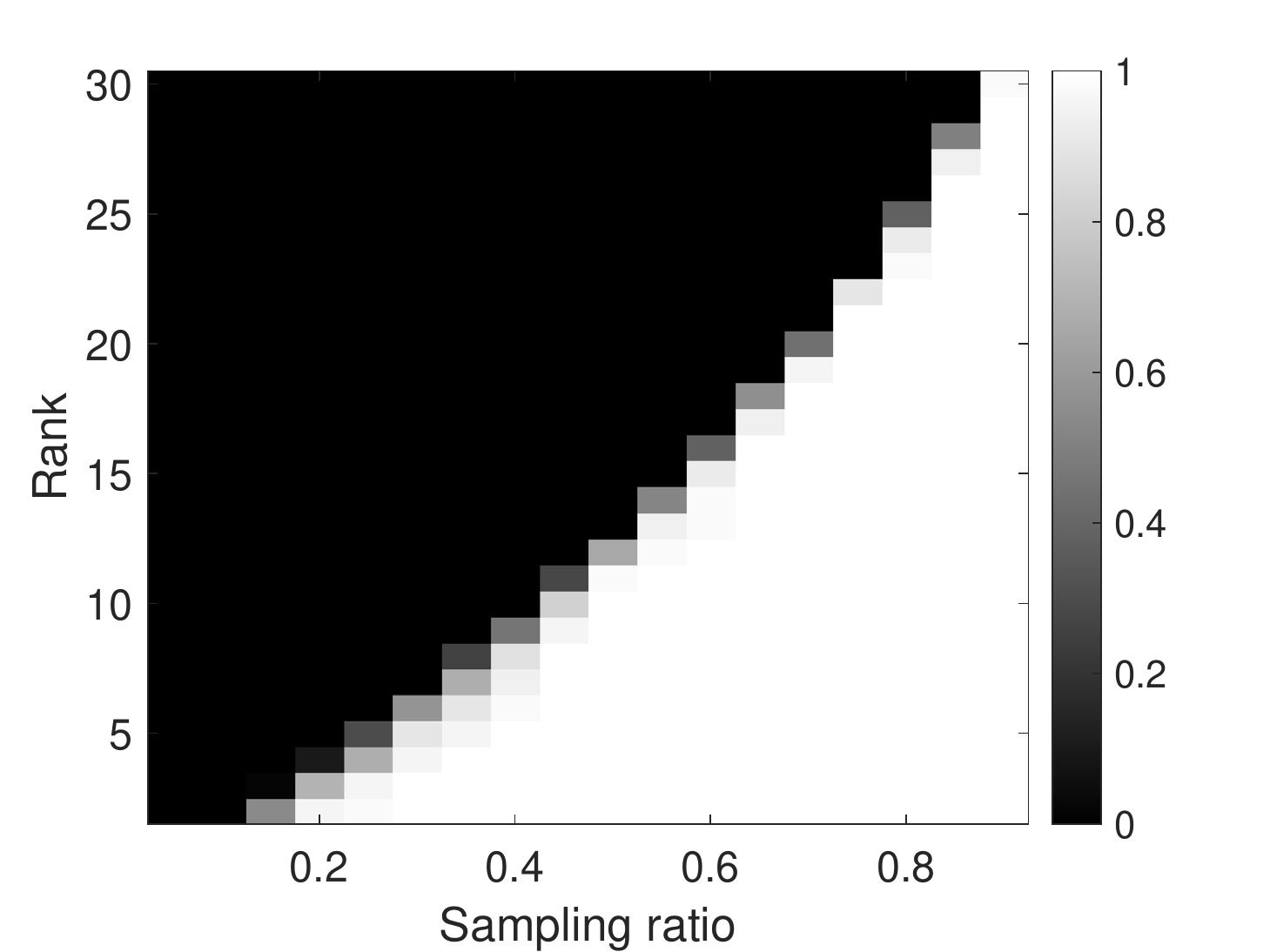}
\end{minipage}}
\subfigure[StoIHT]{
\begin{minipage}[b]{0.321\textwidth}
\label{fig2: parameters-e} \includegraphics[width=1.1\textwidth]{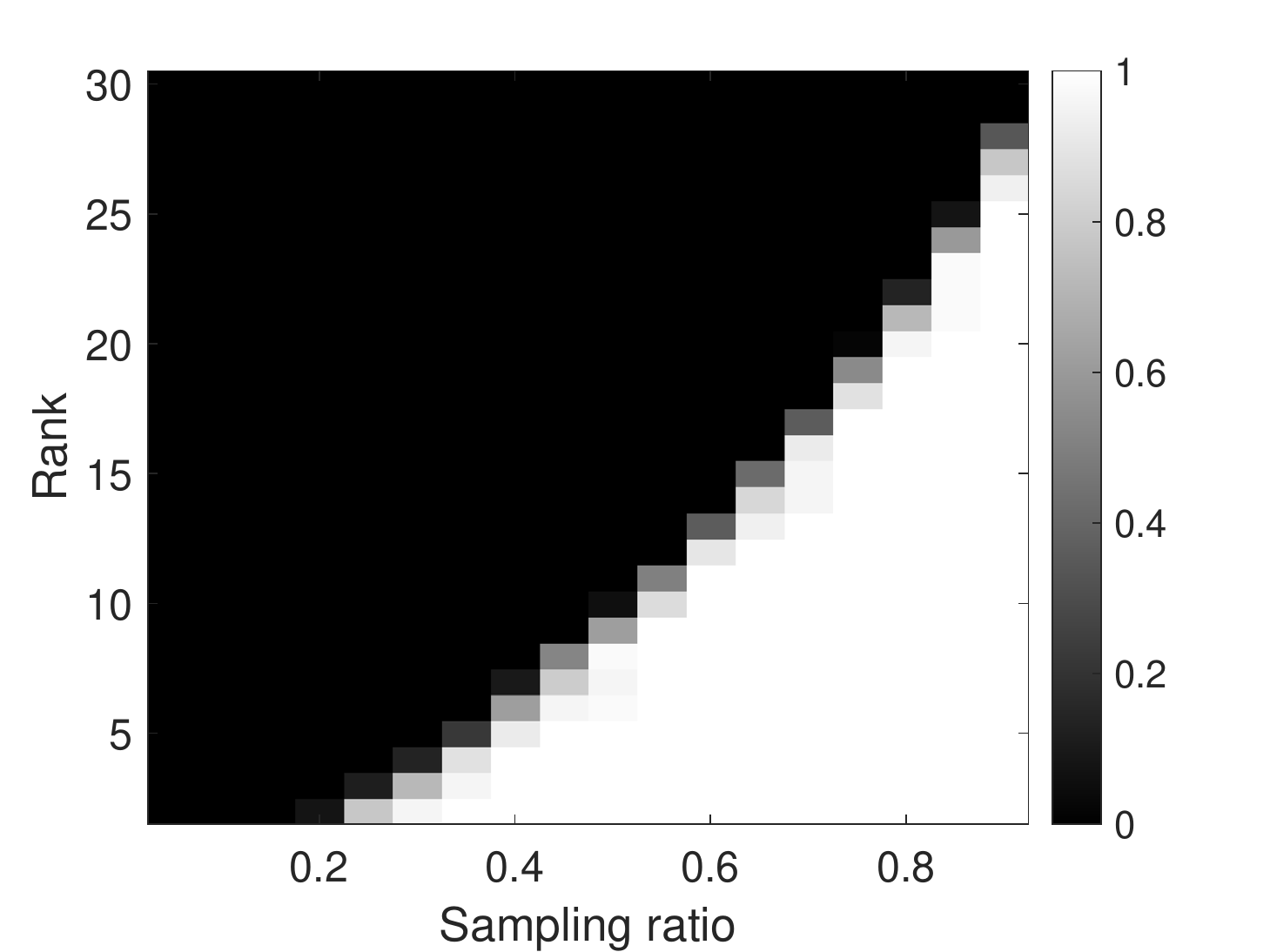}
\end{minipage}}
\subfigure[SVRG-ARM]{
\begin{minipage}[b]{0.321\textwidth}
\label{fig2: parameters-f} \includegraphics[width=1.1\textwidth]{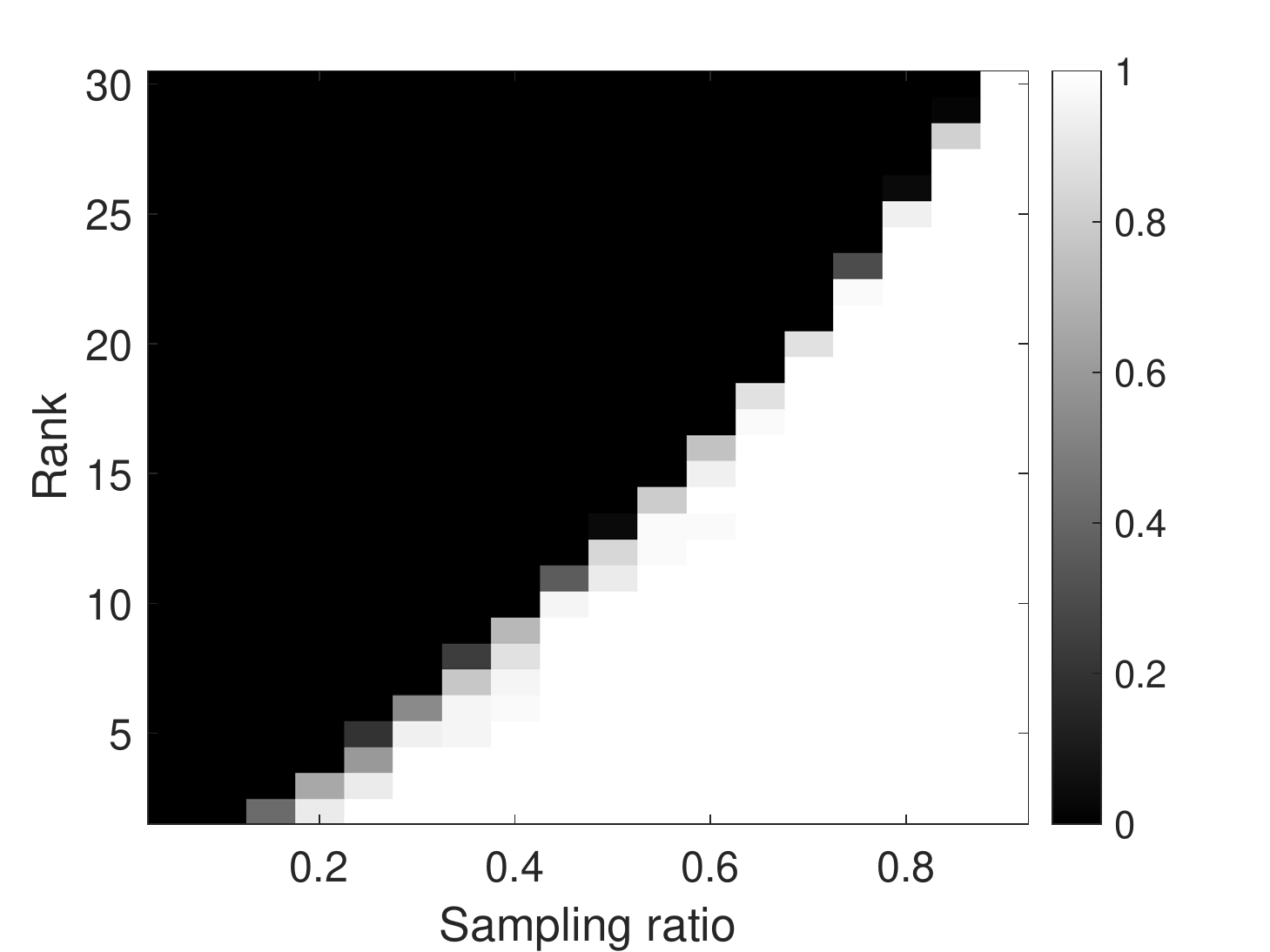}
\end{minipage}}
\quad
\caption{\small Phase transition of low-rank matrix completion using (a) SVT. (b) SVP. (c) NIHT. (d) CGIHT. (c) StoIHT. (d) SVRG-ARM.}
\label{figure2}
\end{figure}

By {\em Theorem \ref{ertyui}}, we have $\mathbb{E}\left(F\left(\widetilde{X}_{k}\right)-F\left(X^{\ast}\right)\right)\leq\beta_{3r}^{k}\mathbb{E}\left(F\left(\widetilde{X}_{0}\right)-F\left(X^{\ast}\right)\right)$. To obtain accuracy of $\epsilon$, i.e., $\mathbb{E}\left(F\left(\widetilde{X}_{k}\right)-F\left(X^{\ast}\right)\right)\leq\epsilon$, SVRG-ARM needs to take $k=\mathcal {O}\left(\text{log}\left(1/\varepsilon\right)\right)$ outer loops. The computational complexity of the proposed algorithm mainly includes two parts: the computation of gradients and singular value decompositions. The complexity of calculating gradients is $\mathcal {O}\left(m+nb\right)$, where $n$ is the number of inner loops and $b=$max$\{i_{0},\ldots,i_{n-1}\}$. Besides, a singular value decomposition is required in each iteration to project the variable $W^{t}$ back onto the rank $r$ matrix feasible solution space and the corresponding complexity can be $\mathcal {O}\left(r^{3}\right)$, where $r$ is the rank of low-rank matrix \cite{Kwei}. Therefore, the overall computational complexity of SVRG-ARM is $\mathcal {O}\left(\left(m+nb+r^{3}\right)\text{log}\left(1/\varepsilon\right)\right)$. If $f_{i}(x)$ is $L$--Lipschitz smooth and $F(x)$ is $\mu$-strongly convex, deterministic full gradient descent method needs $\mathcal {O}\left(\sqrt{\kappa}\text{log}\left(1/\varepsilon\right)\right)$ iterations to find an $\epsilon$-accurate solution, where $\kappa$ is the condition number $L/\mu$ \cite{GD1,GD2}. The overall computational complexity of deterministic full gradient descent method is $\mathcal {O}\left(\left(m+r^{3}\right)\sqrt{\kappa}\text{log}\left(1/\varepsilon\right)\right)$. Thus SVRG-ARM presents a significant improvement over deterministic full gradient descent method when $\kappa$ is large, which has also been validated by numerical experiments in Section $5.1$.
\begin{figure}[H]
\centering
\subfigure[Original image]{
\begin{minipage}[b]{0.23\textwidth}
\label{fig3: parameters-a}\includegraphics[width=1.1\textwidth]{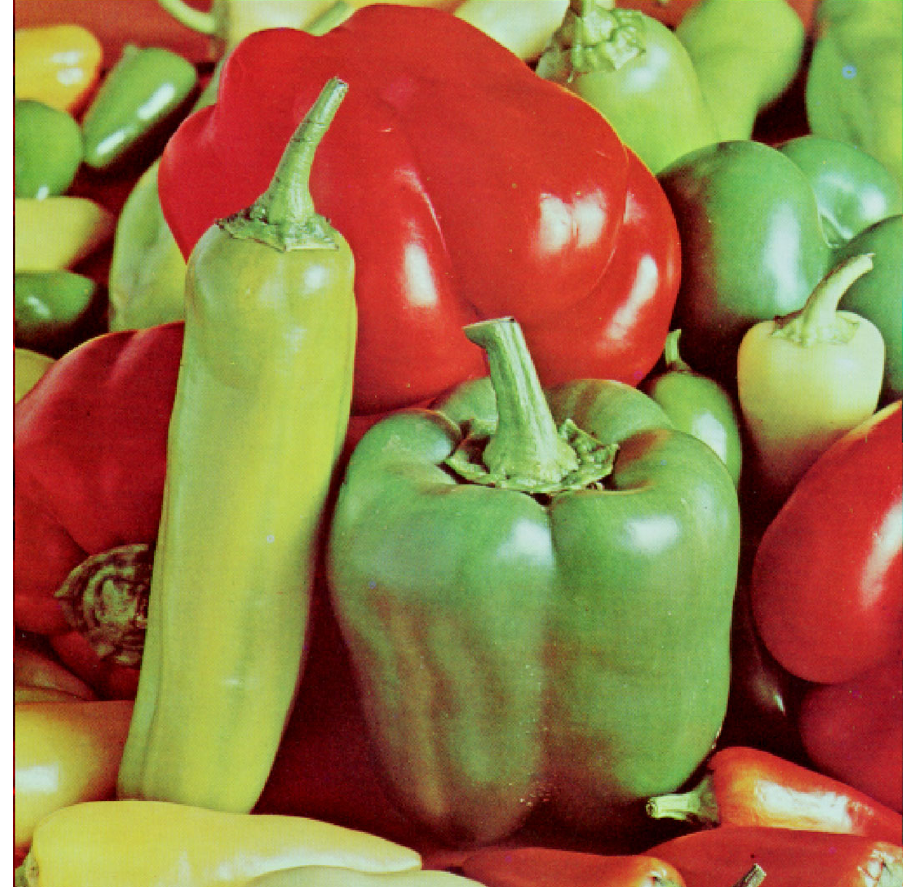}
\end{minipage}}
\subfigure[Observed image with missing pixels]{
\begin{minipage}[b]{0.23\textwidth}
\label{fig3: parameters-b} \includegraphics[width=1.1\textwidth]{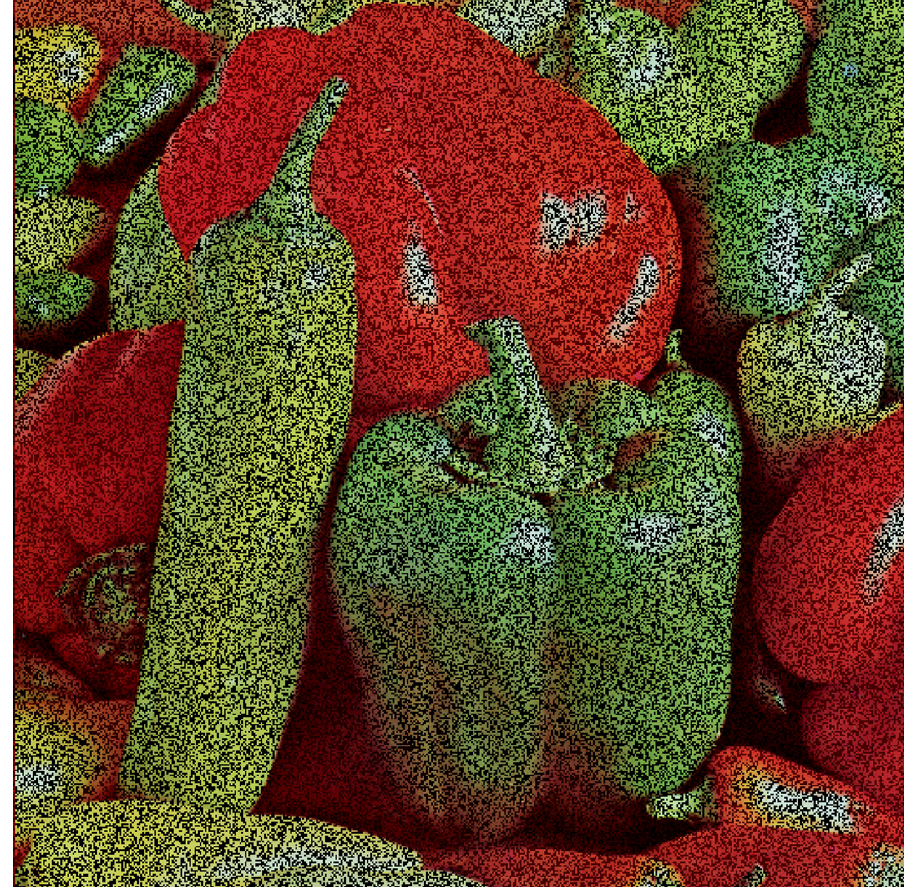}
\end{minipage}}
\subfigure[SVT $29.1242/0.8874$]{
\begin{minipage}[b]{0.23\textwidth}
\label{fig3: parameters-c} \includegraphics[width=1.1\textwidth]{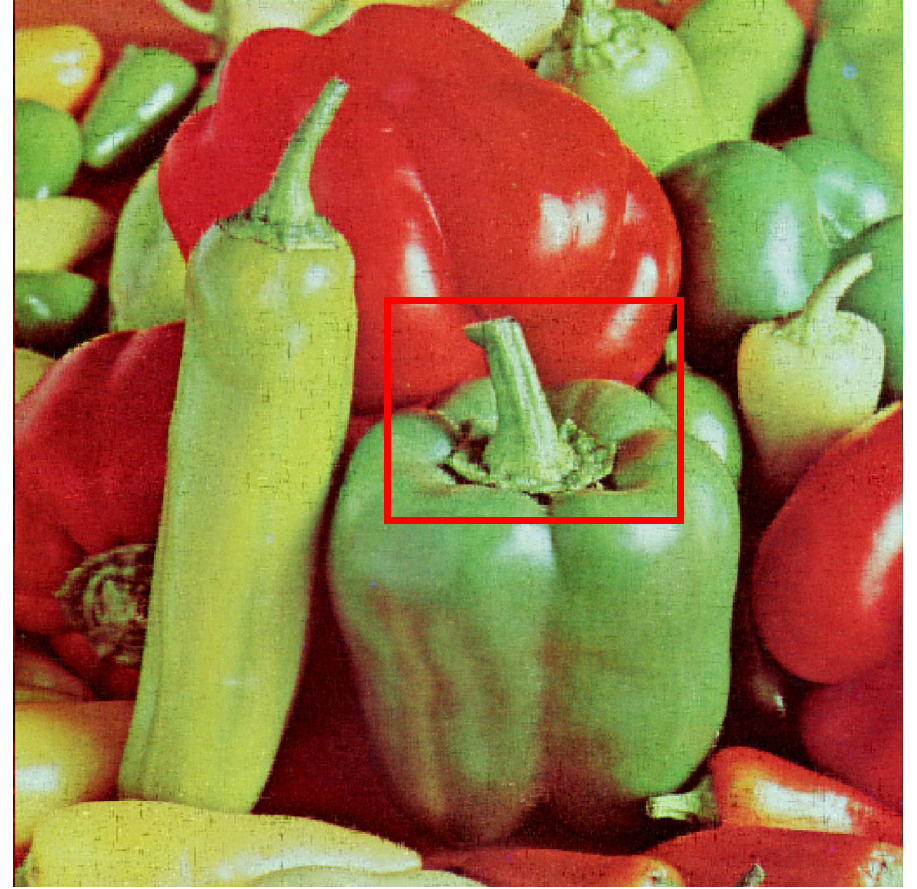}
\end{minipage}}
\subfigure[SVP $31.5434/0.9305$]{
\begin{minipage}[b]{0.23\textwidth}
\label{fig3: parameters-e} \includegraphics[width=1.1\textwidth]{pepper_IHT.pdf}
\end{minipage}}
\subfigure[NIHT $32.0818/0.9375$]{
\begin{minipage}[b]{0.23\textwidth}
\label{fig3: parameters-f} \includegraphics[width=1.1\textwidth]{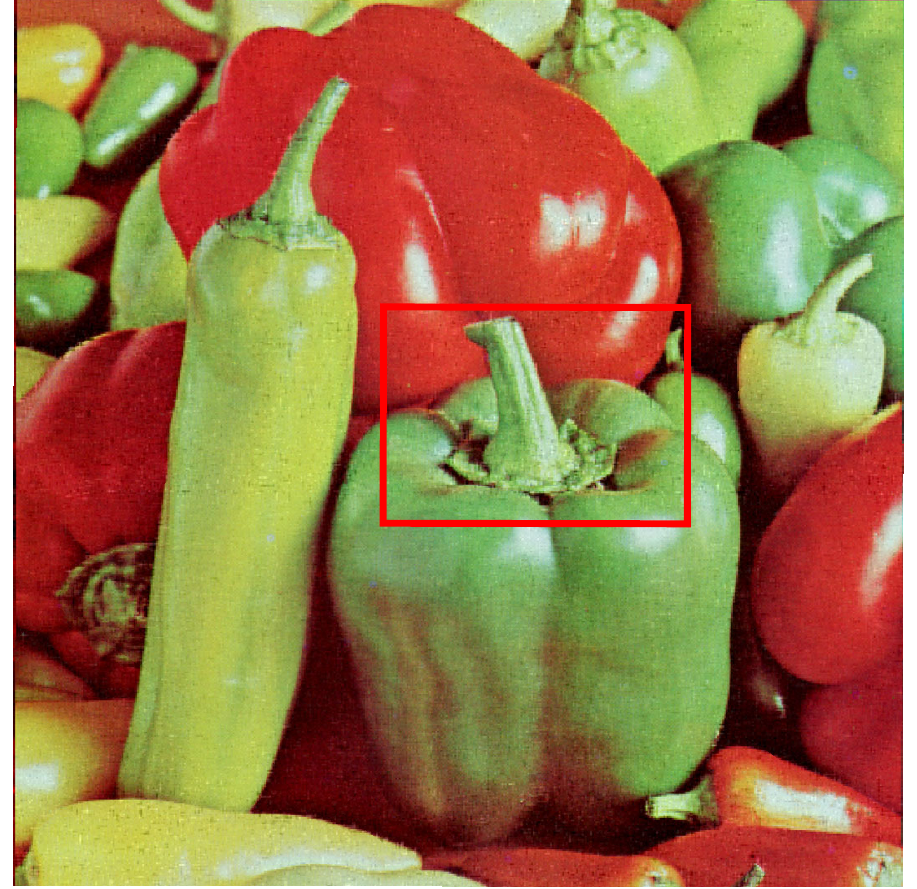}
\end{minipage}}
\subfigure[CGIHT $31.3653/0.9247$]{
\begin{minipage}[b]{0.23\textwidth}
\label{fig3: parameters-h} \includegraphics[width=1.1\textwidth]{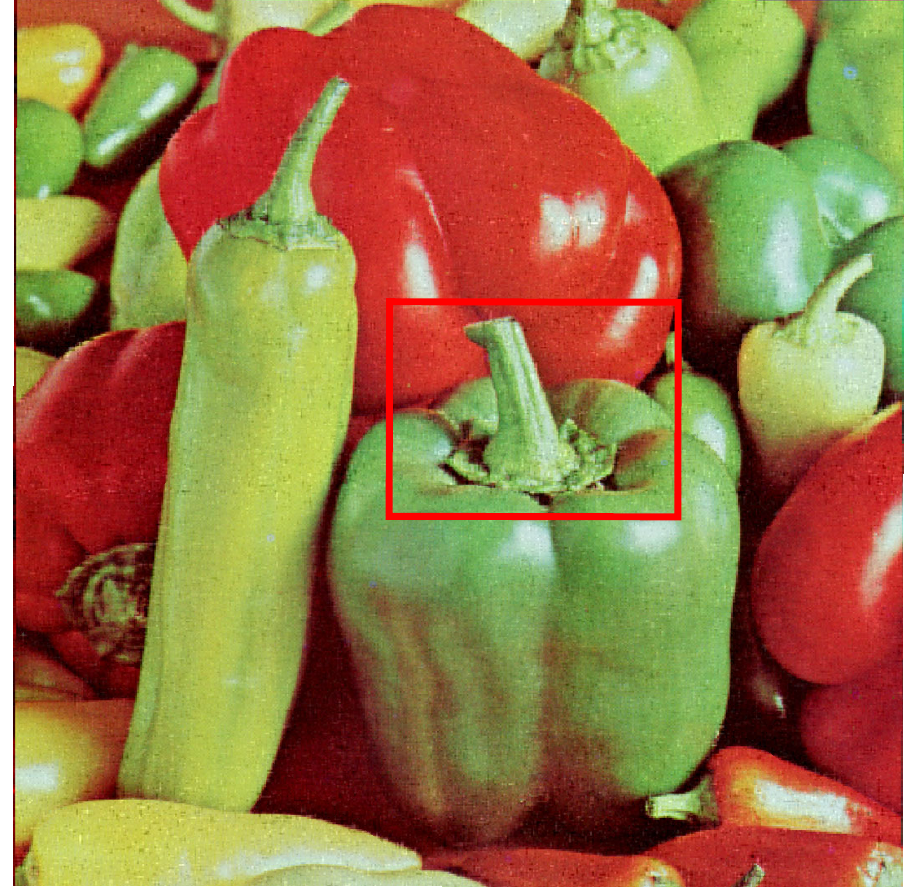}
\end{minipage}}
\subfigure[StoIHT $30.5716/0.8988$]{
\begin{minipage}[b]{0.23\textwidth}
\label{fig3: parameters-i} \includegraphics[width=1.1\textwidth]{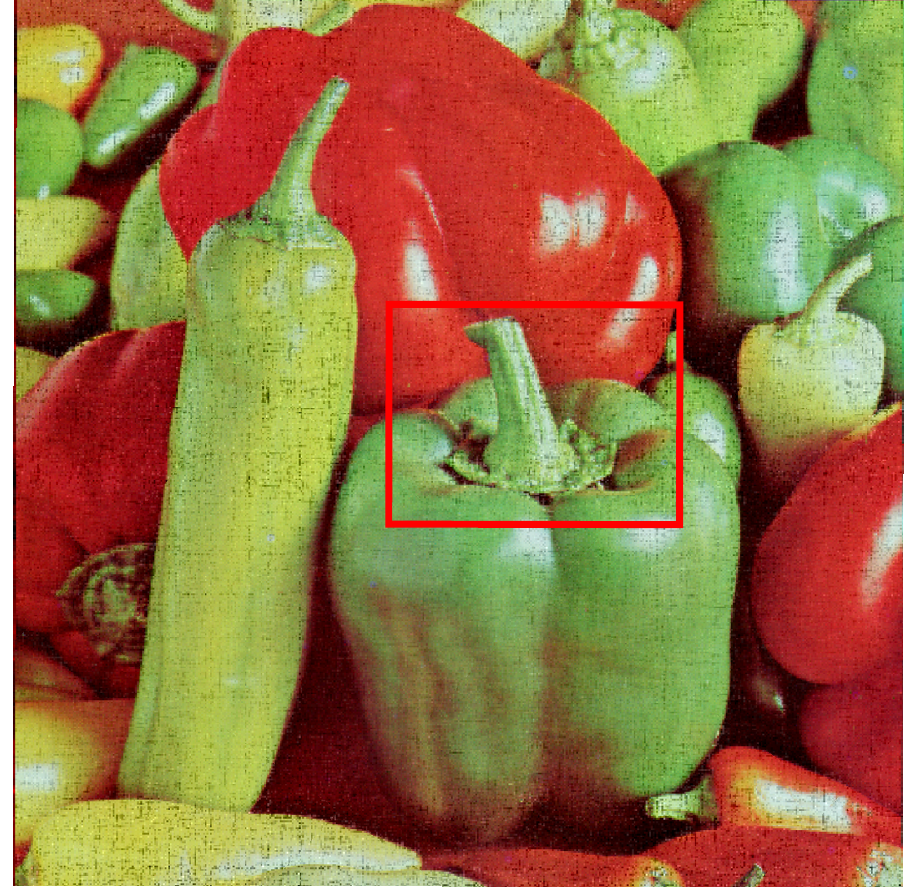}
\end{minipage}}
\subfigure[SVRG-ARM $\bf{32.1117/0.938}$]{
\begin{minipage}[b]{0.23\textwidth}
\label{fig3: parameters-j} \includegraphics[width=1.1\textwidth]{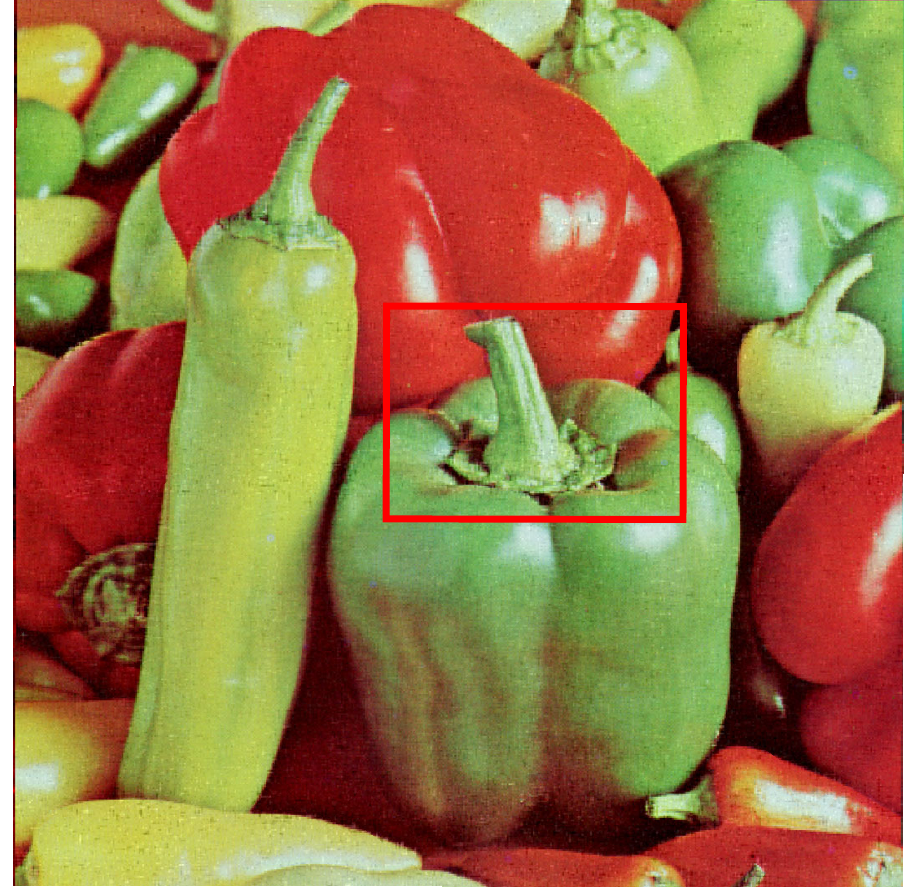}
\end{minipage}}
\quad
\caption{\small Comparison of matrix completion algorithms for image inpainting.(a) Original image. (b) Observed image with missing pixels. (c)-(h) Recovered images by SVT, SVP, NIHT, CGIHT, StoIHT, SVRG-ARM.}
\label{fig4}
\end{figure}

\section{Numerical experiments}
In this section, we present numerical results on synthetic and real data to validate the proposed algorithm. For comprehensive and complete comparisons, we first compare performance within the class of gradient descent algorithms including singular value thresholding (SVT) \cite{SVT}, singular value projection (SVP) \cite{PJain}, normalized iterative hard thresholding (NIHT) \cite{Tannern}, conjugate gradient iterative hard thresholding (CGIHT) \cite{Kwei}, stochastic iterative hard thresholding (StoIHT) \cite{Nguyen}. Among these algorithms, SVP is the simplest iterative hard thresholding gradient descent algorithm with fixed stepsize, while SVT is the iterative soft-thresholding gradient descent algorithm. NIHT is the modified iterative hard thresholding gradient descent algorithm with an adaptive stepsize. SVP and NIHT need to calculate the full gradient at each iteration. CGIHT generates the current estimate along the Riemannian conjugate gradient descent. StoIHT is based on stochastic gradient descent, and SVRG-ARM is designed to reduce the variance of stochastic gradient descent. In addition, we utilize Barzilai-Borwein (BB) \cite{Barzilai,CTan} method to automatically calculate step sizes, where we set the step size $\eta_{k}=\|\widetilde{X}_{k}-\widetilde{X}_{k-1}\|_{F}^{2}/(n\langle\widetilde{X}_{k}-\widetilde{X}_{k-1},g_{k}-g_{k-1}\rangle)$ at each iteration. 

Then the overall performance of SVRG-ARM in terms of execution-time and frequency of exact recovery is compared with other state-of-the-art algorithms including matrix factorization based method solved by ScaledASD \cite{yinn}, nuclear norm minimization (NNM) based
method solved by augmented Lagrange multiplier method \cite{cLin}, iterative reweighted nuclear norm (IRNN) \cite{bZhang,Mohan,Fornasier}, truncated nuclear norm regularization (TNNR) \cite{YHu,bZhang},  $\ell_{p}$ quasi-norm ($0<p<1$) \cite{FPNie,YXie}.

The associated matlab codes can be downloaded from the authors’ webpages or provided by authors in personal communication. A matlab implementation of the proposed algorithm is also available at \url{https://www.dropbox.com/s/9gte2as7gcarl80/SVRG-ARM.zip?dl=0}.
\begin{figure}[H]
\centering
\subfigure[Original image]{
\begin{minipage}[b]{0.25\textwidth}
\label{fig3: parameters-a}\includegraphics[width=1.1\textwidth]{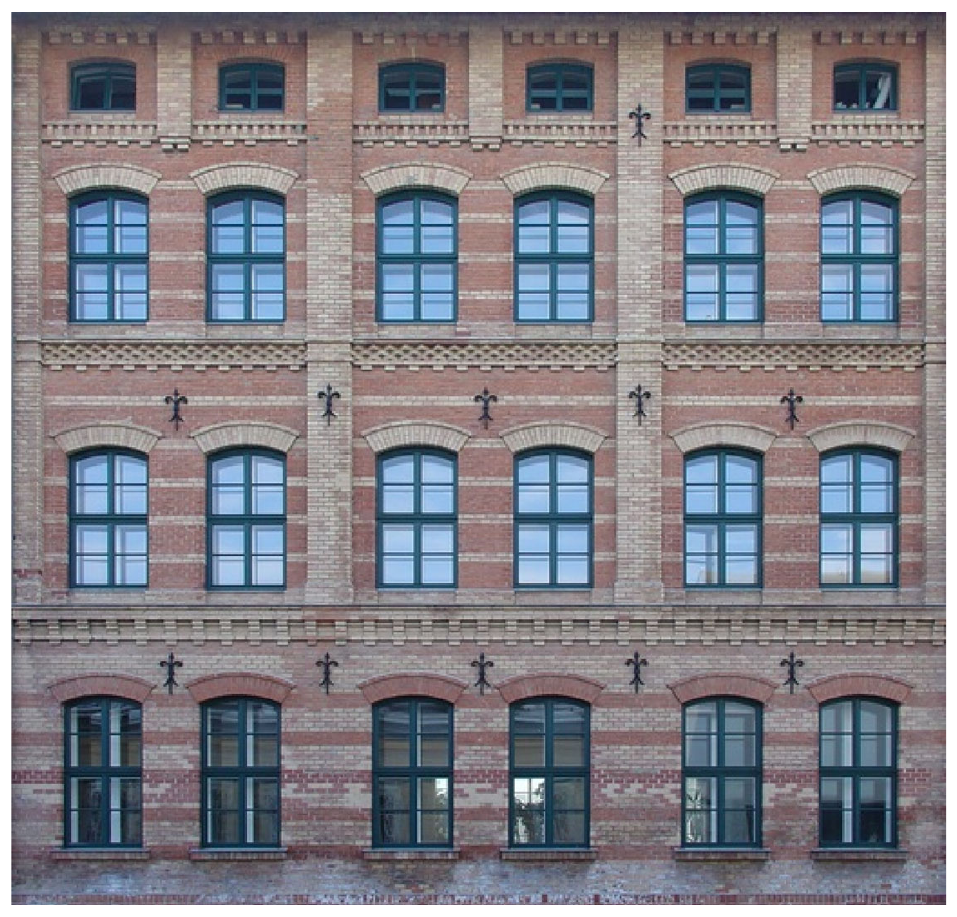}
\end{minipage}}
\subfigure[Observed image with missing pixels]{
\begin{minipage}[b]{0.25\textwidth}
\label{fig3: parameters-b} \includegraphics[width=1.1\textwidth]{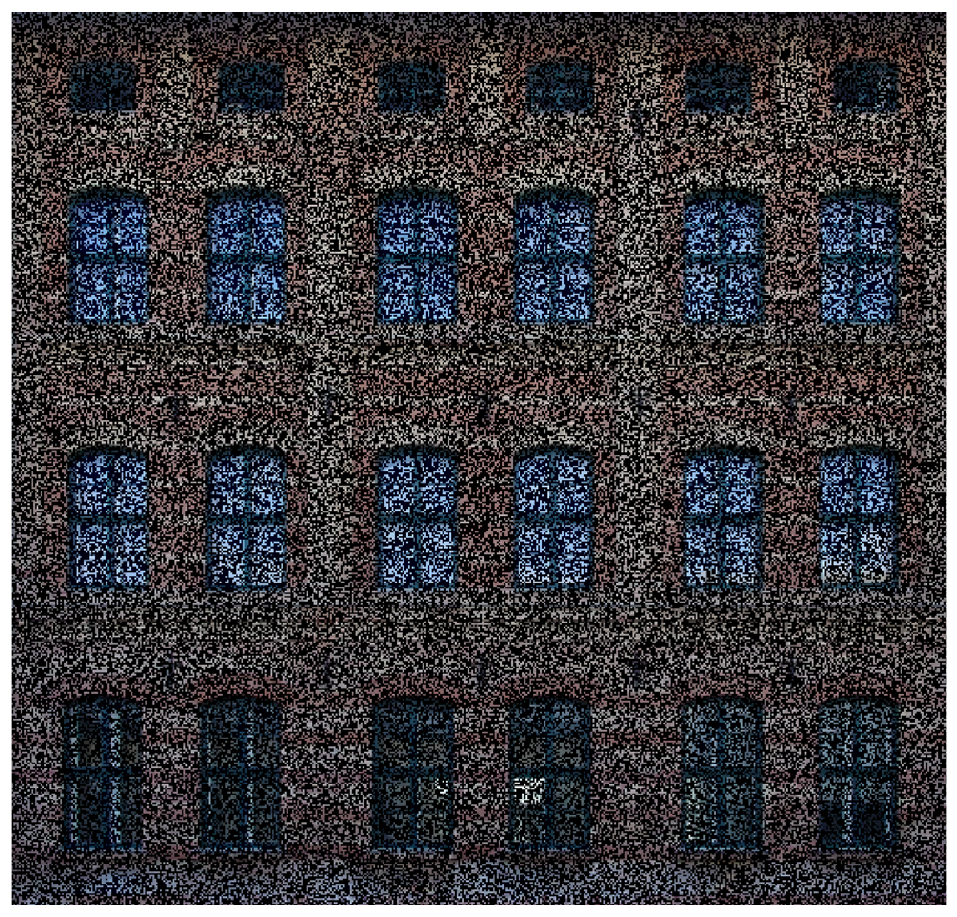}
\end{minipage}}
\subfigure[SVT $21.0977/0.7931$]{
\begin{minipage}[b]{0.25\textwidth}
\label{fig3: parameters-c} \includegraphics[width=1.1\textwidth]{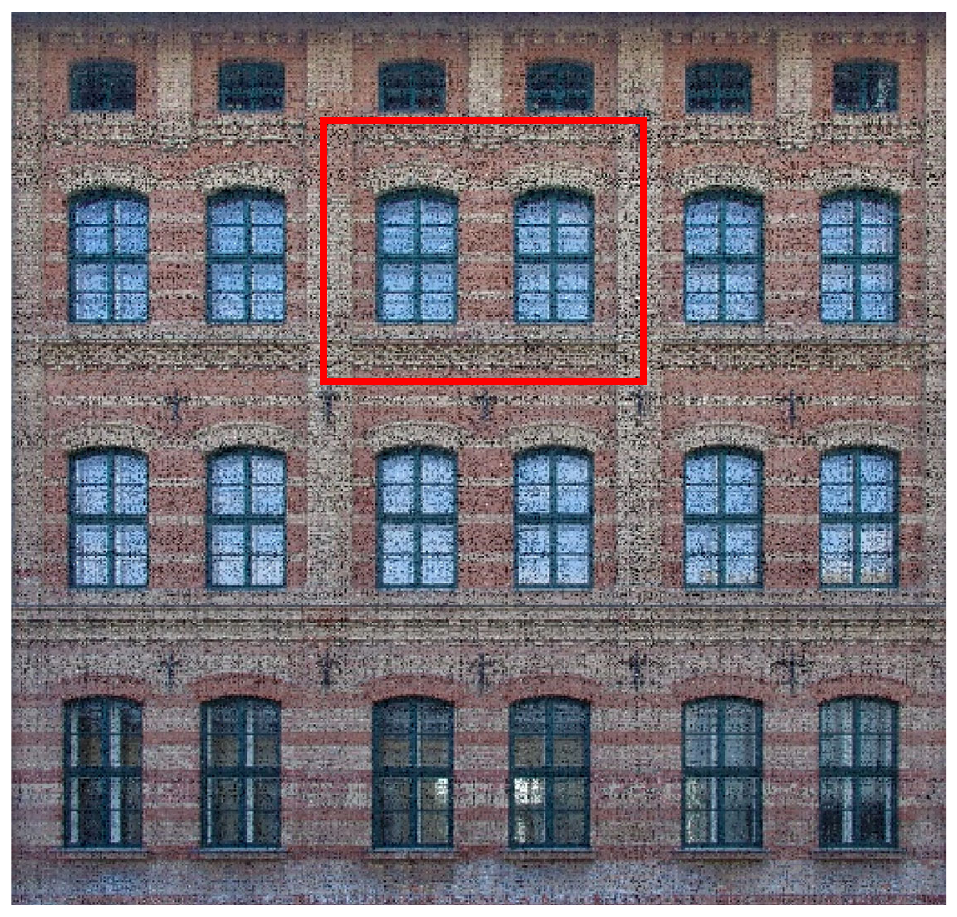}
\end{minipage}}
\subfigure[SVP $21.7933/0.8044$]{
\begin{minipage}[b]{0.25\textwidth}
\label{fig3: parameters-e} \includegraphics[width=1.1\textwidth]{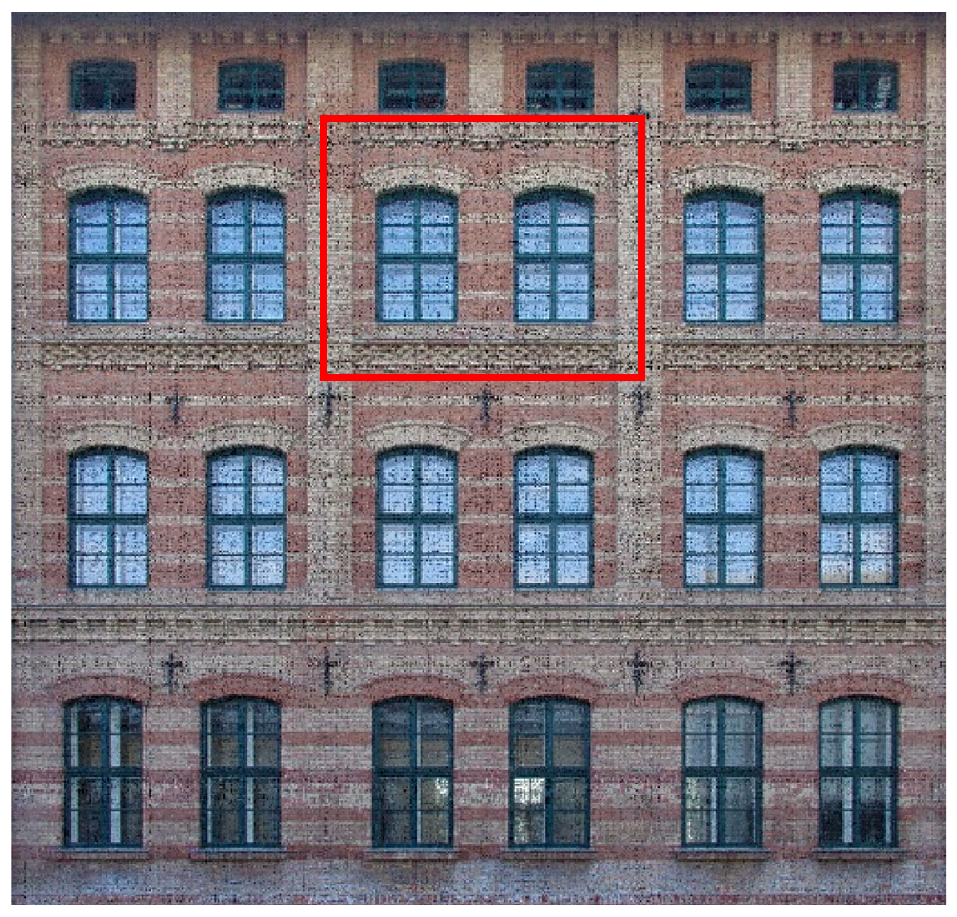}
\end{minipage}}
\subfigure[NIHT $25.0522/0.8833$]{
\begin{minipage}[b]{0.25\textwidth}
\label{fig3: parameters-f} \includegraphics[width=1.1\textwidth]{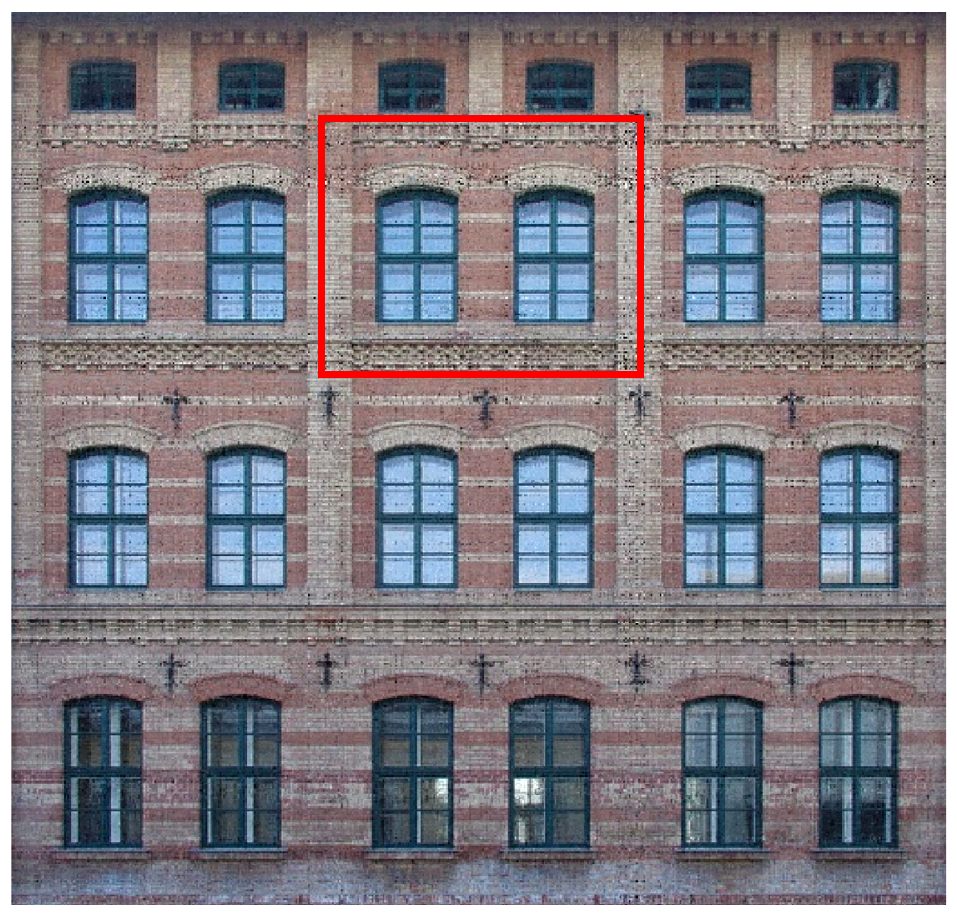}
\end{minipage}}
\subfigure[CGIHT $24.2795/0.8689$]{
\begin{minipage}[b]{0.25\textwidth}
\label{fig3: parameters-h} \includegraphics[width=1.1\textwidth]{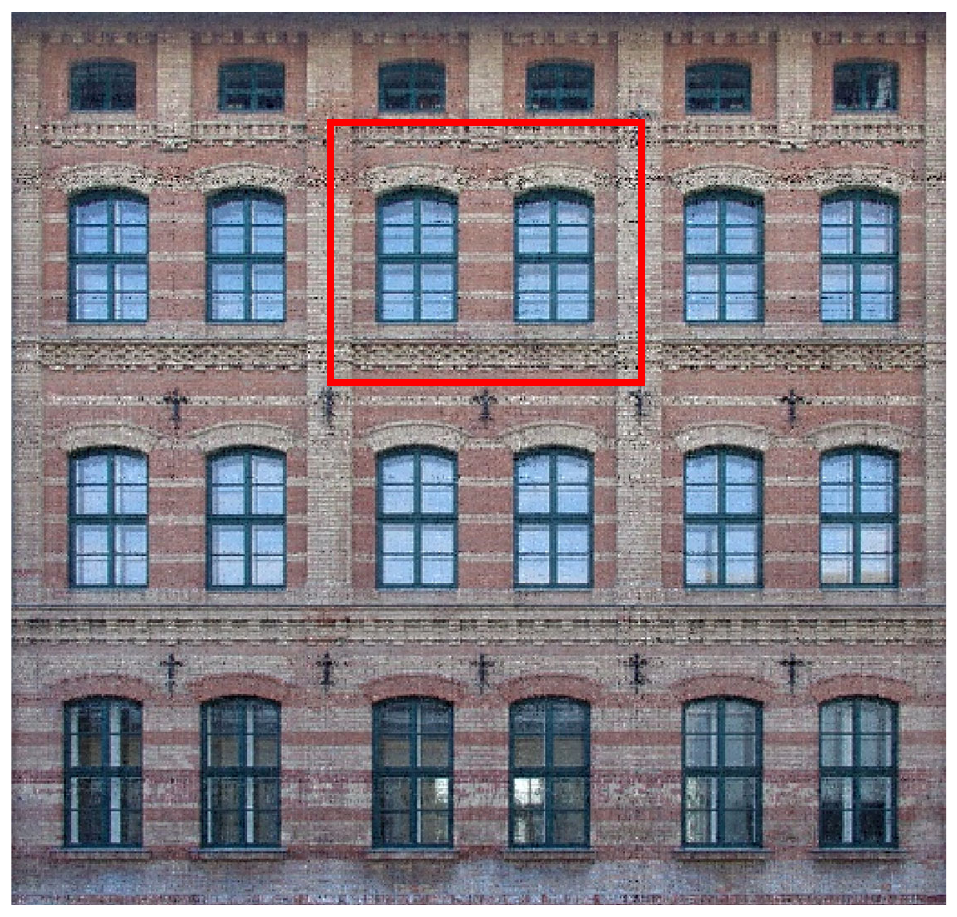}
\end{minipage}}
\subfigure[StoIHT $22.8890/0.8247$]{
\begin{minipage}[b]{0.25\textwidth}
\label{fig3: parameters-i} \includegraphics[width=1.1\textwidth]{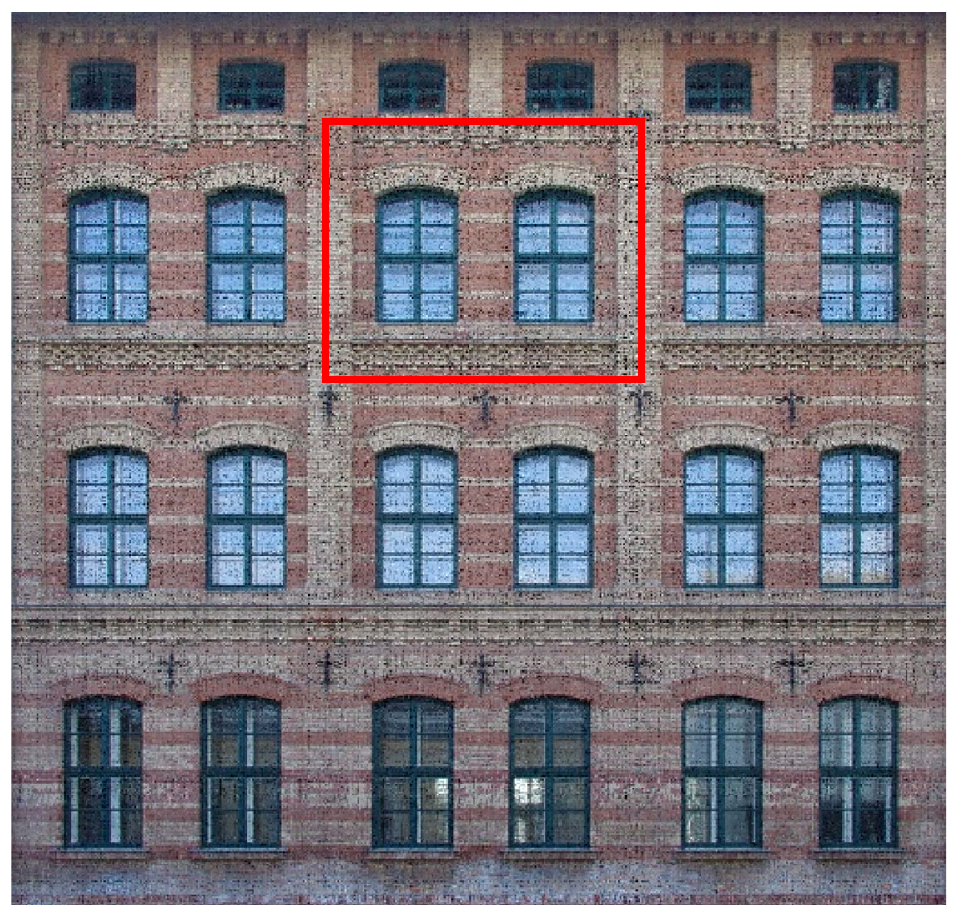}
\end{minipage}}
\subfigure[SVRG-ARM $\bf{25.1282/0.8953}$]{
\begin{minipage}[b]{0.25\textwidth}
\label{fig3: parameters-j} \includegraphics[width=1.1\textwidth]{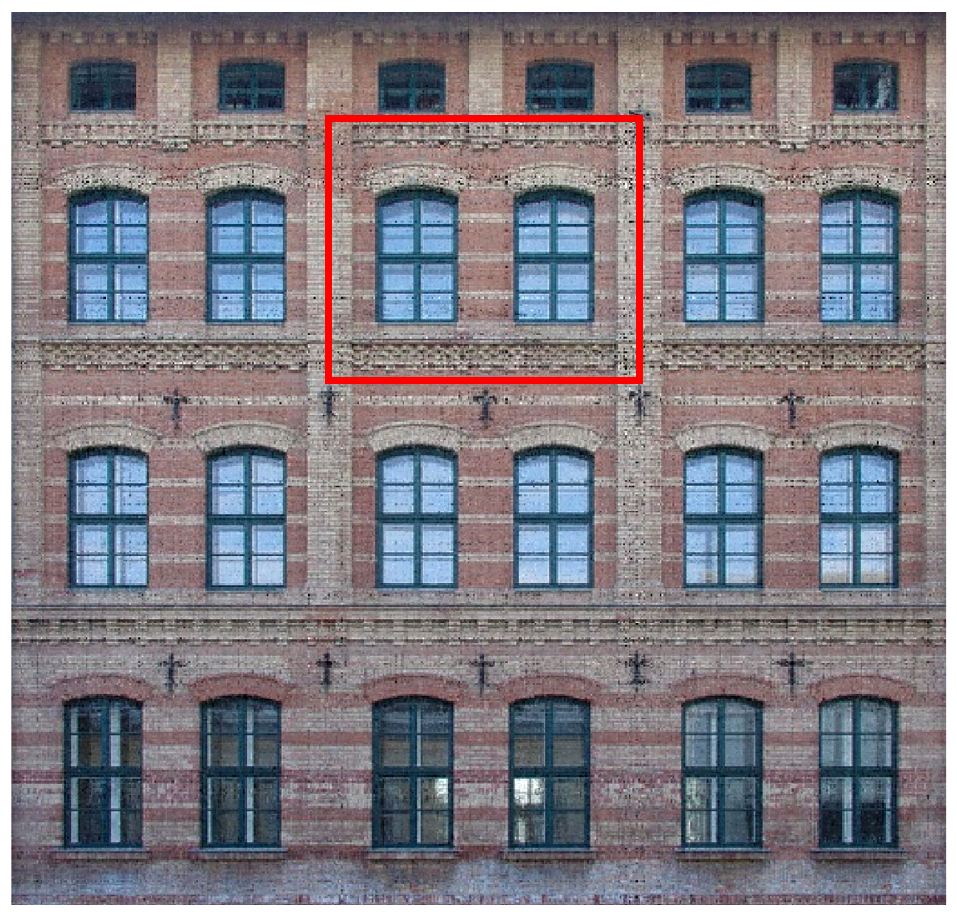}
\end{minipage}}
\quad
\caption{\small Comparison of matrix completion algorithms for image inpainting.(a) Original image. (b) Observed image with missing pixels. (c)-(h) Recovered images by SVT, SVP, NIHT, CGIHT, StoIHT, SVRG-ARM.}
\label{figure5}
\end{figure}

\subsection{Performance comparison within the class of gradient descent algorithms}
In this subsection, we conduct comparisons about matrix completion. The matrix completion as a classical affine rank minimization problem aims to recover a low-rank matrix from partially observed entries. We generate $n_{1}\times n_{2}$ matrices of rank $r$ as a product of a  $n_{1}\times r$ matrix and a $r\times n_{2}$ matrix, whose entries follow the Gaussian distributions. The locations of observed indices are sampled uniformly at random. Let $\rho$ be the sample ratio of observed entries over $n_{1}\times n_{2}$.

The first performance metric refers to the frequency rate of exact recovery. An exact recovery is recorded whenever $\|\widehat{X}-X\|_{F}/\|X\|_{F}\leq10^{-3}$, where $\widehat{X}$ denotes the estimate of original low-rank matrix $X$. We fix the matrix size to be $n_{1}=n_{2}=50$, set the sample ratio $\rho$ to be $0.5$ and vary rank $r$ to investigate the probability of recovery success. Each algorithm is tested for $100$ (random) trials for every rank $r$. Figure \ref{fig1: parameters-a} shows the frequency of exact recovery as a function of the rank. First, the recovery ability can be reflected by critical sparsity. The critical sparsity is the maximal sparsity level of the desired signal at which the exact recovery is ensured. Indeed, higher critical sparsity represents better empirical recovery performance. Figure \ref{fig1: parameters-a} reveals that the critical sparsity of SVRG-ARM is larger than that of other methods. The second metric is the convergence speed. In this experiment, the parametric setting is $n_{1}=n_{2}=50$, $\rho=0.5$, $r=4$. As shown in Figure \ref{fig: parameters-b}, except for SVT, the recovery accuracies of all other algorithms are almost identical. The convergence of SVRG-ARM is faster than other methods to reach the same optimality. These experiments suggest that SVRG-ARM outperforms StoIHT in both frequency of exact recovery and running time. It demonstrates the theoretical findings about variance reduced gradient, namely the conclusion that SVRG can reduce the variance introduced by stochastic gradient descent and accelerate the rate of convergence. To test the robustness to noise, we add the Gaussian noise with zero mean and standard deviation varying from $0$ to $0.4$ to low-rank matrix. Relative errors of all algorithms versus noise level are shown in Figure \ref{fig: parameters-c}. As shown, SVP, NIHT, CGIHT, StoIHT and SVRG-ARM are in the same level and SVRG-ARM slightly outperforms other methods.
\begin{figure}[H]
\centering
\subfigure[Frequency of exact recovery as a function of rank]{
\begin{minipage}[b]{0.321\textwidth}
\label{fig5: parameters-a}\includegraphics[width=1.1\textwidth]{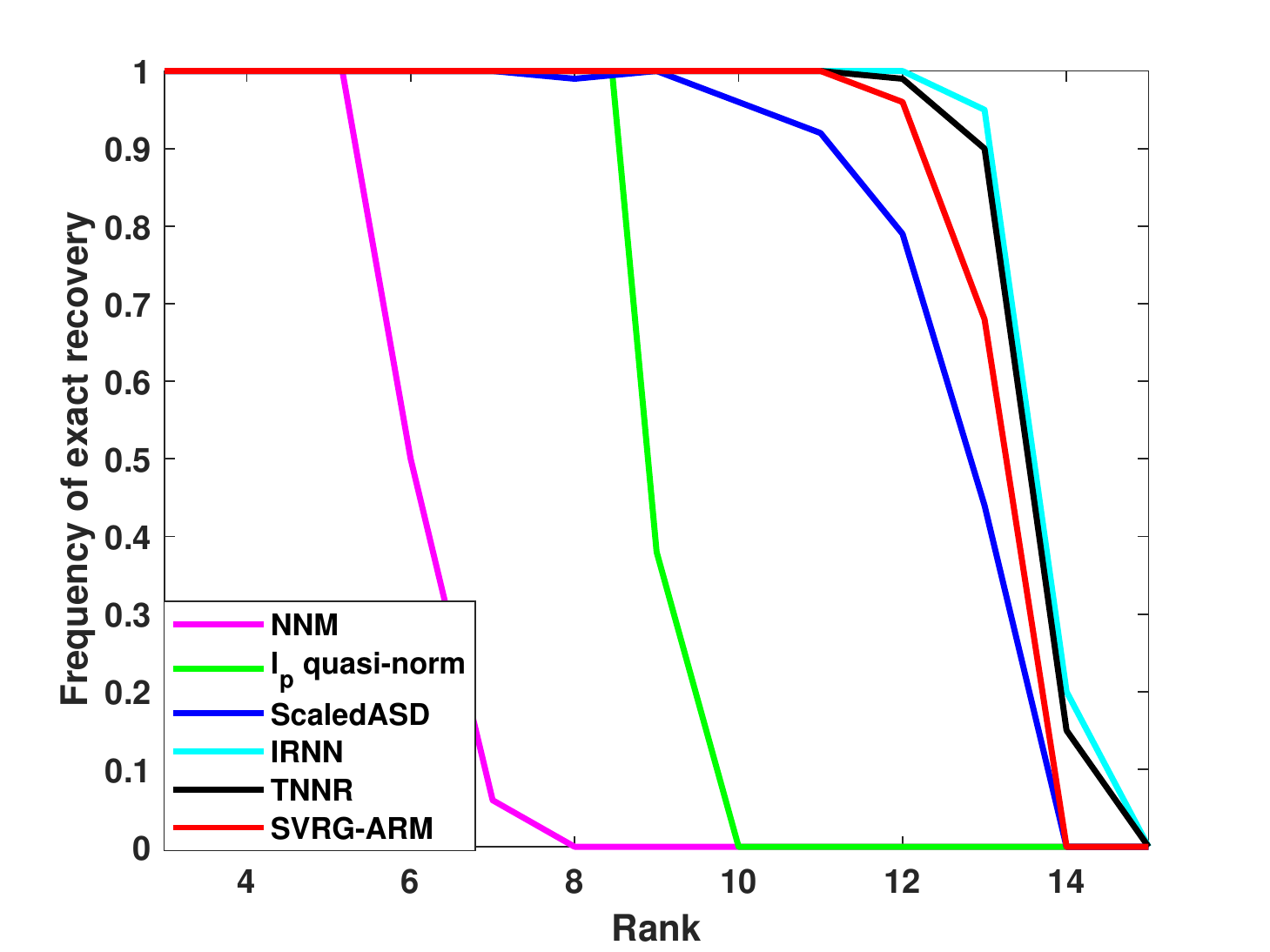}
\end{minipage}}
\subfigure[Convergence speed]{
\begin{minipage}[b]{0.321\textwidth}
\label{fig5: parameters-b} \includegraphics[width=1.1\textwidth]{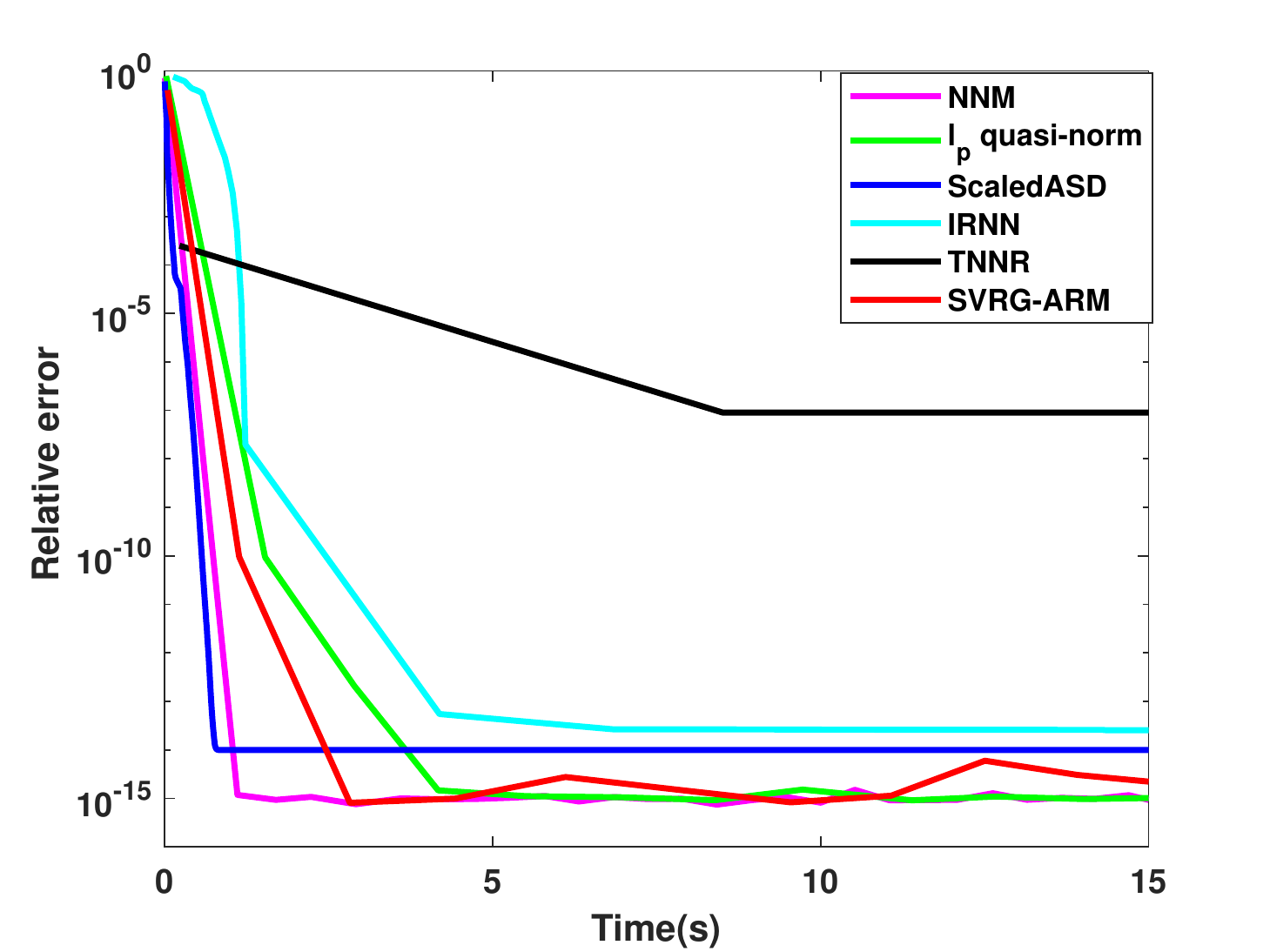}
\end{minipage}}
\subfigure[Normalized mean square
error as a function of noise level]{
\begin{minipage}[b]{0.321\textwidth}
\label{fig5: parameters-c} \includegraphics[width=1.1\textwidth]{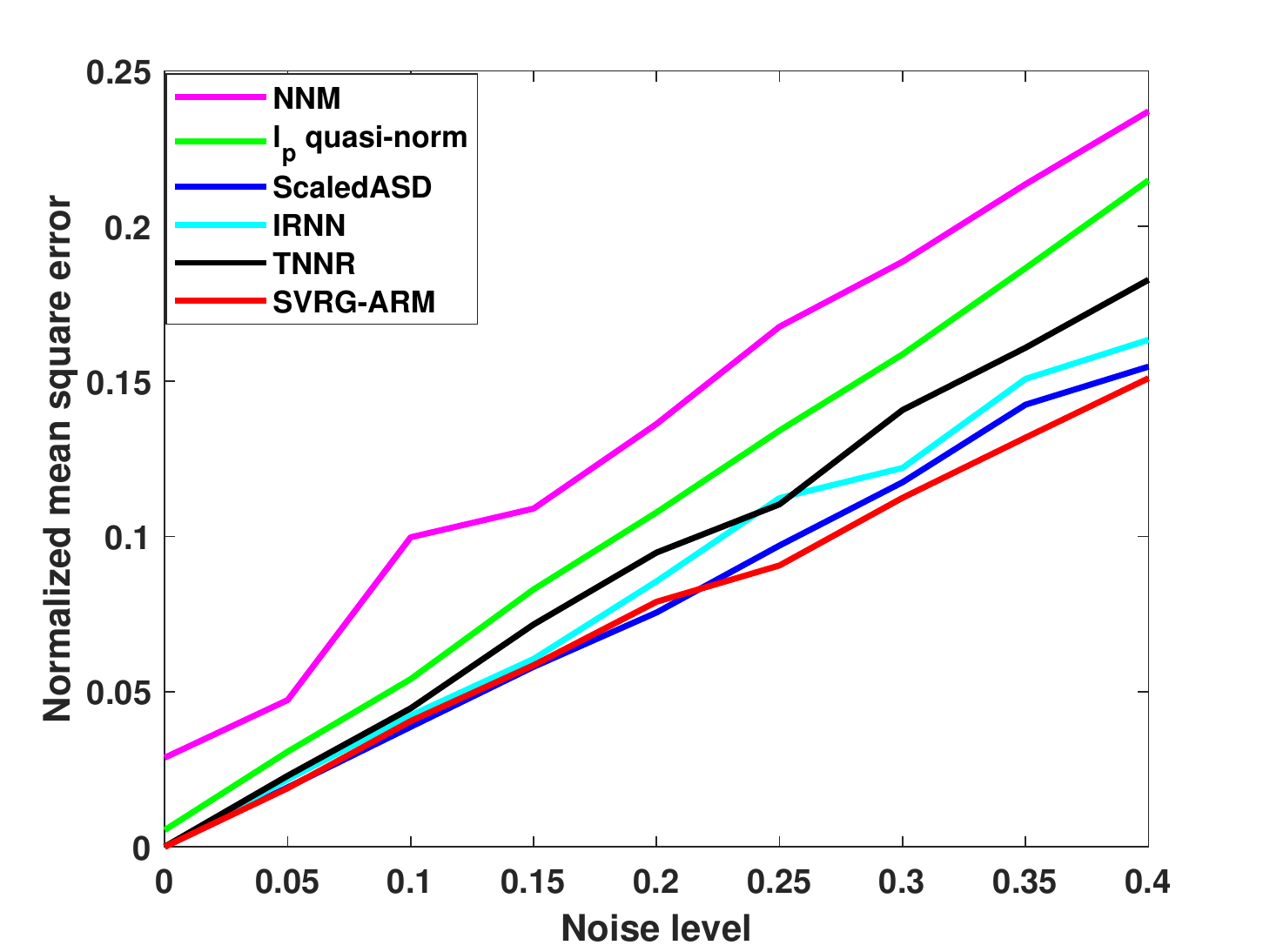}
\end{minipage}}
\quad
\caption{\small (a) Frequency of exact recovery as a function of rank. (b) Convergence speed. (c) Normalized mean square
error as a function of noise level}
\label{fig5}
\end{figure}
To further validate the effectiveness of SVRG-ARM, we check the recovery ability as a function of rank $r$ and proportion of sample ratio $\rho$. We fix the matrix size to be $n_{1}=n_{2}=50$ and vary rank $r$ and sample ratio $\rho$ to investigate the probability of recovery success. For each pair $(r,\rho)$, we simulate $100$ test instances. Figure \ref{figure2} shows the fraction of perfect recovery for each pair (black = $0$ and white =$1$). As known, the smaller the percentage of missing valu1es and the smaller the rank, the larger the region of correct recovery is. It is clear that the performance of our method SVRG-ARM is better than that of other methods.

We then present color image completion results. The size of the first image {\em pepper} is $512\times512$, the set of observed entries are generated randomly and the percentage of observed entries is $0.5$. The comparison is to apply matrix completion method to the luminance channel. Both the peak signal-to-noise ratio (PSNR) and structural similarity index (SSIM) are provided for the comparison. We find PSNR and SSIM by the proposed SVRG-ARM algorithm is better than that of other methods. As shown in Figure \ref{fig4}, in the rectangle region, it can be seen that SVRG-ARM generates high-level visual quality with sharper edges and richer textures in comparison with other methods. The size of the second image {\em facade} is $517\times493$, the set of observed entries are generated randomly and the percentage of observed entries is $0.4$. The quantitative comparisons show that SVRG-ARM can provide larger PSNR and SSIM values than those by other methods. In addition, Figure \ref{figure5} shows that the reconstructed image by SVRG-ARM has higher quality edges with proper sharpness and limited artifacts. The experimental results verify that SVRG-ARM outperforms other gradient descent algorithms in terms of both synthetic and real data.
\begin{figure}[H]
\centering
\subfigure[NNM]{
\begin{minipage}[b]{0.321\textwidth}
\label{fig6: parameters-a}\includegraphics[width=1.1\textwidth]{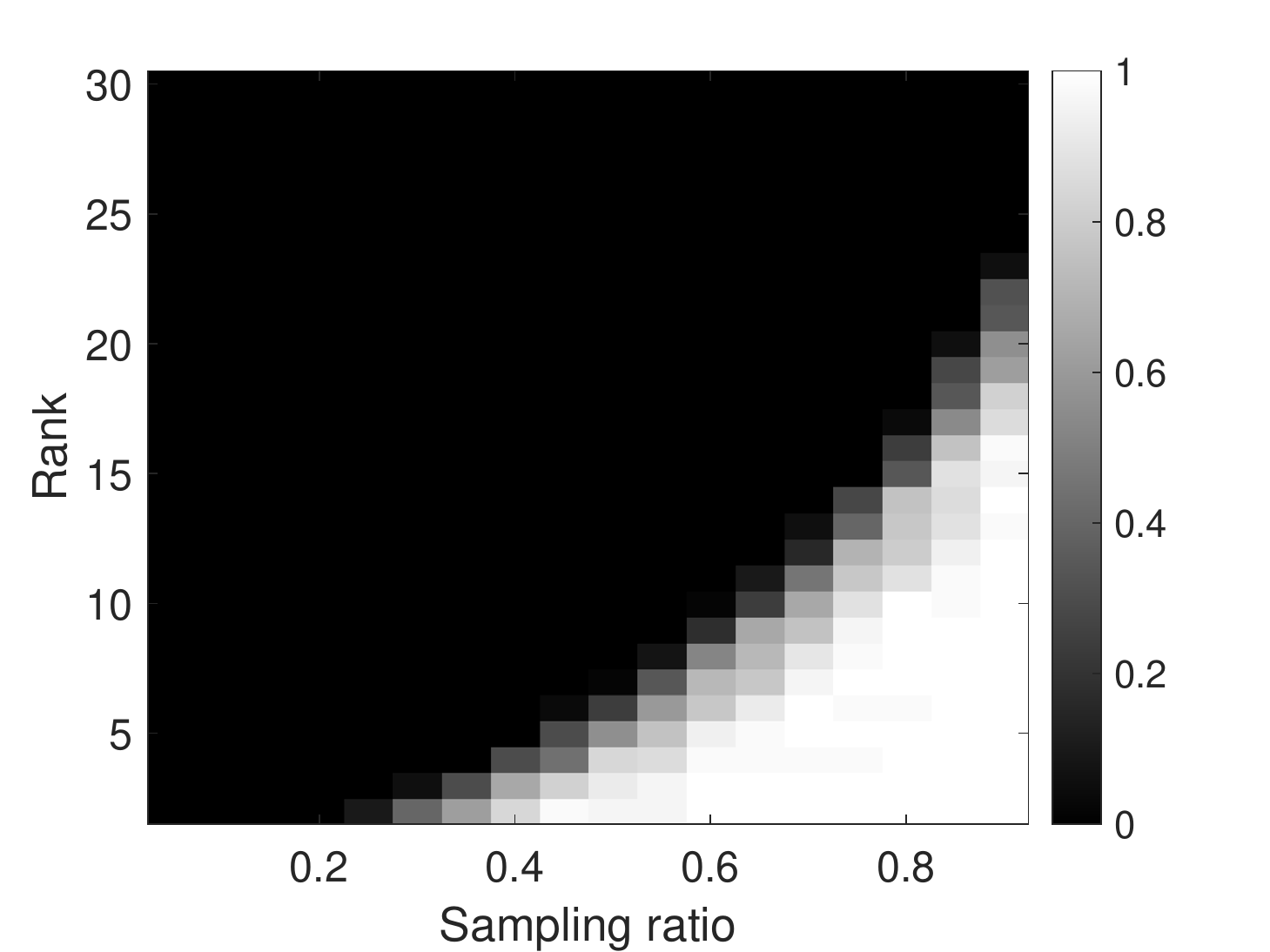}
\end{minipage}}
\subfigure[$\ell_{p}$ quasi-norm]{
\begin{minipage}[b]{0.321\textwidth}
\label{fig6: parameters-b} \includegraphics[width=1.1\textwidth]{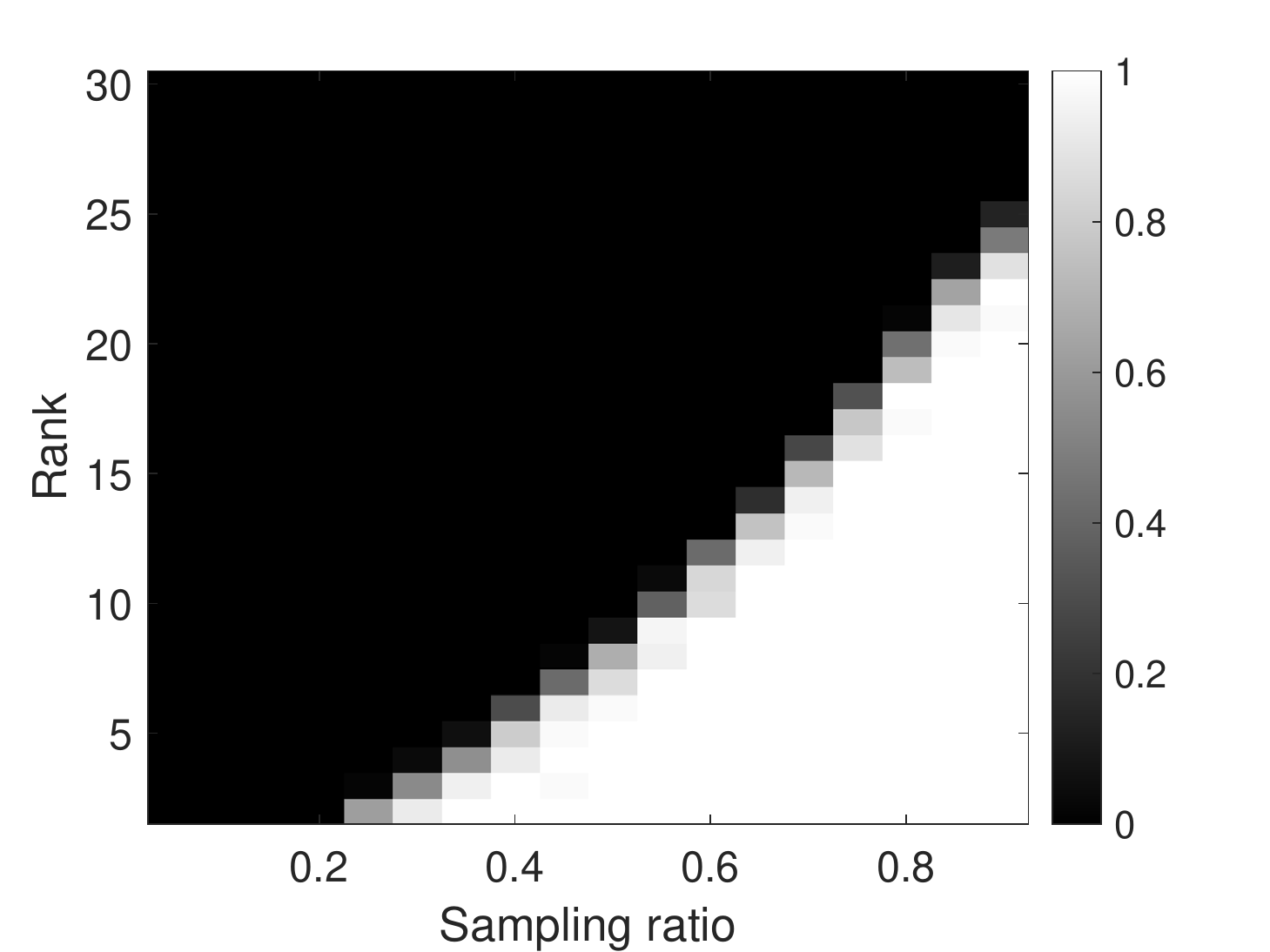}
\end{minipage}}
\subfigure[ScaledASD]{
\begin{minipage}[b]{0.321\textwidth}
\label{fig6: parameters-c} \includegraphics[width=1.1\textwidth]{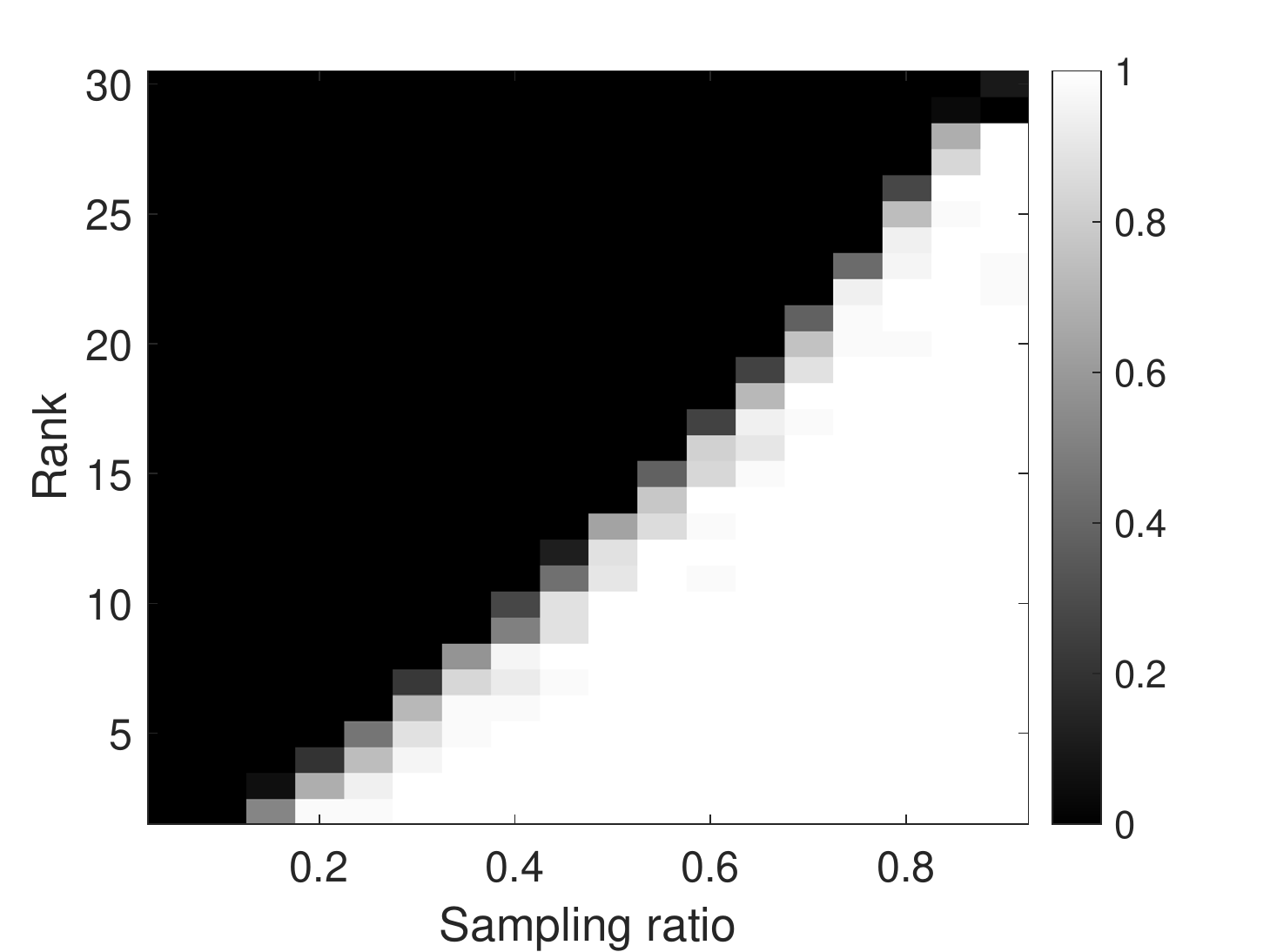}
\end{minipage}}
\subfigure[IRNN]{
\begin{minipage}[b]{0.321\textwidth}
\label{fig6: parameters-d}\includegraphics[width=1.1\textwidth]{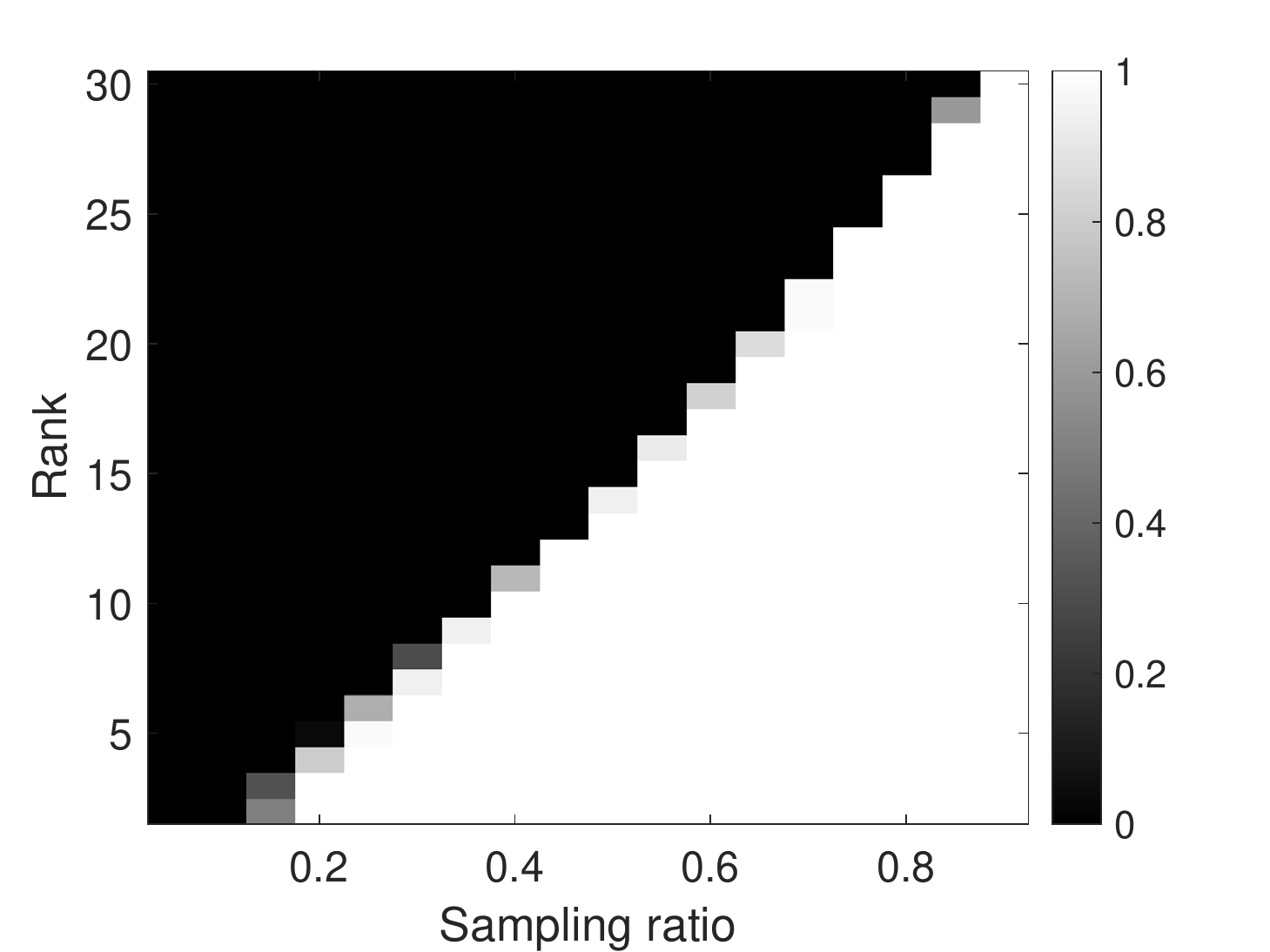}
\end{minipage}}
\subfigure[TNNR]{
\begin{minipage}[b]{0.321\textwidth}
\label{fig6: parameters-e} \includegraphics[width=1.1\textwidth]{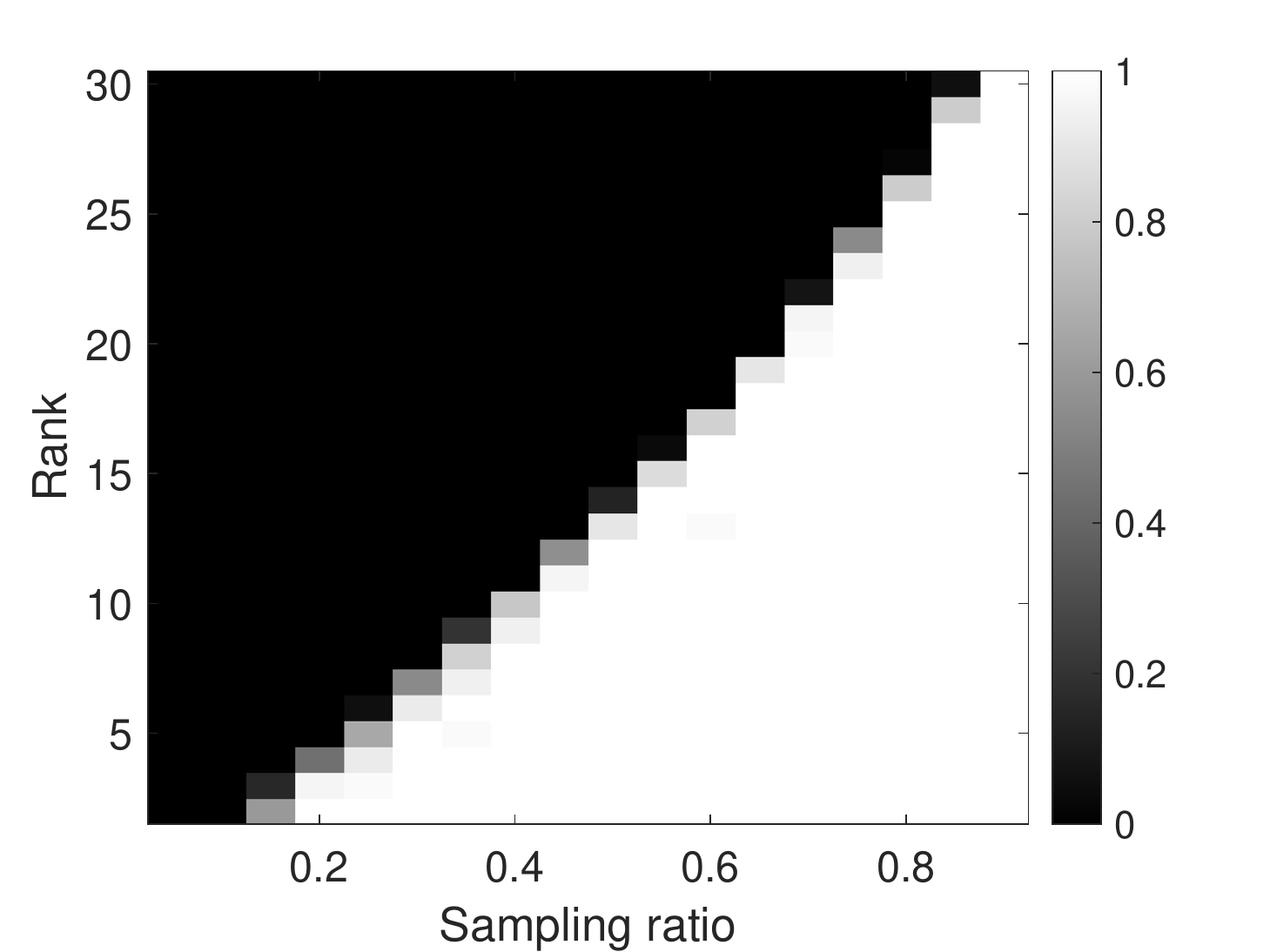}
\end{minipage}}
\subfigure[SVRG-ARM]{
\begin{minipage}[b]{0.321\textwidth}
\label{fig6: parameters-f} \includegraphics[width=1.1\textwidth]{phase_transition_SVRG.pdf}
\end{minipage}}
\quad
\caption{\small Phase transition of low-rank matrix completion using (a) NNM. (b) $\ell_{p}$ quasi-norm. (c) ScaledASD. (d) IRNN. (c) TNNR. (d) SVRG-ARM.}
\label{figure6}
\end{figure}

\subsection{Overall comparison with state-of-the-art algorithms}

Presented here are comparisons among SVRG-ARM and state-of-the-art techniques such as NNM, $\ell_{p}$ quasi-norm, ScaledASD, IRNN, and TNNR in terms of frequency of exact recovery, convergence speed and robustness. In the first experiment, rank varies from $3$ to $15$ with matrix size of $n_{1}=n_{2}=50$ and sample ratio $\rho=0.5$. As shown in Figure \ref{fig5: parameters-a}, IRNN and TNNR present better recovery performances than SVRG-ARM. In the following experiments, we set $n_{1}=n_{2}=50$, $\rho=0.5$ and $r=8$. For comparison, the particular rank 
selection $r=5$ is used for NNM. For the execution-time comparison, ScaledASD achieves the best convergence speed, and however SVRG-ARM presents a better recovery accuracy than ScaledASD. To test the robustness to noise, we add the Gaussian noise with zero mean and standard deviation varying from $0$ to $0.4$ to low-rank matrix. Relative errors of all algorithms versus noise level are shown in Figure \ref{fig5: parameters-c}. As shown, SVRG-ARM is more robust than other algorithms. Experimental results suggest that no algorithm is consistently superior for all cases. But SVRG-ARM is observed to have obviously advantageous balance of efficiency, accuracy and robustness compared with other algorithms.

We then compare phase transitions of low-rank matrix completion using different methods, where the recovery ability as a function of rank $r$ and proportion of sample ratio $\rho$ is investigated. Successful recovery is indicated by white and failure by black. Results are averaged over $100$ independent trials. Figure \ref{figure6} shows that SVRG-ARM still delivers reasonable performance better than that of NNM, $\ell_{p}$ quasi-norm, ScaledASD, though slightly underperforms that of IRNN and TNNR.

Finally, we conduct image completions to compare different methods. The size of the first image {\em flower} is $512\times480$, the set of observed entries are generated randomly and the percentage of observed entries is $0.5$. Two common image quality evaluation criteria PSNR and SSIM are still employed to reflect the image recovery quality. Figure \ref{figure7} shows that our algorithm achieves the highest PSNR and SSIM among all methods. From the rectangle region, it can also be observed that SVRG-ARM provides high-level visual quality with sharper edges and richer textures. To further illustrate the effectiveness of the proposed method, we show the reconstructed results of image {\em baboon} by different methods in Figure \ref{figure8}. In this experiment, the size of the image is $512\times512$, the set of observed entries are generated randomly and the percentage of observed entries is $0.4$. This experiment manifests that SVRG-ARM can obtain good results especially referring to edges (high frequency details). Numerical results about image completions demonstrate the effectiveness of SVRG-ARM among different low-rank matrix completion algorithms. 
\begin{figure}[H]
\centering
\subfigure[Original image]{
\begin{minipage}[b]{0.25\textwidth}
\label{fig7: parameters-a}\includegraphics[width=1.1\textwidth]{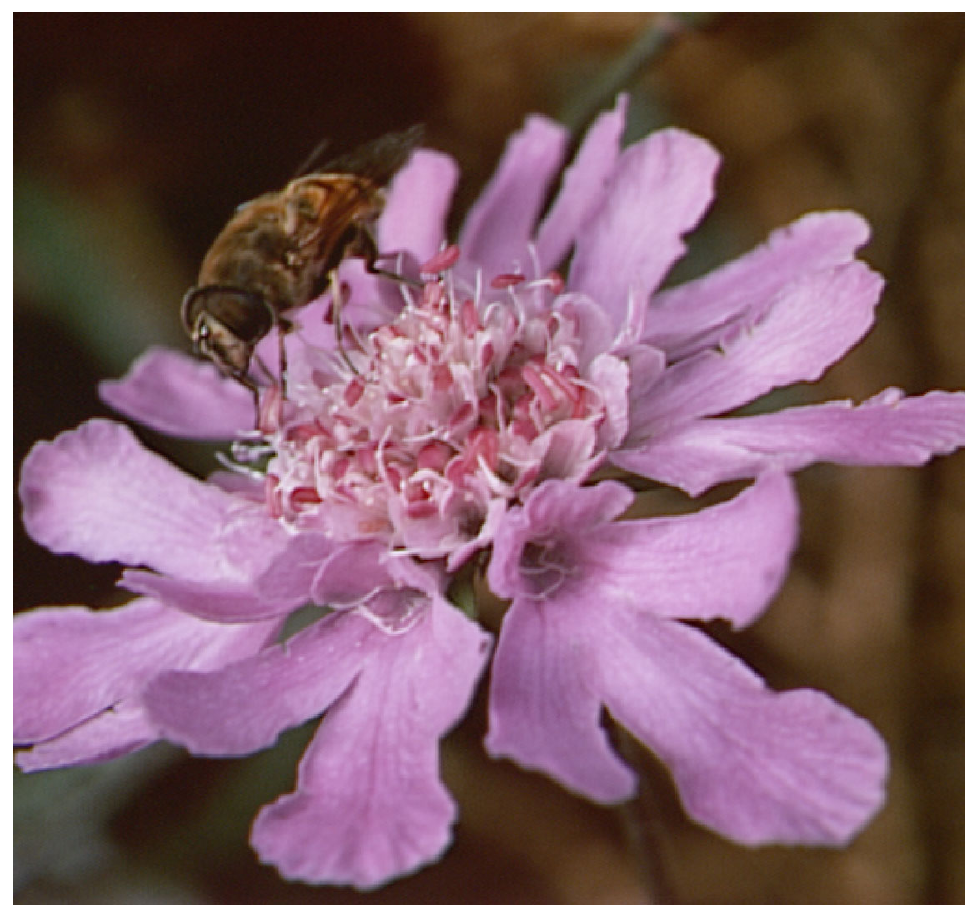}
\end{minipage}}
\subfigure[Observed image with missing pixels]{
\begin{minipage}[b]{0.25\textwidth}
\label{fig7: parameters-b} \includegraphics[width=1.1\textwidth]{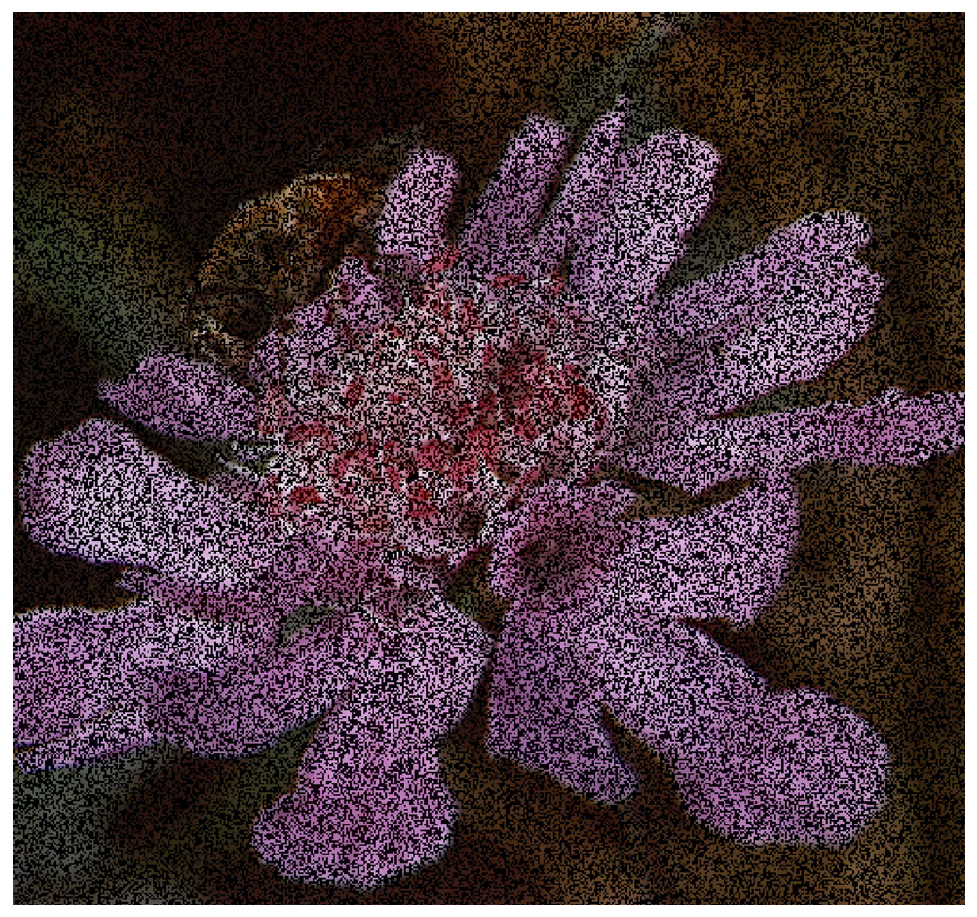}
\end{minipage}}
\subfigure[NNM $30.4058/ 0.9189$]{
\begin{minipage}[b]{0.25\textwidth}
\label{fig7: parameters-c} \includegraphics[width=1.1\textwidth]{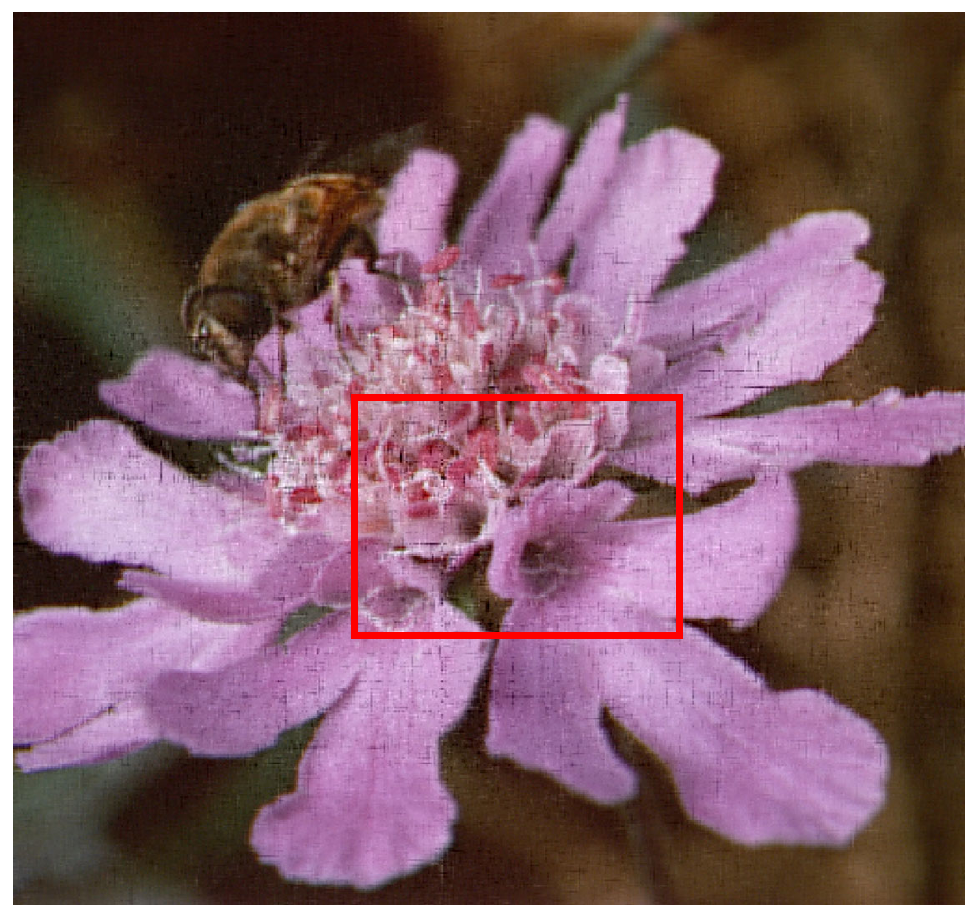}
\end{minipage}}
\subfigure[$\ell_{p}$ quasi-norm $30.5831/0.9225$]{
\begin{minipage}[b]{0.25\textwidth}
\label{fig7: parameters-e} \includegraphics[width=1.1\textwidth]{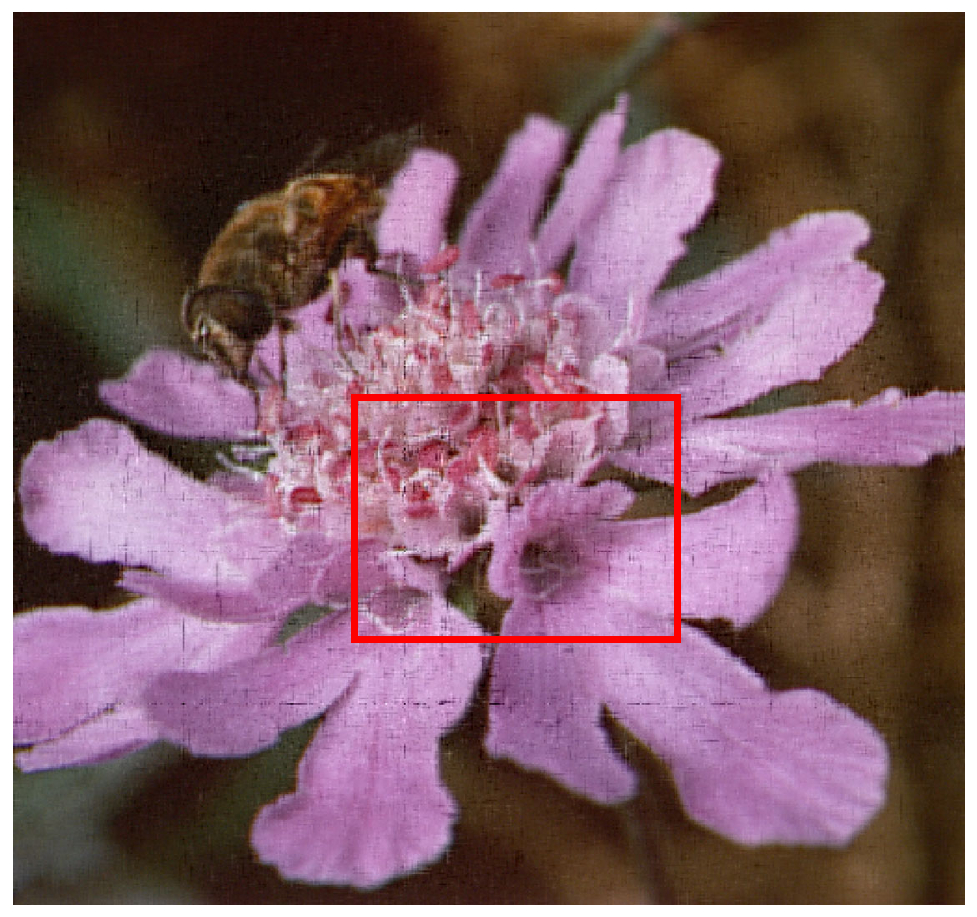}
\end{minipage}}
\subfigure[ScaledASD $33.5536/0.9586$]{
\begin{minipage}[b]{0.25\textwidth}
\label{fig7: parameters-f} \includegraphics[width=1.1\textwidth]{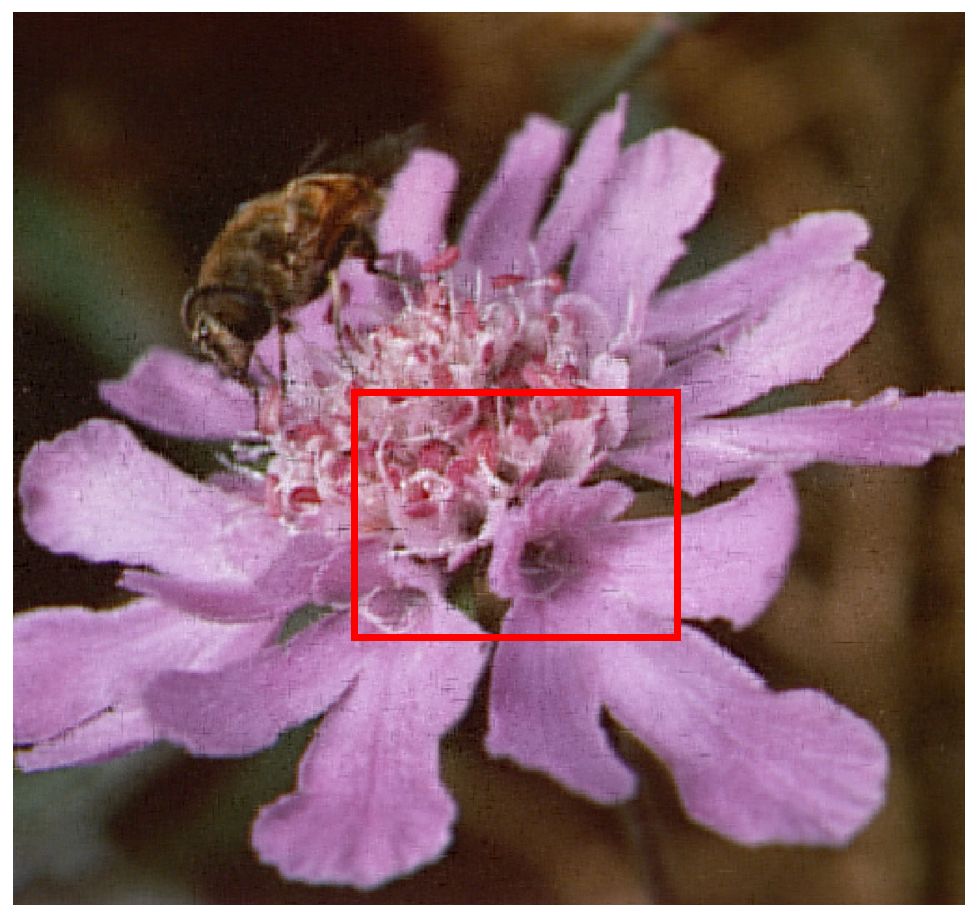}
\end{minipage}}
\subfigure[IRNN $33.0308/0.9539$]{
\begin{minipage}[b]{0.25\textwidth}
\label{fig7: parameters-h} \includegraphics[width=1.1\textwidth]{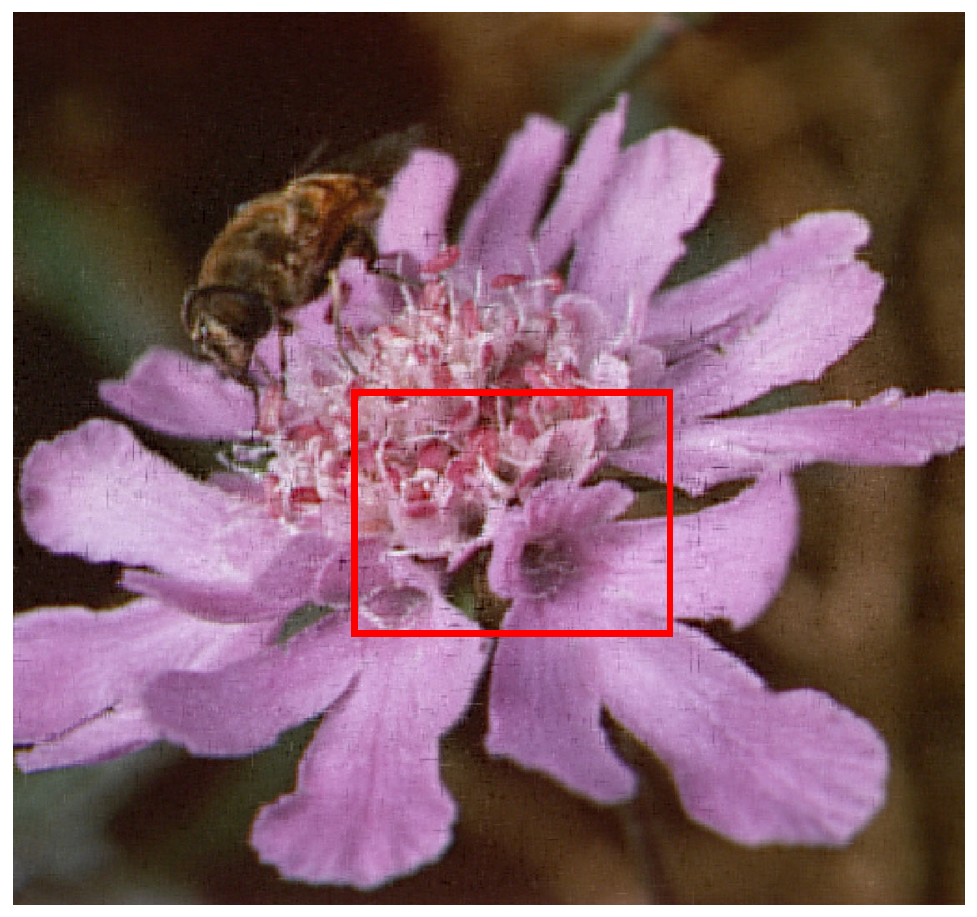}
\end{minipage}}
\subfigure[TNNR $34.4427/0.9649$]{
\begin{minipage}[b]{0.25\textwidth}
\label{fig7: parameters-i} \includegraphics[width=1.1\textwidth]{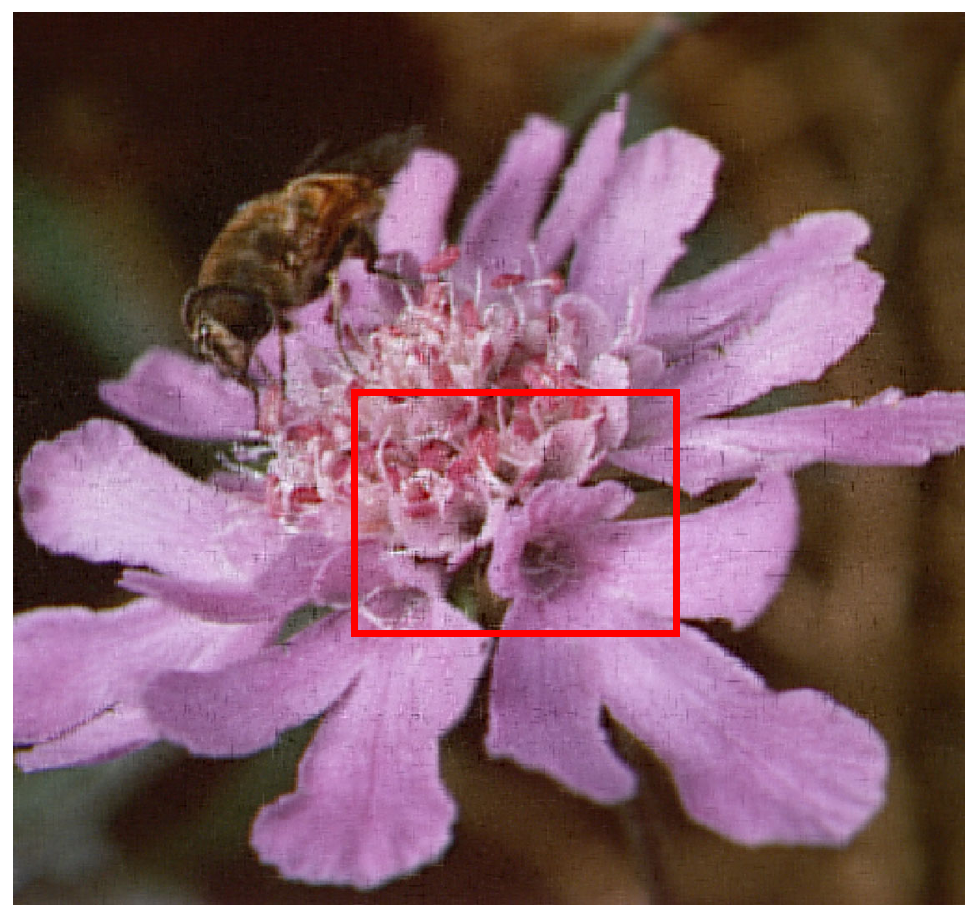}
\end{minipage}}
\subfigure[SVRG-ARM $\bf{34.5725/0.9665}$]{
\begin{minipage}[b]{0.25\textwidth}
\label{fig7: parameters-j} \includegraphics[width=1.1\textwidth]{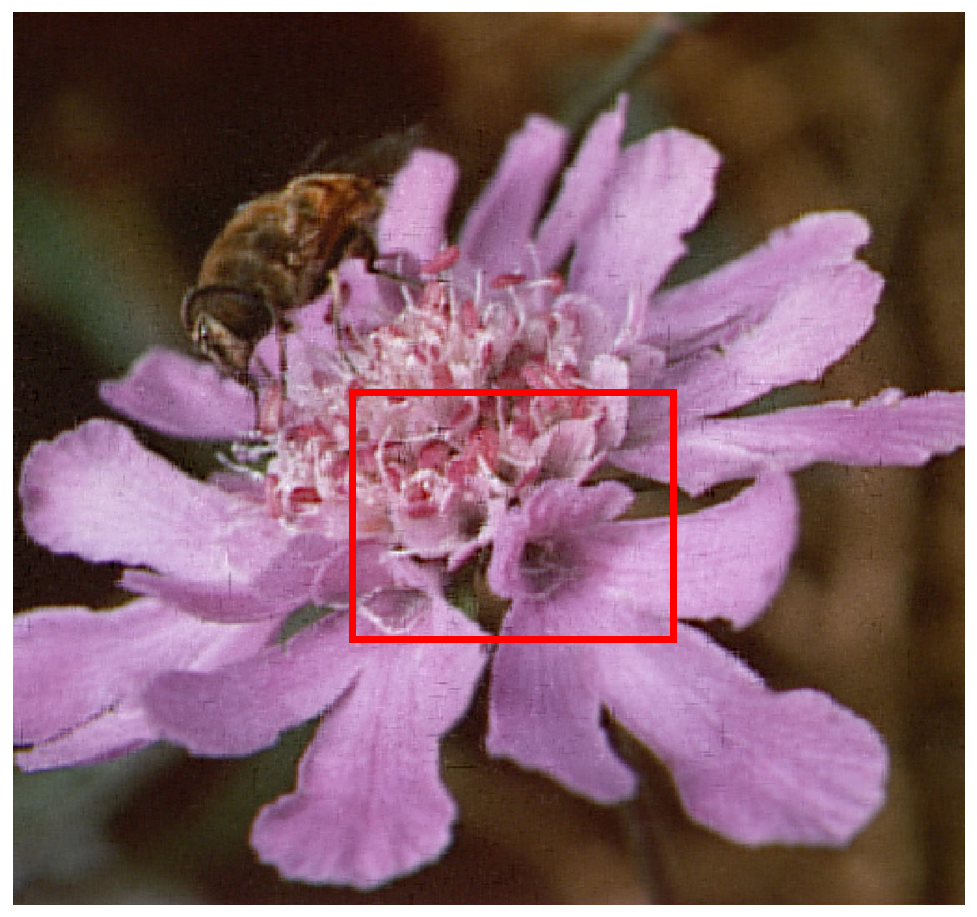}
\end{minipage}}
\quad
\caption{\small Comparison of matrix completion algorithms for image completion. (a) Original image. (b) Observed image with missing pixels. (c)-(h) Recovered images by NNM, $\ell_{p}$ quasi-norm, ScaledASD, IRNN, TNNR, SVRG-ARM.}
\label{figure7}
\end{figure}
\section{Conclusion}
We introduce a particularly simple yet highly efficient stochastic variance reduced gradient descent algorithm to solve the affine rank minimization problem consists of finding a matrix of minimum rank from linear measurements. We prove that the proposed algorithm converges linearly in expectation to the solution under a restricted isometry condition. It should be pointed out that the linear convergence condition is not necessarily optimal at present times, which can be relaxed with perhaps plenty of rooms to improve. The proposed algorithm is observed to have obviously advantageous balance of efficiency, adaptivity, and accuracy compared with other state-of-the-art greedy algorithms. A matlab implementation of the proposed algorithm is also available at \url{https://www.dropbox.com/s/9gte2as7gcarl80/SVRG-ARM.zip?dl=0}. 
\begin{figure}[H]
\centering
\subfigure[Original image]{
\begin{minipage}[b]{0.25\textwidth}
\label{fig8: parameters-a}\includegraphics[width=1.1\textwidth]{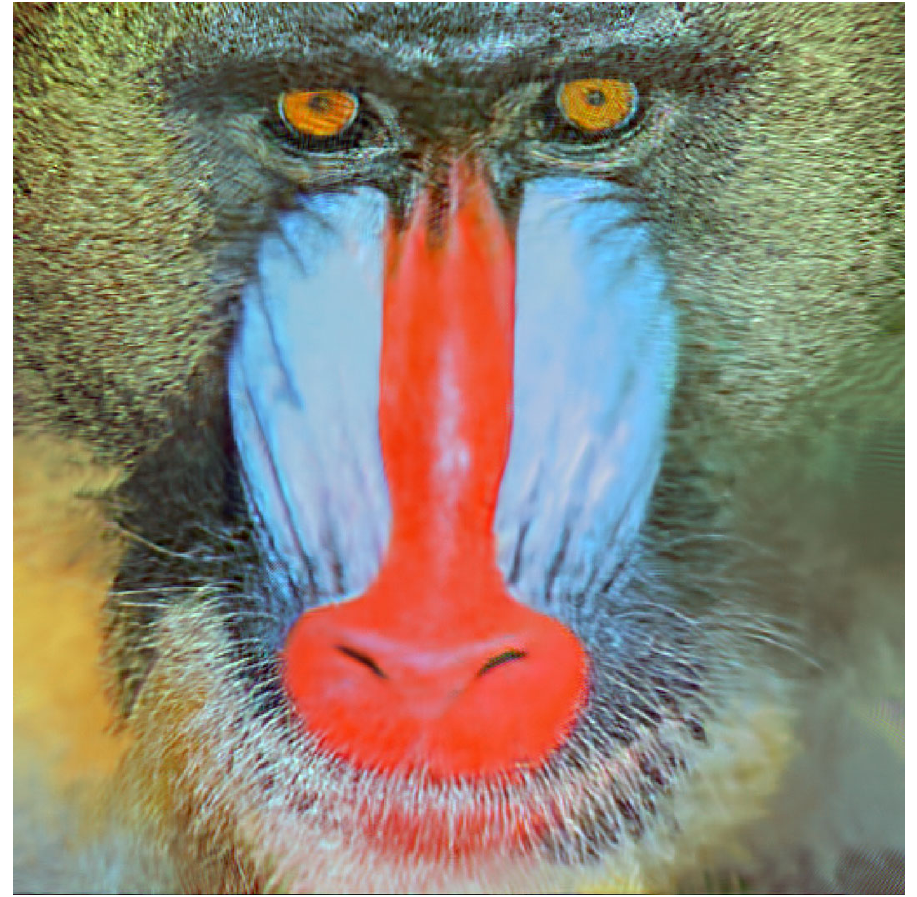}
\end{minipage}}
\subfigure[Observed image with missing pixels]{
\begin{minipage}[b]{0.25\textwidth}
\label{fig8: parameters-b} \includegraphics[width=1.1\textwidth]{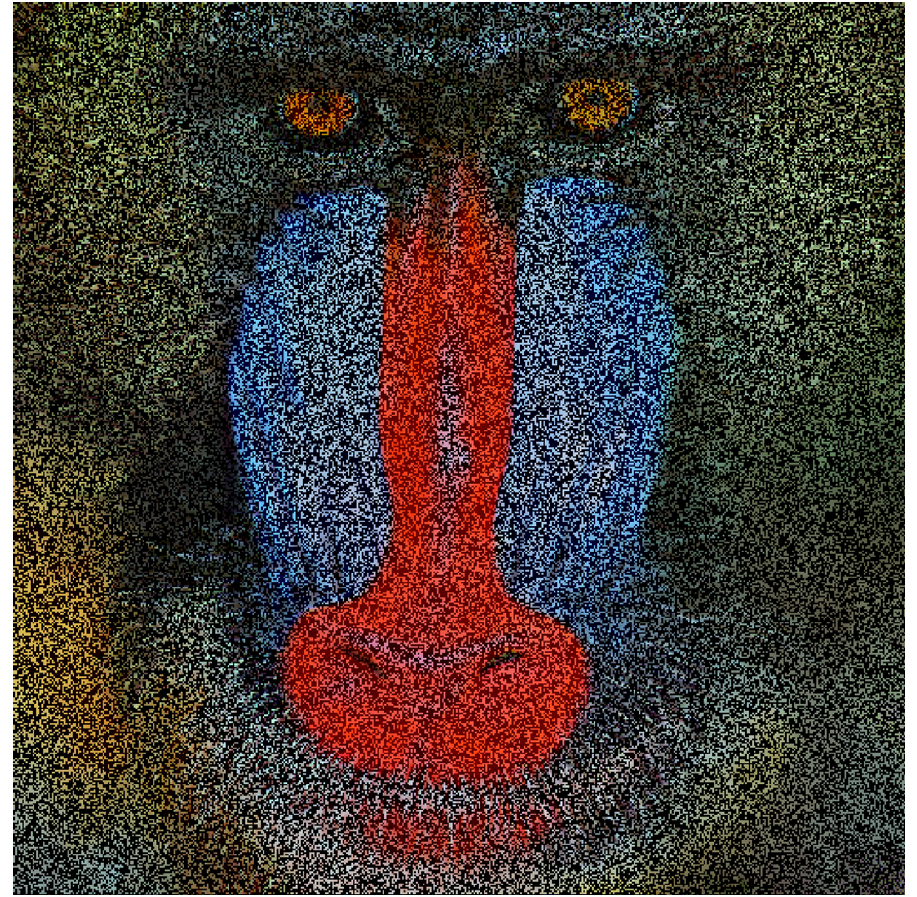}
\end{minipage}}
\subfigure[NNM $21.5856/0.6405$]{
\begin{minipage}[b]{0.25\textwidth}
\label{fig8: parameters-c} \includegraphics[width=1.1\textwidth]{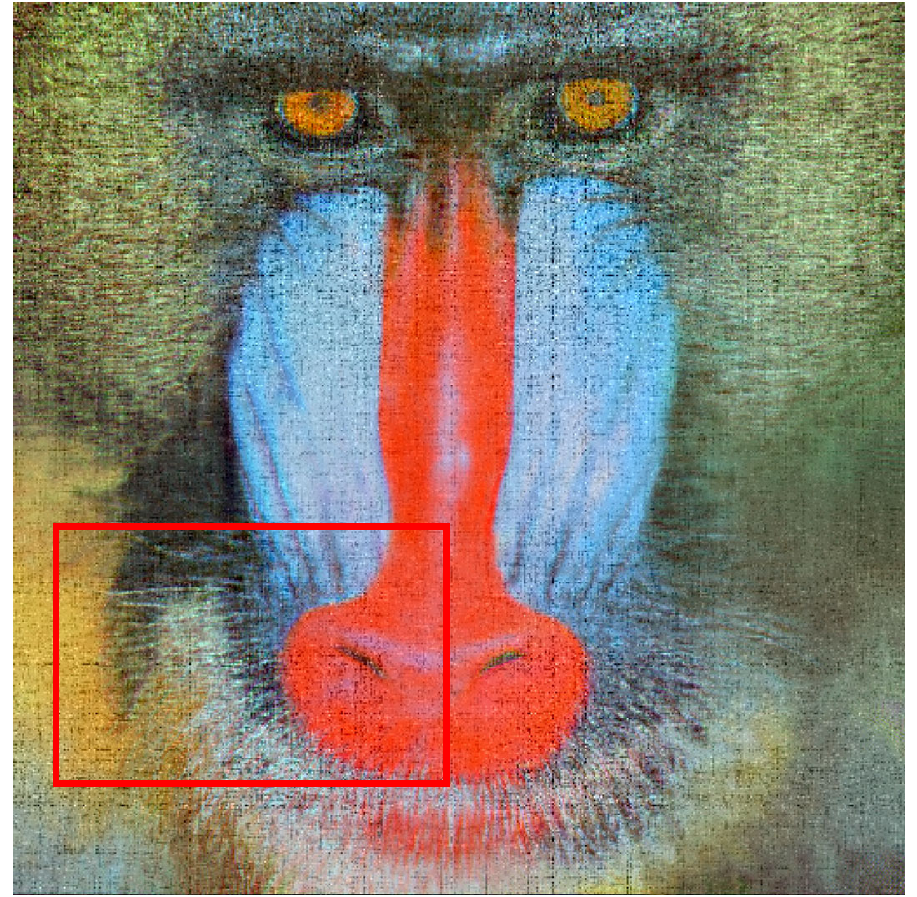}
\end{minipage}}
\subfigure[$\ell_{p}$ quasi-norm $22.4057/0.6831$]{
\begin{minipage}[b]{0.25\textwidth}
\label{fig8: parameters-e} \includegraphics[width=1.1\textwidth]{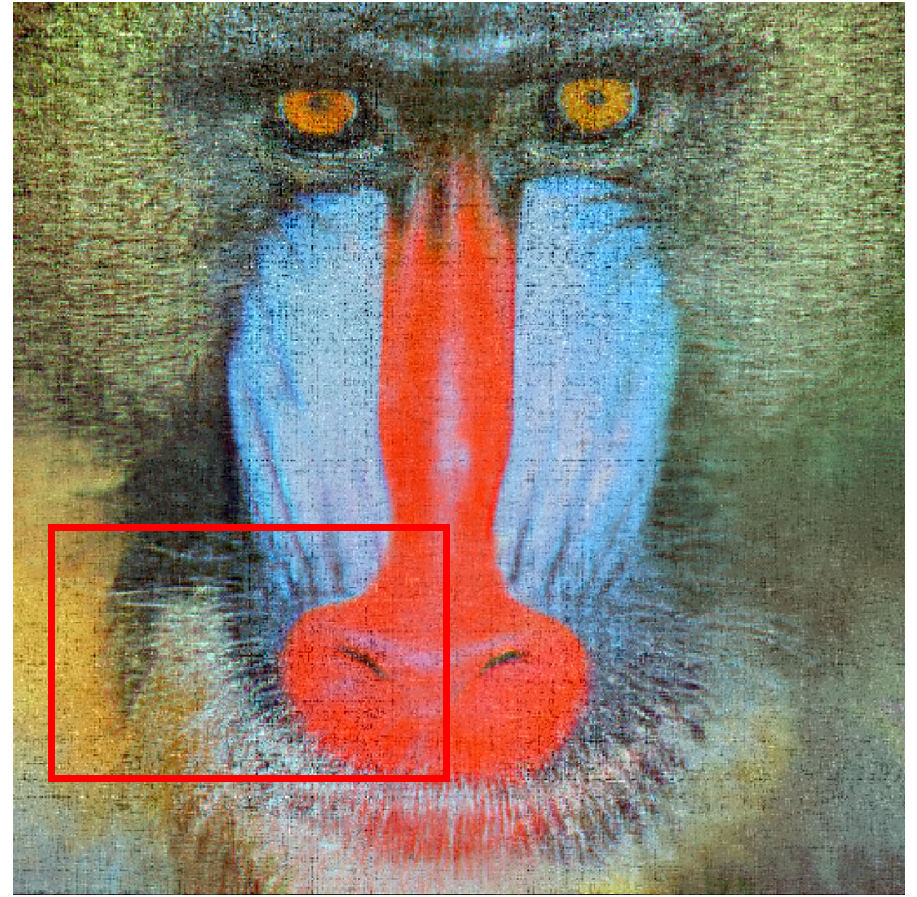}
\end{minipage}}
\subfigure[ScaledASD $23.4571/0.7465$]{
\begin{minipage}[b]{0.25\textwidth}
\label{fig8: parameters-f} \includegraphics[width=1.1\textwidth]{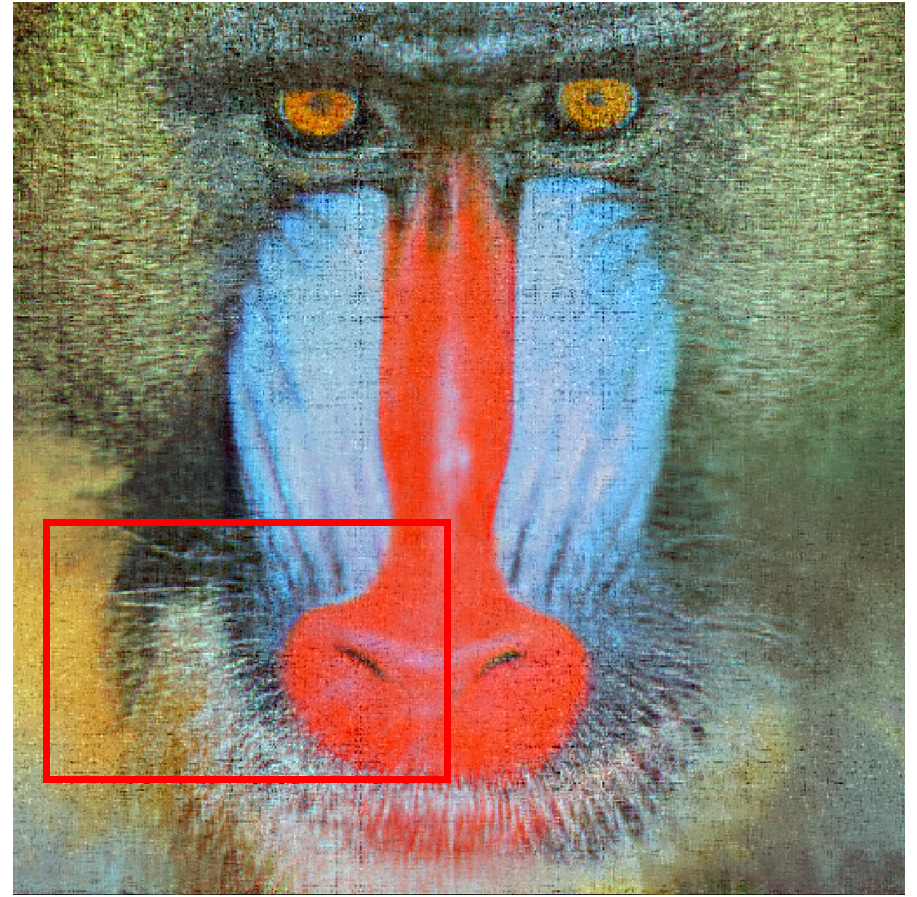}
\end{minipage}}
\subfigure[IRNN $23.7856/0.7602$]{
\begin{minipage}[b]{0.25\textwidth}
\label{fig8: parameters-h} \includegraphics[width=1.1\textwidth]{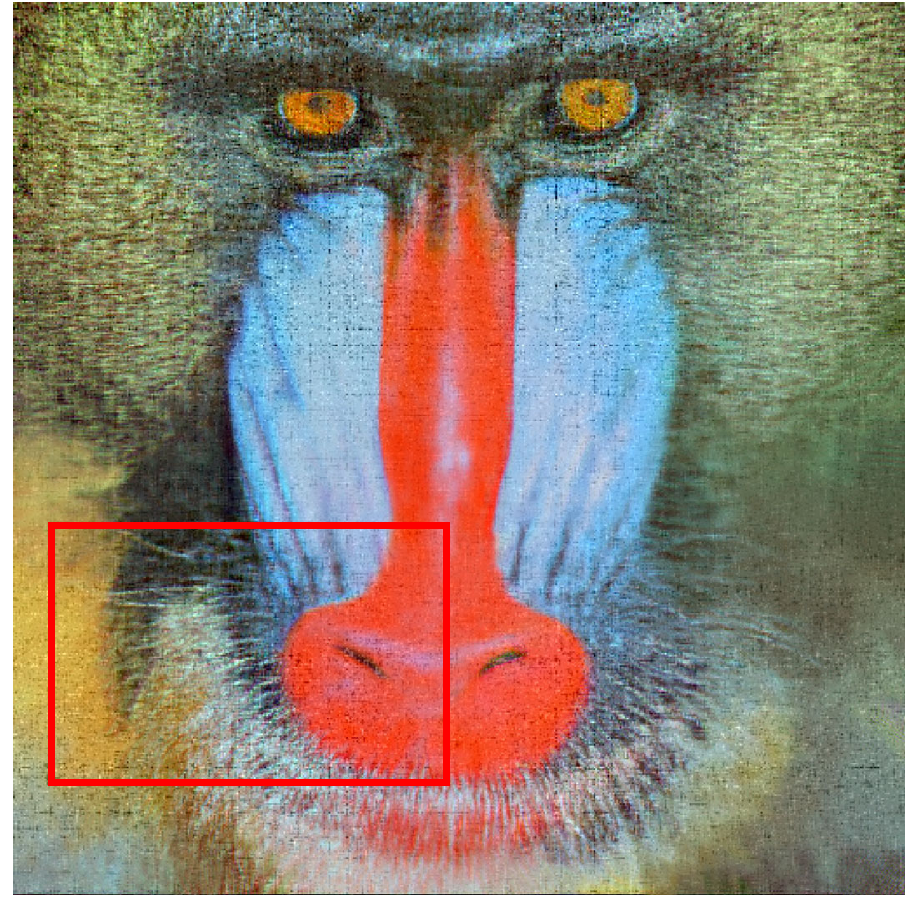}
\end{minipage}}
\subfigure[TNNR $23.8627/0.7638$]{
\begin{minipage}[b]{0.25\textwidth}
\label{fig8: parameters-i} \includegraphics[width=1.1\textwidth]{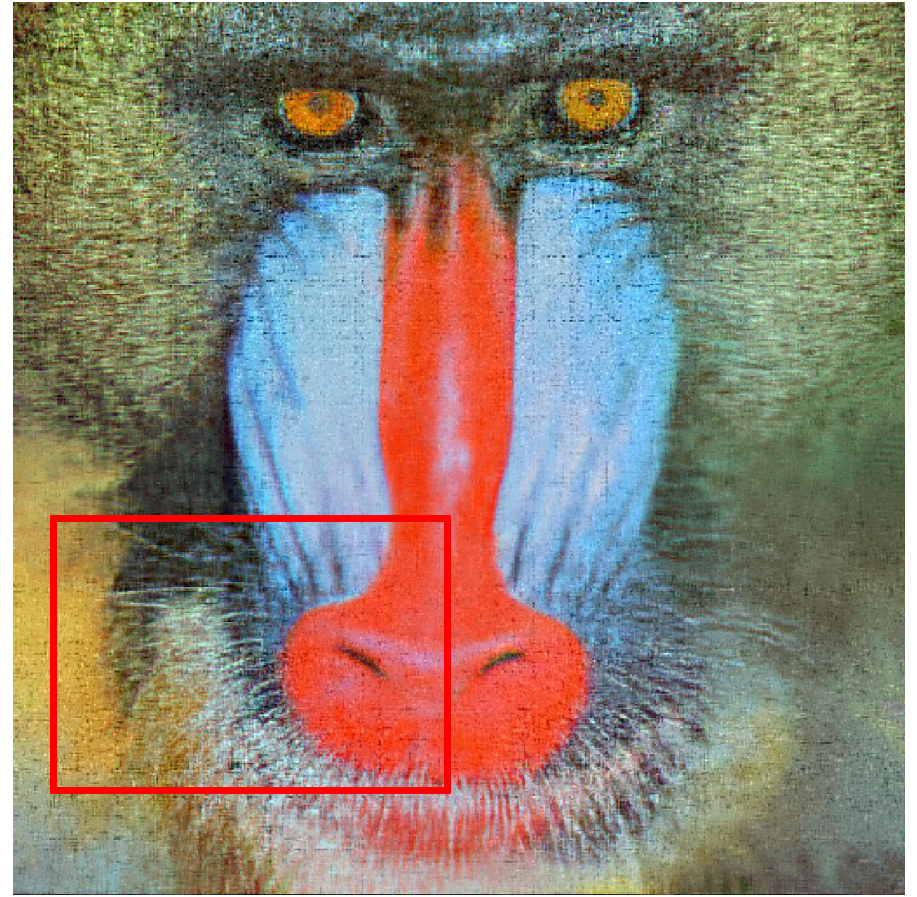}
\end{minipage}}
\subfigure[SVRG-ARM $\bf{24.2060/0.7796}$]{
\begin{minipage}[b]{0.25\textwidth}
\label{fig8: parameters-j} \includegraphics[width=1.1\textwidth]{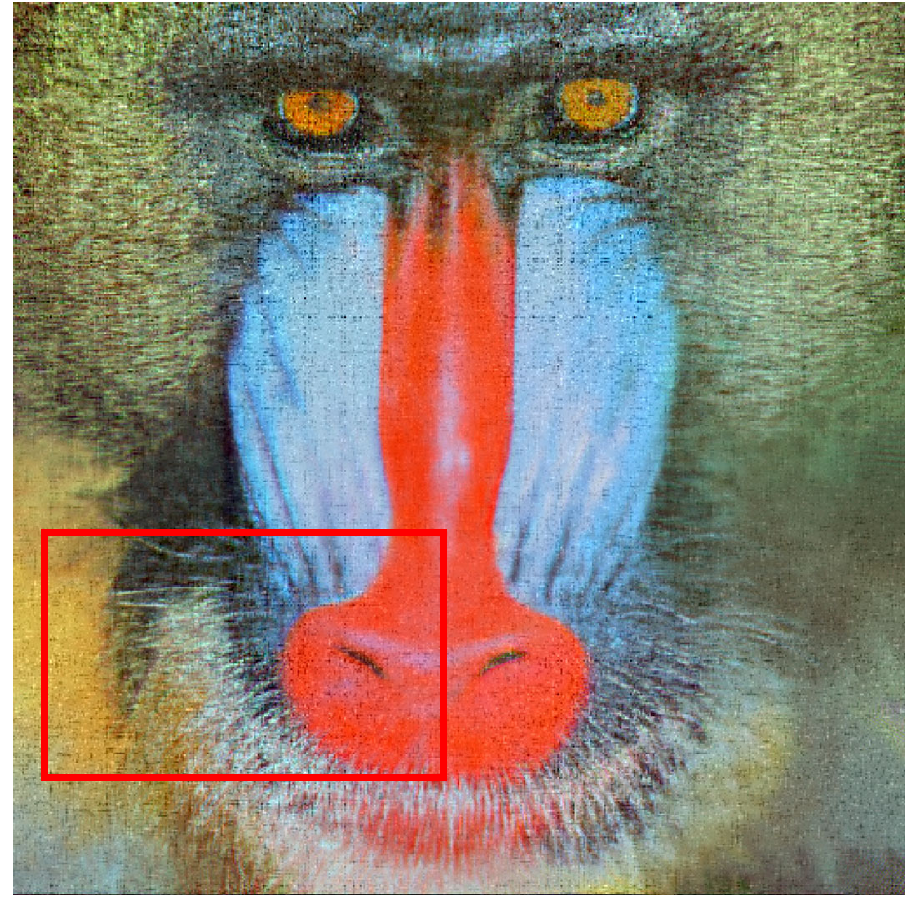}
\end{minipage}}
\quad
\caption{\small Comparison of matrix completion algorithms for image inpainting.(a) Original image. (b) Observed image with missing pixels. (c)-(h) Recovered images by NNM, $\ell_{p}$ quasi-norm, ScaledASD, IRNN, TNNR, SVRG-ARM.}
\label{figure8}
\end{figure}
\section*{Acknowledgments}

This work is supported in part by  GuangDong Basic and Applied Basic Research Foundation under grant 2021A1515110530, the Foundation for Distinguished Young Talents of Guangdong under grant 2021KQNCX075, National Natural Science Foundation of China under grants U21A20455, 61972265, 11871348 and 61373087, the Natural Science Foundation of Guangdong Province of China under grant 2020B1515310008, the Educational Commission of Guangdong Province of China undergrant 2019KZDZX1007, and the Guangdong Key Laboratory of Intelligent Information Processing, China. M. Ng's research is supported 
in part by the HKRGC GRF 12300218, 12300519, 17201020, 17300021, C1013-21GF, C7004-21GF, and Joint
NSFC-RGC N-HKU76921.

\end{document}